\documentclass[11pt]{amsart}

\usepackage[utf8]{inputenc}
\usepackage[margin=2.5cm]{geometry}
\usepackage{amsmath,amssymb,amsfonts,amscd}
\usepackage{mathrsfs}
\usepackage{mathtools}
\usepackage{tikz}
\usepackage{xcolor}
\usepackage{verbatim}
\usepackage{url}
\usepackage{parskip}
\usepackage{hyperref}

\allowdisplaybreaks
\mathtoolsset{showonlyrefs}

\makeatletter
\@namedef{subjclassname@2020}{\textup{2020} Mathematics Subject Classification}
\makeatother

\newcommand{\C}{\mathscr{C}}
\newcommand{\R}{\mathbb{R}}
\newcommand{\oo}[1]{\left({#1}\right)}
\newcommand{\cc}[1]{\left[{#1}\right]}
\newcommand{\inner}[1]{\left\langle{#1}\right\rangle}

\newcommand{\norm}[1]{\left|{#1}\right|}
\newcommand{\llll}[1]{\left\|{#1}\right\|}

\newcommand{\intM}[1]{\int_{\Sigma}{{#1}\,d\mu}}
\newcommand{\intgamma}[1]{\int_{\left[\gamma>0\right]}{{#1}\,d\mu}}

\newcommand{\nablanorm}[3]{\left|\nabla_{\oo{#1}}{#2}\right|^{#3}}

\newcommand{\cgam}{c_{\gamma}}
\newcommand{\ctild}{c_{\tilde{\gamma}}}
\newcommand{\cmss}{c_{\textrm{M\hskip-0.1mm S\hskip-0.1mm S}}}
\newcommand{\Ao}[1]{\left\|A^{o}\right\|_{2,\cc{\gamma>0}}^{#1}}
\newcommand{\A}[1]{\left\|A\right\|_{2,\cc{\gamma>0}}^{#1}}

\newcommand{\Aoc}[1]{\left\|A^{o}\right\|_{2}^{{#1}}}
\newcommand{\Aoi}[1]{\left\|A^{o}\right\|_{\infty}^{{#1}}}

\newcommand{\Kappa}{\mathcal K}

\newcommand{\ben}{\begin{eqnarray*}}
\newcommand{\een}{\end{eqnarray*}}

\newtheorem{theorem}{Theorem}
\newtheorem{corollary}[theorem]{Corollary}
\newtheorem{remark}[theorem]{Remark}

\newtheorem{proposition}[theorem]{Proposition}
\newtheorem{lemma}[theorem]{Lemma}

\makeatletter
\let\oldhypertarget\hypertarget
\renewcommand{\hypertarget}[2]{%
  \oldhypertarget{#1}{#2}%
    \protected@write\@mainaux{}{%
        \string\expandafter\string\gdef
          \string\csname\string\detokenize{#1}\string\endcsname{#2}%
    }%
  }
\newcommand{\myhyperlink}[1]{%
  \hyperlink{#1}{\csname #1\endcsname}%
  }
\makeatother

\begin{document}

\date{\today}

\title{Concentration-compactness for the geometric polyharmonic heat flow}

\author[J. McCoy]{James McCoy}
\address{Department of Mathematical and Geospatial Sciences, School of Science, Royal Melbourne Institute of Technology,\\
124 La Trobe Street, Melbourne VIC $3000$\\
Australia}
\email{James.McCoy@rmit.edu.au}

\author[S. Parkins]{Scott Parkins}
\address{Institute for Mathematics and its Applications, School of Mathematics and Applied Statistics\\
University of Wollongong\\
Northfields Ave, Wollongong, NSW $2500$\\
Australia}
\email{srp854@uowmail.edu.au}

\author[G. Wheeler]{Glen Wheeler}
\address{Institute for Mathematics and its Applications, School of Mathematics and Applied Statistics\\
University of Wollongong\\
Northfields Ave, Wollongong, NSW $2500$\\
Australia}
\email{glenw@uow.edu.au}

\author[V. Wheeler]{Valentina Wheeler}
\address{Institute for Mathematics and its Applications, School of Mathematics and Applied Statistics\\
University of Wollongong\\
Northfields Ave, Wollongong, NSW $2500$\\
Australia}
\email{vwheeler@uow.edu.au}

\thanks{
The research of the second author was supported by an Australian Postgraduate Award.
The research of the first and fourth authors was supported by Discovery Project DP180100431 of the Australian Research Council.  The research of the first author was also supported by DP250103952 of the Australian Reseqrach Council and a President's International Fellowship Initiative Visiting Fellowship of the Chinese Academy of Sciences, grant 2024PVA0042.
The research of the fourth author was also supported by grant DE190100379 of the Australian Research Council.
Some results from this paper appear in part of the Ph.D. thesis of the second author.
The authors would like to thank Graham Williams for useful discussions, and Ben Andrews for his interest in this work.
}


\begin{abstract}
We develop a concentration-compactness theory for geometric evolution equations of arbitrarily high order, using the geometric polyharmonic heat flow of closed immersed surfaces in \(\R^3\) as the model case.
The flow is the \((2p+2)\)-order normal evolution
\[
    \partial_t f=(-1)^{p+1}\Delta^p H\,\nu,\qquad p\geq1,
\]
which includes the surface diffusion flow when \(p=1\).
We prove localised energy estimates with sharp cut-off bookkeeping, interior estimates, a lifespan/concentration alternative, tracefree-curvature \(\varepsilon\)-regularity estimates, and a gap theorem for stationary solutions.
These tools are then combined with a blowup argument, the preservation of signed enclosed volume, and the monotonicity of area to rule out singularities below a small tracefree-curvature threshold.
Consequently, for connected initial immersions satisfying \(\|A^o\|_2^2<\varepsilon\), where \(\varepsilon>0\) depends only on the order of the flow, the solution exists for all time and converges exponentially in \(C^\infty\) to a round sphere with the preserved enclosed volume.
\end{abstract}

\maketitle

\tableofcontents


\section{Introduction}

The concentration-compactness method, developed by Lions and others in the calculus of variations \cite{lions1,lions2}, has become one of the central tools for analysing loss of compactness in nonlinear problems.
Its role in geometric evolution equations began with Struwe's work on the harmonic map heat flow of surfaces \cite{Struwe1}.
Ecker used local integral estimates to obtain concentration-compactness and regularity results for the mean curvature flow of surfaces \cite{ecker95}, and Kuwert and Sch\"atzle adapted this circle of ideas to the Willmore flow \cite{Kuwert2}.
Their work showed that localised curvature estimates can replace pointwise maximum-principle arguments in genuinely higher-order geometric flows.

The purpose of this article is to extend this strategy to curvature flows of arbitrary even order.
The main analytical difficulty is that repeated covariant differentiation of local cut-off functions produces many curvature terms and, if the powers of the cut-off are not tracked precisely, one loses exactly the weights required to absorb the error terms into the diffusion.
A substantial part of the paper is therefore devoted to localised Sobolev and interpolation inequalities, and to weighted curvature estimates in which the cut-off powers are kept sharp.

We illustrate the method for the geometric polyharmonic heat flow.
Let \(f:\Sigma\times[0,T)\rightarrow\R^3\), \(T>0\), be a one-parameter family of compact immersed surfaces
\(f(\cdot,t)=f_t:\Sigma\rightarrow f_t(\Sigma)=\Sigma_t\subset\R^3\).
For an integer \(p\geq1\), we say that \(f\) evolves by the geometric polyharmonic heat flow of order \(2p+2\) if
\begin{equation}
    \partial_tf=(-1)^{p+1}(\Delta^pH)\,\nu,
    \label{PolyharmonicIntro1}\tag{GPHF}
\end{equation}
with smooth initial immersion \(f(\cdot,0)=f_0\).
Here \(\Delta\), \(H\), and \(\nu\) denote respectively the Laplace-Beltrami operator, the scalar mean curvature, and the outer unit normal of \(\Sigma_t\).
The case \(p=1\) is the surface diffusion flow.

\begin{remark}
The mean curvature vector is \(\vec H=-H\nu\).
Using the induced normal-bundle Laplacian \(\Delta^\perp\), the leading part of \eqref{PolyharmonicIntro1} may be written schematically as
\[
    \partial_tf=(-1)^p(\Delta^\perp)^p\Delta f,
\]
where \(\Delta f\) is the coordinatewise Laplace-Beltrami operator applied to the immersion.
The flow is therefore a quasilinear geometric parabolic system of order \(2p+2\), modulo the usual degeneracy under reparametrisation.
\end{remark}

The flows \eqref{PolyharmonicIntro1} are closely related to higher-order analogues of the Willmore flow.
They share the same leading term as the \(L^2(d\mu)\)-gradient flow of
\begin{equation}
    E_p[f]:=\intM{\left|\Delta^{\frac{p-1}{2}}H\right|^2},
    \qquad p\geq1.
    \label{penergy}
\end{equation}
Throughout the article we use the convention
\(\Delta^{(2k+1)/2}H=\nabla\Delta^kH\).
The energy \(E_1\) is a multiple of the Willmore energy; the case \(E_2\) is related to surfaces of minimal mean curvature variation \cite{xu2007g2}.
For curves, corresponding higher-order gradient flows have been studied in \cite{ideal,blatt,mccoy2020sixth, MWW2}.

Local existence for \eqref{PolyharmonicIntro1} follows from the standard theory for quasilinear geometric parabolic systems after fixing the diffeomorphism invariance; see for example \cite{baker2011mean,Mantegazza1,Mantegazza2}.
We shall use the following form.

\begin{theorem}[Local well-posedness of the flows]
\label{PMste}
Let \(\Sigma\) be a compact surface without boundary and let \(f_0:\Sigma\rightarrow\R^3\) be a smooth immersion.
There exists a maximal \(T>0\) and a unique smooth solution
\(f:\Sigma\times[0,T)\rightarrow\R^3\) of \eqref{PolyharmonicIntro1} with
    $f(\cdot,t)\rightarrow f_0(\cdot)$ as $t\searrow0$
in the smooth topology.
\end{theorem}

Our analysis is localised by cut-off functions of the form
\(\gamma=\tilde\gamma\circ f:\Sigma\rightarrow[0,1]\), where
\(\tilde\gamma:\R^3\rightarrow[0,1]\) is smooth and satisfies
\begin{equation}
    0\leq\tilde\gamma\leq1,
    \qquad
    |D^k\tilde\gamma|\leq c_k\ctild^k
    \quad\text{for every }k\geq1.
    \label{GammaProp0}
\end{equation}
Typically \(\tilde\gamma\equiv1\) on \(B_\rho(x)\), \(\tilde\gamma\equiv0\) outside \(B_{2\rho}(x)\), and \(\ctild\simeq\rho^{-1}\).

The first main local estimate is the following weighted curvature inequality.
It is the estimate that drives both the lifespan theorem and the interior estimates.  First we define the relevant concentration function.  For \(x\in\R^3\), \(t\in[0,T)\), and \(\rho>0\), set
\[
    \C(x,t,\rho)
    :=
    \int_{f^{-1}(B_\rho(x))}|A|^2\,d\mu,
    \qquad
    \C(t,\rho):=\sup_{x\in\R^3}\C(x,t,\rho).
\]

\begin{theorem}[Local curvature estimate]\label{T:firstmainest}
Let \(f:\Sigma\times[0,T)\rightarrow\R^3\) satisfy \eqref{PolyharmonicIntro1}.
Let \(k\in\mathbb N_0\), set \(m=k+p+1\), and let \(s\geq2m\).
There exist \(\varepsilon_{\mathrm{base}}>0\), depending only on \(p\), and \(C=C(k,p,s)\) such that, if
\[
    \int_{[\gamma>0]}|A|^2\,d\mu\leq\varepsilon_{\mathrm{base}},
\]
then for every \(\delta\in(0,1]\),
\[
\frac{d}{dt}\intM{|\nabla_{(k)}A|^2\gamma^s}
+
\left(2-\delta-C\|A\|^2_{2,[\gamma>0]}\right)
\intM{|\nabla_{(m)}A|^2\gamma^s}
\leq
C_\delta\cgam^{2m}\|A\|^2_{2,[\gamma>0]}.
\]
Consequently, for each fixed \(k\) and each fixed admissible choice of \(s=s(k,p)\), there is a threshold
\(\varepsilon_{\mathrm{loc}}(k,p)>0\) such that, if \(\tilde\gamma\) is a standard cut-off between \(B_\rho(x)\) and \(B_{2\rho}(x)\), and if
\[
    \sup_{0\leq\tau\leq t}\C(x,\tau,2\rho)
    \leq \varepsilon_{\mathrm{loc}}(k,p),
\]
then
\[
\begin{split}
\int_{f^{-1}(B_\rho(x))}|\nabla_{(k)}A|^2\,d\mu
&+
\int_0^t\int_{f^{-1}(B_\rho(x))}|\nabla_{(m)}A|^2\,d\mu\,d\tau
\\
&\leq
\left.\int_{f^{-1}(B_{2\rho}(x))}|\nabla_{(k)}A|^2\,d\mu\right|_{\tau=0}
+
C\rho^{-2m}\int_0^t\C(x,\tau,2\rho)\,d\tau .
\end{split}
\]
The same threshold may be chosen simultaneously for any finite set of derivative orders by taking the minimum of the corresponding fixed-order thresholds.
\end{theorem}

Theorem \ref{T:firstmainest} says that, as long as \(\C(x,t,2\rho)\) remains below the local threshold, all higher curvature derivatives are controlled on \(B_\rho(x)\) by the initial local energy and a scale-correct error term.
For \(k=0\) this already gives linear-in-time control of local curvature concentration.
A standard Struwe-type argument then yields the following lifespan alternative.

\begin{theorem}[Lifespan theorem]\label{LifespanTheorem}
There exist \(\varepsilon_{\mathrm{life}}>0\) and \(C_{\mathrm{life}}<\infty\), depending only on \(p\), such that the following holds.
Let \(f:\Sigma\times[0,T)\rightarrow\R^3\) be the maximal smooth solution of \eqref{PolyharmonicIntro1} with smooth compact initial data.
If \(\rho>0\), \(0<\varepsilon\leq\varepsilon_{\mathrm{life}}\), and
\begin{equation}
    \C(0,\rho)\leq\varepsilon,
    \label{Lifespansmallness}
\end{equation}
then
\begin{equation}
    T\geq C_{\mathrm{life}}^{-1}\rho^{2(p+1)},
    \label{lifespan1}
\end{equation}
and
\begin{equation}
    \C(t,\rho)\leq C_{\mathrm{life}}\varepsilon
    \qquad\text{for }0\leq t\leq C_{\mathrm{life}}^{-1}\rho^{2(p+1)}.
    \label{lifespan2}
\end{equation}
\end{theorem}

For every \(\varepsilon\in(0,\varepsilon_{\mathrm{life}})\), compactness of the initial surface guarantees that \eqref{Lifespansmallness} holds for sufficiently small \(\rho\).
Define the maximal nonconcentrating scale
\begin{equation}
    \rho^\varepsilon(t)
    :=
    \sup\{r>0:\C(t,r)\leq\varepsilon\}.
    \label{E:rhoepsdef}
\end{equation}
If \(T<\infty\), then Theorem \ref{LifespanTheorem} implies
$\rho^\varepsilon(t)
    \leq
    C(T-t)^{\frac1{2(p+1)}}
    \longrightarrow0$
 as $t\nearrow T$.
Equivalently, at a finite singular time at least \(\varepsilon\) units of \(L^2\)-curvature concentrate at arbitrarily small spatial scales.
This is the concentration-compactness alternative used later in the blowup analysis.

\subsection{Results for the flow near a sphere}

Round spheres, of arbitrary centre and radius, are equilibria of \eqref{PolyharmonicIntro1}; more generally, every constant-mean-curvature immersion is stationary.
To obtain convergence specifically to a round sphere, it is therefore natural to impose smallness of the scale-invariant tracefree curvature,
\begin{equation}
    \int_\Sigma |A^o|^2\,d\mu\leq\varepsilon.
\end{equation}
This condition singles out the umbilic equilibria: on a connected surface, \(A^o\equiv0\) forces the immersion to be either a round sphere or a plane.

A key input is the high-order Sobolev Interpolation Theorem proved later as Theorem \ref{HigherOrderSobolevTheorem1}.
For
\[
    S_l:=\intM{\left|\Delta^{l/2}H\right|^2\gamma^{2l}},
\]
it gives, under small local \(\|A^o\|_2\), estimates of the form
\[
    S_l
    \leq
    C\left(
        \Ao{\frac{2m}{l+m}}S_{l+m}^{\frac{l}{l+m}}
        +
        \cgam^{2l}\Ao{2}
    \right).
\]
Together with the local estimates above, this produces the following tracefree-curvature \(\varepsilon\)-regularity statement.
Here
\[
    \C^o(x,\rho):=\int_{f^{-1}(B_\rho(x))}|A^o|^2\,d\mu.
\]

\begin{theorem}[Local \(\varepsilon\)-regularity]\label{EpsilonRegularity}
There exists a universal \(\varepsilon_{\mathrm{reg}}>0\) such that the following holds.
Let \(f:\Sigma\rightarrow\R^3\) be a smooth immersion and suppose
\begin{equation}
    \C^o(x,2\rho)\leq\varepsilon_{\mathrm{reg}}.
    \label{EpsilonRegularity1}
\end{equation}
Then, for every \(m\geq3\), there is a constant \(C=C(m)\) such that
\[
\|A^o\|_{\infty,f^{-1}(B_\rho(x))}^{2m}
\leq
C\|A^o\|_{2,f^{-1}(B_{2\rho}(x))}^{2(m-1)}
\left(
    \left\|\Delta^{m/2}H\right\|_{2,f^{-1}(B_{2\rho}(x))}^{2}
    +
    \rho^{-2m}\|A^o\|_{2,f^{-1}(B_{2\rho}(x))}^{2}
\right).
\]
In particular, if \(\Delta^pH\) is constant for some \(p\in\mathbb N\), then
\[
    \|A^o\|_{\infty,f^{-1}(B_\rho(x))}
    \leq
    C(p)\rho^{-1}\|A^o\|_{2,f^{-1}(B_{2\rho}(x))}
    \leq
    C(p)\rho^{-1}\varepsilon_{\mathrm{reg}}^{1/2}.
\]
\end{theorem}

The corresponding Gap Theorem classifies stationary solutions with sufficiently small tracefree curvature.

\begin{theorem}[Gap theorem]\label{GapTheorem}
There exists \(\varepsilon_{\mathrm{gap}}>0\) such that the following holds.
Let \(f:\Sigma\rightarrow\R^3\) be a proper smooth immersion with \(\Sigma\) connected and
\(\Delta^pH\equiv C\) for some constant \(C\).
If
\begin{equation}
    \intM{|A^o|^2}\leq\varepsilon_{\mathrm{gap}},
    \label{Smallness}
\end{equation}
then either \(f(\Sigma)\) is a standard round sphere or \(f(\Sigma)\) is a standard flat plane.
\end{theorem}

\begin{remark}
In the gap theorem the constant on the right-hand side of \(\Delta^pH\equiv C\) need not vanish.
Indeed, applying the preceding \(\varepsilon\)-regularity estimate with \(m=2p+1\) uses \(\nabla\Delta^pH=0\), which is enough to conclude that \(A^o\equiv0\) under the smallness hypothesis.
\end{remark}

We now describe how the preceding estimates enter the global analysis near spheres.
Assume, for contradiction, that a solution with initially small tracefree curvature develops a finite singular time \(T\).
Choose \(\varepsilon>0\) below all thresholds above and let \(\rho_j=\rho^\varepsilon(t_j)\rightarrow0\) for a sequence \(t_j\nearrow T\).
Choosing centres \(x_j\) where the concentration is essentially attained, consider the rescaled flows
\[
    f_j(q,\tau)
    =
    \rho_j^{-1}\Big(f(q,t_j+\rho_j^{2(p+1)}\tau)-x_j\Big).
\]
The lifespan theorem and interior estimates give uniform curvature bounds for \(f_j\) on compact subsets of space-time, provided \(\varepsilon\) is chosen sufficiently small.
A compactness theorem then produces a proper blowup limit \(\hat f\).
The finite dissipation supplied by the energy estimates forces this limit to be stationary, while the small tracefree-curvature hypothesis allows the gap theorem to be applied.
The gap theorem leaves only two possible models: a plane or a round sphere.  The plane is incompatible with the normalisation of the blowup by a positive amount of curvature concentration.  A compact round-sphere model is also impossible: after undoing the blowup scaling it would force the area of the original closed surface to tend to zero, while the signed enclosed volume is preserved and the isoperimetric inequality gives a positive lower bound for the area.

Once global existence is known, the preserved almost-sphericity estimates give exponential decay of the umbilic energy \(\|A^o\|_2^2\).
The interpolation and interior estimates then bootstrap this to exponential decay of all curvature derivatives, yielding smooth convergence to a round sphere.
The main global consequence is the following.

\begin{theorem}[Global existence and exponential convergence to round spheres]\label{MainTheorem}
There exists \(\varepsilon_1=\varepsilon_1(p)>0\) with the following property.
Let \(f:\Sigma\times[0,T)\rightarrow\R^3\) be the maximal smooth solution of \eqref{PolyharmonicIntro1} with connected compact initial immersion \(f_0\).
Assume that the orientation is chosen so that the preserved signed enclosed volume \(V_0\) is positive.
If
\begin{equation}
    \left.\intM{|A^o|^2}\right|_{t=0}
    \leq\varepsilon_1,
    \label{MainTheoremSmallness}
\end{equation}
then the flow exists for all time and, after reparametrisation, \(f_t\) converges exponentially in the \(C^\infty\) topology to a parametrisation of the round sphere
    $\mathbb S_{\left(3V_0/4\pi\right)^{1/3}}(x_0)$
for some \(x_0\in\R^3\).
\end{theorem}

\section{Preliminaries} \label{S:prelim}

Our primary object of study is a surface $\Sigma$ immersed into $\R^3$ via a smooth immersion $f_0:\Sigma\rightarrow\mathbb{R}^{3}$ or, via a family of smooth immersions $f:\Sigma\times[0,T)\rightarrow\R^3$.
The induced metric $g$ on $\Sigma$ is given by pulling back the standard Euclidean inner product on $\mathbb{R}^{3}$ along $f$. It is defined pointwise by
\begin{equation*}
g_{ij}=(\partial_{i}f,\partial_{j}f).\label{Prelimiaries1}
\end{equation*}
Here $\partial$ denotes the coordinate derivatives on $\Sigma$ and $(\cdot,\cdot)$ the standard Euclidean inner product. This induces an inner product on all tensors along $f$ of similar type, which are defined via traces over pairs of indices. For example, the inner product on the $\oo{1,2}$-tensors $S,T$ is defined by
\[
\inner{S,T}=g^{jm}g^{kn}g_{il}S_{jk}^{i}T_{mn}^{l}.
\]
Here and throughout the paper we adopt the Einstein convention: repeated indices are summed from $1$ to $2$.
The norm squared of a tensor $T$ is the inner product of $T$ with itself and is denoted $\norm{T}^{2}$.
For tensors $S$ and $T$ we frequently use the notation of Hamilton \cite{Hamilton1}, using $S*T$ to denote a linear combination of tensors formed by contracting pairs of indices of $S$ and $T$ by $g$.
We also use $T^{\#2}$ for $T*T$.
A useful property of $*-$notation is that
\[
|S*T|\leq c\norm{S}\norm{T}\qquad
\text{and}\qquad
|T^{\#2}| \le c|T|^2
\]
for some constant $c$.
We will use this without comment in our estimates.
Throughout this article when we write $c$ without any decoration we denote bounded, absolute constants that may vary from term to term and line to line.
For a tensor $T$, we define
\[
P_{j}^{i}\oo{T}=\sum_{r_{1}+\dots+r_{j}=i} \nabla_{\oo{r_{1} }}T*\cdots*\nabla_{\oo{r_{j} }}T\,.
\]
That is, $P^i_j(T)$ is a lineaer combination of tensors, each term of which contains $j$ factors of $T$ with total order of covariant derivative equal to $i$.
Here $\nabla_{\oo{r_{k} }}T$ denotes the $\left( r_k\right)$-th iterated covariant derivative of $T$.  
We also use the convention that
\[
    P_0^0(T)=1,
    \qquad
    P_0^i(T)=0 \quad\text{for } i>0 .
\]

The $P_j^i$ notation is particularly useful when only the number of factors of $T$ and the total number of derivatives is important.  
In some cases, we may need to estimate a large collection of terms with additional important structure that is not captured by the $P^i_j$ notation.
For our analysis to treat these cases it suffices to extend this notation slightly.
We define
\[
P_{j}^{i,k}\oo{T}=\sum_{\substack{r_{1}+\dots+r_{j}=i\\\text{each}\ r_l\le k}}\nabla_{\oo{r_{1} }}T*\cdots*\nabla_{\oo{r_{j} }}T.
\]
For example, $|\nabla_{(2)}A|^2$ can be represented by $\nabla_{(2)}A*\nabla_{(2)}A$, $(\nabla_{(2)}A)^{\#2}$, $P^4_2(A)$ and $P^{4,2}_2(A)$.

The second fundamental form $A$ has components
\[
A_{ij}=-(\partial_{ij}f,\nu) \mbox{,}
\]
and the Weingarten map is $A_i^j = g^{ik}A_{jk}$.

From these we have the mean curvature
\[
H=g^{ij}A_{ij}=A_{i}^{i} \mbox{,}
\]
that is, the ordinary trace of the Weingarten map, or equivalently the trace with respect to metric $g$ of the second fundamental form.  The Gauss curvature $K$ is the determinant of the Weingarten map:
\[
K = \text{det }A_i^j\,.
\]
The tracefree second fundamental form is the symmetric
tracefree part of $A$, denoted $A^{o}$, with components 
\[
A_{ij}^{o}=A_{ij}-\frac{1}{2}H\, g_{ij}
\,.
\]
The norm of the tracefree second fundamental form is then
\begin{equation} \label{E:nAo}
  \left| A^o \right|^2 = \left| A \right|^2 - \frac{1}{2} H^2 \mbox{.}
\end{equation}
A short calculation shows that
\begin{equation}
\oo{\nabla^{*}A^{o}}_{j}:=\nabla^{i}A_{ij}^{o}=g^{ip}\nabla_{p}\oo{A_{ij}-\frac{1}{2}H\, g_{ij}}=\frac{1}{2}\nabla_{j}H;
\label{EQabove10}
\end{equation}
note that $\nabla^{*}$ is the geometric divergence operator or formal adjoint of $\nabla$ in $L^2(d\mu)$.
Here we use $d\mu = \sqrt{\text{det }g}\,d\mathcal{L}^2$ to denote the measure on $(\Sigma,g)$, where $d\mathcal{L}^2$ is standard Lebesgue measure.

In \eqref{EQabove10} we have used the total symmetry of the $\oo{0,3}$-tensor $\nabla A$, known as the Codazzi equations:
\[
\nabla_{i}A_{jk}=\nabla_{j}A_{ki}=\nabla_{k}A_{ij}.
\]
It follows that
\begin{equation*}
\norm{\nabla H}^{2}= 4\norm{\nabla^{*}A^{o}}^{2}\leq4\norm{\nabla A^{o}}^{2}.
\end{equation*}
Moreover, for any $k\in\mathbb{N}$ we have
\begin{equation}
\nablanorm{k}{H}{2}\leq4\nablanorm{k}{A^{o}}{2}\label{Prelimiaries2}
\end{equation}
and hence
\begin{equation}
\nablanorm{k}{A}{2}=\nablanorm{k}{A^{o}}{2}+\frac{1}{2}\nablanorm{k}{H}{2}\leq3\nablanorm{k}{A^{o}}{2}.\nonumber
\end{equation}
This implies the remarkable fact that only control on $H$, $A$, and the derivatives of the \emph{tracefree} second fundamental form are needed to control the regularity of $f$.
The remaining derivatives of $A$ are not needed.

The Laplace-Beltrami operator acts on the components of an $\oo{m,n}$-tensor $S$ via
\[
\Delta S_{j_{1}\dots j_{n}}^{i_{1}\dots i_{m}}=g^{pq}\nabla_{pq}S_{j_{1}\dots j_{n}}^{i_{1}\dots i_{m}}.
\]
We define the integral of a compactly supported function
$h:\Sigma\rightarrow\mathbb{R}$ as 
\[
\intM{h\,},
\]
where $d\mu$ is as above.
We denote the area  of $\Sigma$ by
\[
\mu(\Sigma) = |\Sigma| = \int_\Sigma\,d\mu\,.
\]

\subsection{On the localisation function.}
Our estimates are primarily local, in that we aim to control a quantity in a small (ambient) ball in terms of data in a slightly larger (still ambient) ball.
To achieve this, we use a cutoff function $\gamma=\tilde{\gamma}\circ f:\Sigma\rightarrow\cc{0,1}$, satisfying \eqref{GammaProp0}.
The existence of such a function is standard, for instance see \cite{willmore1997riemannian} for an explicit construction.

We now seek to show that the $k$-th iterated covariant derivative of $\gamma$, $\nabla_{(k)}\gamma$, can be controlled by the curvature and its derivatives.
So we have
\[
\nabla\gamma = D\tilde\gamma|_f*Df
\qquad
\text{and}
\qquad
\nabla_{(2)}\gamma = D^2\tilde\gamma|_f*Df + D\tilde\gamma*\nabla Df\,.
\]
In the above equation we have abused the $*$ notation to mean application of multilinear forms.
This will also occur below.
The standard Fa\'a di Bruno formula then gives the general expression:
\begin{equation}
\label{E:faadibruno}
\nabla_{i_1\cdots i_n}\gamma
= \sum_{\pi\in P(i_1,\ldots,i_n)}
    D^{|\pi|}\tilde\gamma|_f
    *
    \prod_{B\in\pi}
        \nabla_Bf
\,.
\end{equation}
Above we use $P(i_1,\ldots,i_n)$ to denote the set of partitions of $i_1,\ldots,i_n$.
If the indices are left un-named, we use the notation $P(n)$ to denote the set of partitions of a sequence of $n$ numbers, each of which are either 1 or 2.

Each partition $B = (b_1,\ldots,b_k)$ where $k$ is a divisor of $n$, and $\nabla_B = \nabla_{b_1}\cdots\nabla_{b_{k-1}}D_{b_k}$.
Note that each of these indices is either 1 or 2.

In a chart $(\nabla Df)(\partial_i,\partial_j) = -A_{ij}\nu$ is another expression for the (normal valued) second fundamental form (see \cite[Chapter 6]{DoCarmo1}).
Further derivatives of this expression must be calculated in order for us to properly estimate $\nabla_{(k)}\gamma$.
This becomes quite tedious to do in general, and so for our purposes here we focus only on the structure of $\nabla_{(k)}\gamma$.

\begin{lemma}
\label{L:highcovf}
We have for all $k\ge1$
\[
    \nabla_{(k)}Df
    = \sum_{i=0}^{\lfloor {k/2} \rfloor} 
        P_{2i+1}^{k-1-2i}(A)\nu
    + \sum_{i=1}^{\lfloor {k/2} \rfloor} 
        P_{2i}^{k-2i}(A)*Df
\]
and for $k>1$
\[
    \nabla_{(k)}Df
= \nabla_{(k-1)}A\,\nu
    + \sum_{i=1}^{\lfloor {k/2} \rfloor} 
        \bigg(
        P_{2i+1}^{k-1-2i}(A)\nu
        + P_{2i}^{k-2i}(A)*Df
        \bigg)
        \,.
\]
\end{lemma}
The proof of Lemma \ref{L:highcovf} follows readily by induction, using the relations above.
Combining Lemma \ref{L:highcovf} with the Fa\'a di Bruno formula \eqref{E:faadibruno} yields the following.
\begin{lemma}\label{L:highcovgamma}
For any $n\in\mathbb{N}$ each term in $\nabla_{(n)}\gamma$ is a
contraction of the form
\[
    D^q\tilde\gamma|_f * \mathcal F_{r_1}*\cdots *\mathcal F_{r_q},
    \qquad
    q\geq1,
    \qquad
    r_1+\cdots+r_q=n,
\]
where $\mathcal F_1=Df,$
and, for $r\geq2$,
\[
\mathcal F_r
=
\sum_{i=0}^{\lfloor (r-2)/2\rfloor}
        P_{2i+1}^{r-2-2i}(A)\nu
+
\sum_{i=1}^{\lfloor (r-1)/2\rfloor}
        P_{2i}^{r-1-2i}(A)*Df .
\]
In particular, no negative-order $P$-term is present: the case $r=1$ is
carried only by the tangential factor $Df$.
\end{lemma}
Lemma \ref{L:highcovgamma} implies that there exists a constant $c=c(k)$ and a universal constant $\cgam>0$ such that
\begin{equation}
\norm{\nabla_{\oo{k}}\gamma}
\leq c\bigg(\cgam\sum_{i=1}^{k-1}\norm{P_{i}^{k-1-i}\oo{A}}
    + \cgam^2\sum_{i=1}^{k-2}\norm{P_{i}^{k-2-i}\oo{A}}
    + \dots + \cgam^{k}
    \bigg)\,.
    \label{GammaProp}
\end{equation}
and, as a consequence,
\begin{equation}
|\nabla_{(k)}\gamma^{s}| \leq c \bigg(\cgam^{k}\,\gamma^{s-k} + \sum_{l=1}^{k-1}\sum_{a=1}^{k-l}\sum_{b=1}^{a}\cgam^{k-a}|P_{b}^{a-b}(A)|\gamma^{s-l}\bigg).
\label{E:highcovgammapow}
\end{equation}
for every $k\geq 1$. 
Note that, in contrast to \cite{Kuwert2}, we include an extra factor of $c_\gamma$ in the first term on the right hand side to preserve scaling ($c_\gamma$ scales like $\left| A \right|$).
We conclude this section by providing some estimates on the cutoff function and its time derivative. Under \eqref{PolyharmonicIntro1}, by the definition of $\gamma$ it follows that
\begin{equation}
\partial_{t}\gamma=\partial_{t}\oo{\tilde{\gamma}\circ f}
=\oo{D\tilde{\gamma},\partial_{t}f}
=\oo{-1}^{p+1}\oo{D\tilde{\gamma}|_f,\Delta^{p}H\cdot\nu}
=\oo{-1}^{p+1}\,D_{\nu}\tilde{\gamma}|_f\oo{\Delta^{p}H}\,,
\label{EvolutionOfGammaEqn}
\end{equation} 
where $D_{\nu}\tilde{\gamma}|_f = \oo{D\tilde{\gamma}|_f,\nu}$ denotes the directional derivative of $\tilde\gamma$ in the direction $\nu$.
We may use the strategy above to also estimate this term.
In particular, Lemma \ref{L:highcovf} applies with $\nu$ and $Df$ swapped.
The consequent estimate is as follows:
\begin{equation}
|\nabla_{(k-1)}D_\nu\tilde\gamma|_f|
\leq c\bigg(\cgam\sum_{i=1}^{k-1}\norm{P_{i}^{k-1-i}\oo{A}}
    + \cgam^2\sum_{i=1}^{k-2}\norm{P_{i}^{k-2-i}\oo{A}}
    + \cgam^3\sum_{i=1}^{k-3}\norm{P_{i}^{k-3-i}\oo{A}}
    + \dots + \cgam^{k}
    \bigg)\,.
\label{E:gammanormderhighcovest}
\end{equation}
The above with $k = 2p+1$ can be used to estimate $\partial_t\gamma$.
For later convenience we record here the following estimate involving $\partial_t\gamma$.
Note that in its proof the local product interpolation estimate and Lemma \ref{NewLemma1} are used to simplify the explosion of terms arising from \eqref{E:highcovgammapow} and \eqref{E:gammanormderhighcovest}.
The reader may wish to skip its proof until coming to the proof of Proposition \ref{EvolutionProposition1}, where it is used and similar techniques are applied.

We also use in this proof and throughout the paper some additional notation for localised $L^p(d\mu)$ norms.
For $p, s\in\mathbb{N}$ and for a tensor $T$, define
\begin{equation}
||T||_{p,\gamma^{s}}^{p}:=\intM{|T|^{p}\gamma^{s}}.\label{TensorTNotation}
\end{equation}

\begin{lemma}[Local product interpolation]\label{L:local-product-interpolation}
Let $m\in\mathbb N$, let $s\geq 2m$, and set
\[
    M_m:=\llll{\nabla_{(m)}A}_{2,\gamma^s}^{2},
    \qquad
    E:=\A{2}.
\]
Assume $E\leq1$.  If
$\theta\in\{1,\ldots,m\}$, $r\geq2$, and no factor in the schematic
term $P_r^{2m-r+2-\theta}(A)$ contains more than $m$ derivatives of
$A$, then, for every $\eta\in(0,1]$,
\begin{equation}
\cgam^{\theta}
\intM{\left|P_r^{2m-r+2-\theta}(A)\right|\gamma^{s-\theta}}
\leq
\eta M_m+C_\eta \cgam^{2m}E .
\label{E:localinterp-positive-defect}
\end{equation}
If $\theta=0$ and $r\geq4$, then
\begin{equation}
\intM{\left|P_r^{2m-r+2}(A)\right|\gamma^s}
\leq
C E M_m+C\cgam^{2m}E .
\label{E:localinterp-critical}
\end{equation}
Here $C=C(m,r,s)$ and $C_\eta=C_\eta(m,r,s,\eta)$.
\end{lemma}

\begin{proof}
It is enough to consider a single monomial
\[
    P_r^{2m-r+2-\theta}(A)
    =
    \nabla_{(i_1)}A*\cdots *\nabla_{(i_r)}A,
    \qquad
    \sum_{\alpha=1}^r i_\alpha=2m-r+2-\theta .
\]
The local Gagliardo--Nirenberg inequality obtained from Lemma
\ref{EvolutionLemma3}, followed by Lemma \ref{NewLemma1}, gives
\[
\intM{\left|P_r^{2m-r+2-\theta}(A)\right|\gamma^{s-\theta}}
\leq
C E^{\frac{r-2}{2}+\frac{\theta}{2m}}
M_m^{1-\frac{\theta}{2m}}
+
C\cgam^{2m-\theta}E .
\]
Multiplying by $\cgam^\theta$ and applying Young's inequality proves
\eqref{E:localinterp-positive-defect}; the smallness assumption
$E\leq1$ ensures that every remaining power of $E$ is bounded by $E$.
The same estimate with $\theta=0$ gives
\[
\intM{\left|P_r^{2m-r+2}(A)\right|\gamma^s}
\leq
C E^{\frac{r-2}{2}}M_m+C\cgam^{2m}E,
\]
and this implies \eqref{E:localinterp-critical} when $r\ge4$.
Finally, if a term carries the weight $\gamma^{s-\ell}$ with
$\ell\leq\theta$, then it is controlled by the same right-hand side,
because $0\leq\gamma\leq1$.
\end{proof}

\begin{lemma}\label{Claim2}
Let $f:\Sigma\times [0,T)\rightarrow\R^{3}$ satisfy
\eqref{PolyharmonicIntro1}.  Let $k\in\mathbb N_0$, set
$m=k+p+1$, and let $s\geq 2m$.  There exists $\varepsilon_0>0$,
depending only on $p$, and a constant $C=C(k,p,s)$ such that, if
\[
    \int_{[\gamma>0]} |A|^2\,d\mu \leq \varepsilon_0,
\]
then for every $\delta\in(0,1]$,
\begin{equation}
\left|
\intM{\nablanorm{k}{A}{2}\partial_t\gamma^s}
\right|
\leq
\delta \llll{\nabla_{(m)}A}_{2,\gamma^s}^{2}
+
C_\delta \cgam^{2m}\A{2}.
\label{Claim2Eqn1}
\end{equation}
\end{lemma}

\begin{proof}
By \eqref{EvolutionOfGammaEqn},
\[
\partial_t\gamma^s
=
s\gamma^{s-1}\partial_t\gamma
=
s(-1)^{p+1}D_\nu\tilde\gamma\,\Delta^pH\,\gamma^{s-1}.
\]
Schematically, \(\Delta^pH=\nabla_{(2p)}A\), and therefore
\begin{equation}
\left|
\intM{\nablanorm{k}{A}{2}\partial_t\gamma^s}
\right|
\leq
c\left|
\intM{
\nabla_{(k)}A*\nabla_{(k)}A*
D_\nu\tilde\gamma*
\nabla_{(2p)}A\,\gamma^{s-1}}
\right|.
\label{E:Claim2-start}
\end{equation}

First suppose \(k\leq p-1\).  Integrating by parts \(p-k-1\) times gives
\begin{align}
\left|
\intM{\nablanorm{k}{A}{2}\partial_t\gamma^s}
\right|
&\leq
c\sum_{u+v+w=p-k-1}
\left|
\intM{
\nabla_{(m)}A*
P_2^{2k+u}(A)*
\nabla_{(v)}D_\nu\tilde\gamma*
\nabla_{(w)}\gamma^{s-1}}
\right|.
\label{E:Claim2-case1}
\end{align}
The point of this integration by parts is that the distinguished
curvature factor now carries exactly \(m=k+p+1\) derivatives.

We next estimate the two cut-off factors.  Combining
\eqref{E:gammanormderhighcovest} and \eqref{E:highcovgammapow}, and
including the case \(w=0\), yields
\begin{equation}
\begin{split}
|\nabla_{(v)}D_\nu\tilde\gamma|\,|\nabla_{(w)}\gamma^{s-1}|
\leq
C\sum_{\ell=0}^{w}
\bigg(
\cgam^{v+w+1}
+
\sum_{a=1}^{v+w-\ell}\sum_{b=1}^{a}
\cgam^{v+w+1-a}|P_b^{a-b}(A)|
\bigg)
\gamma^{s-1-\ell}.
\end{split}
\label{E:Claim2-cutoff}
\end{equation}
The lower limit \(\ell=0\) is essential: when \(w=0\), the factor
\(\nabla_{(w)}\gamma^{s-1}\) is simply \(\gamma^{s-1}\).

Consider first the pure cut-off contribution in
\eqref{E:Claim2-cutoff}.  Since \(u+v+w=p-k-1\), we have
\[
\nabla_{(m)}A*P_2^{2k+u}(A)
=
P_3^{2m-1-\theta}(A),
\qquad
\theta:=v+w+1\geq 1.
\]
Moreover \(\ell+1\leq \theta\), and hence
\(\gamma^{s-1-\ell}\leq \gamma^{s-\theta}\).  Thus
\eqref{E:localinterp-positive-defect} gives, for every
\(\eta\in(0,1]\),
\[
\cgam^{v+w+1}
\left|
\intM{
\nabla_{(m)}A*P_2^{2k+u}(A)\gamma^{s-1-\ell}}
\right|
\leq
\eta M_m+C_\eta\cgam^{2m}E.
\]
For the remaining terms in \eqref{E:Claim2-cutoff}, we combine
\(P_b^{a-b}(A)\) with \(P_2^{2k+u}(A)\).  The integrand becomes
\[
\cgam^{v+w+1-a}
P_{b+3}^{2m-(b+3)+2-\theta}(A),
\qquad
\theta:=v+w+1-a.
\]
Since \(a\leq v+w-\ell\), we again have \(\theta\geq \ell+1\geq 1\).
Therefore \eqref{E:localinterp-positive-defect} applies to each such
term.  Summing over the finite index set in \eqref{E:Claim2-case1}, we
obtain
\begin{equation}
\left|
\intM{\nablanorm{k}{A}{2}\partial_t\gamma^s}
\right|
\leq
\eta M_m+C_\eta\cgam^{2m}E
\label{E:Claim2-case1-final}
\end{equation}
in the case \(k\leq p-1\).

Now suppose \(p\leq k\leq 2p\).  Integrating by parts \(k-p+1\) times in
\eqref{E:Claim2-start} gives
\[
\intM{\nablanorm{k}{A}{2}\partial_t\gamma^s}
=
\intM{
\nabla_{(p-1)}A*
\nabla_{(k-p+1)}
\big(
\nabla_{(k)}A*
D_\nu\tilde\gamma*
\nabla_{(2p)}A*
\gamma^{s-1}
\big)}.
\]
After expanding the derivative, a typical term has the form
\[
\intM{
\nabla_{(p-1)}A*
\nabla_{(k+a)}A*
\nabla_{(b)}D_\nu\tilde\gamma*
\nabla_{(2p+c)}A*
\nabla_{(d)}\gamma^{s-1}},
\qquad
a+b+c+d=k-p+1.
\]
Here
\[
k+a\leq 2k-p+1\leq k+p+1=m,
\qquad
2p+c\leq k+p+1=m,
\]
so no curvature factor has more than \(m\) derivatives.  Estimating the
two cut-off factors exactly as in \eqref{E:Claim2-cutoff}, each resulting
term is of the form
\[
\cgam^\theta
\intM{P_r^{2m-r+2-\theta}(A)\gamma^{s-\ell}},
\qquad
1\leq \ell\leq \theta,\quad \theta\geq 1,
\]
with no factor differentiated more than \(m\) times.  Hence
\eqref{E:localinterp-positive-defect} gives
\[
\left|
\intM{\nablanorm{k}{A}{2}\partial_t\gamma^s}
\right|
\leq
\eta M_m+C_\eta\cgam^{2m}E
\]
also in this case.

It remains to consider \(k\geq 2p+1\).  We integrate by parts \(p+1\)
times in \eqref{E:Claim2-start}:
\[
\intM{\nablanorm{k}{A}{2}\partial_t\gamma^s}
=
\intM{
\nabla_{(k-p-1)}A*
\nabla_{(p+1)}
\big(
\nabla_{(k)}A*
D_\nu\tilde\gamma*
\nabla_{(2p)}A*
\gamma^{s-1}
\big)}.
\]
A typical expanded term is
\[
\intM{
\nabla_{(k-p-1)}A*
\nabla_{(k+a)}A*
\nabla_{(b)}D_\nu\tilde\gamma*
\nabla_{(2p+c)}A*
\nabla_{(d)}\gamma^{s-1}},
\qquad
a+b+c+d=p+1.
\]
Now
\[
k+a\leq k+p+1=m,
\qquad
2p+c\leq 3p+1\leq k+p=m-1,
\]
because \(k\geq 2p+1\).  Thus, as above, after applying
\eqref{E:Claim2-cutoff}, every term has positive defect
\(\theta\geq 1\) and no factor contains more than \(m\) derivatives of
\(A\).  Another application of
\eqref{E:localinterp-positive-defect} yields
\[
\left|
\intM{\nablanorm{k}{A}{2}\partial_t\gamma^s}
\right|
\leq
\eta M_m+C_\eta\cgam^{2m}E .
\]

Combining the three cases and choosing
\(\eta=c(k,p,s)^{-1}\delta\), where the constant absorbs the finite
number of terms produced above, proves \eqref{Claim2Eqn1}.
\end{proof}

\section{Evolution equations for geometric quantities}
The induced metric, measure, normal and curvature satisfy the following evolution equation under \eqref{PolyharmonicIntro1}.
\begin{lemma}\label{EvolutionLemma1}
Let $f:\Sigma\times\left[0,T\right)\rightarrow\mathbb{R}^{3}$ satisfy \eqref{PolyharmonicIntro1}. Then
\begin{align*}
&\frac{\partial}{\partial t}g=2\oo{-1}^{p+1}\oo{\Delta^p H}A,\hspace*{.5cm}\frac{\partial}{\partial t}d\mu=\oo{-1}^{p+1}H\oo{\Delta^{p}H}d\mu, \hspace*{.5cm}\frac{\partial}{\partial t}\nu=\oo{-1}^{p}\nabla\Delta^{p}H,\\
&\frac{\partial}{\partial t}A=\oo{-1}^{p}\Delta^{p+1}A+\sum_{j=0}^{2p}\nabla_{\oo{j}}\oo{\nabla_{\oo{2p-j}}A*A*A},\mbox{ \textnormal{and}  } \frac{\partial}{\partial t}H=\oo{-1}^{p}\left[ \Delta^{p+1}H+(\Delta^{p}H)\norm{A}^{2}\right].
\end{align*}
\end{lemma}
The proof of Lemma \ref{EvolutionLemma1} is standard.
Techniques of the proof that also apply in our case appear for example in the classic paper of Huisken \cite{Huisken2}.

We shall also require the evolution of higher-order derivatives of curvature.
We make use of the $*-$notation, observing that any tensor field $T$ evolves under \eqref{PolyharmonicIntro1} according to
\begin{equation}
\partial_{t}\nabla_{\oo{k}}T=\nabla_{\oo{k}}\partial_{t}T+\sum_{i=0}^{k-1}\nabla_{\oo{i}}\oo{\textrm{II}*\nabla_{\oo{k-1-i}}T}.\label{GeneralTensor1}
\end{equation}
This can be derived formally by induction.  
Here
\[
\textrm{II}:=\partial_{t}\Gamma= \left( \Delta^p H\right) * \nabla A + \nabla \left( \Delta^p H\right) * A =\nabla_{\oo{2p}}A*\nabla A+\nabla_{\oo{2p+1}}A*A\,.
\]
Again, this evolution equation can be derived using techniques as in Huisken \cite{Huisken2}.

We also have
\[
\nabla_{\oo{k}}\Delta^{p+1}T=\nabla_{\oo{k-1}}\Delta^{p+1}\nabla T+\sum_{i=0}^{2p+k}\nabla_{\oo{i}}\oo{\nabla_{\oo{2p+k-i}}T*A*A},
\]
which we use inductively to obtain
\begin{equation}
\nabla_{\oo{k}}\Delta^{p+1}A=\Delta^{p+1}\nabla_{\oo{k}}A+\sum_{i=0}^{2p+k}\nabla_{\oo{i}}\oo{\nabla_{\oo{2p+k-i}}A* A* A}.\label{InterpolationIdentity}
\end{equation}
By combining the evolution equations in Lemma \ref{EvolutionLemma1} with \eqref{GeneralTensor1} and \eqref{InterpolationIdentity}, we obtain the following.

\begin{lemma}\label{EvolutionLemma2}
Let $f:\Sigma\times\left[0,T\right)\rightarrow\mathbb{R}^{3}$ satisfy \eqref{PolyharmonicIntro1}. Then for any $k\in\mathbb{N}_{0}$
\[
\frac{\partial}{\partial t}\nabla_{\oo{k}}A=\oo{-1}^{p}\nabla^{i_{p+1}\dots i_{1}}\nabla_{i_{1}\dots i_{p+1}}\nabla_{\oo{k}}A+\sum_{i=0}^{2p+k}\nabla_{\oo{i}}\oo{\nabla_{\oo{2p+k-i}}A* A* A}.
\]
\end{lemma}

\begin{corollary}\label{EvolutionCorollary1}
Let $f:\Sigma\times\left[0,T\right)\rightarrow\mathbb{R}^{3}$ satisfy \eqref{PolyharmonicIntro1}. Then for any $k\in\mathbb{N}_{0}:= \mathbb{N}\cup \left\{ 0 \right\}$
\begin{equation*}
\frac{\partial}{\partial
t}\nablanorm{k}{A}{2}\\
=2\oo{-1}^{p}\inner{\nabla_{\oo{k}}A,\nabla^{i_{p+1}i_{p}\dots i_{1}}\nabla_{i_{1}\dots i_{p}i_{p+1}}\nabla_{\oo{k}}A}+\sum_{i=0}^{2p+k}\nabla_{\oo{i}}\oo{\nabla_{\oo{2p+k-i}}A* A* A}*\nabla_{\oo{k}}A.
\end{equation*}
\end{corollary}
\begin{proof}
The result follows immediately from Lemma \ref{EvolutionLemma2} and the evolution equation for $g$ from Lemma \ref{EvolutionLemma1}.
\end{proof}
These evolution equations are now used to determine how $||\nabla_{(k)}A||^2_{2,\gamma^s}$ evolves in time.

\begin{corollary}\label{EvolutionCorollary2}
Let $f:\Sigma\times\left[0,T\right)\rightarrow\mathbb{R}^{3}$ satisfy \eqref{PolyharmonicIntro1}. Then for any $k,s\in\mathbb{N}_{0}$
\begin{multline*}
\frac{d}{dt}\intM{\nablanorm{k}{A}{2}\gamma^{s}}+2\intM{\nablanorm{k+p+1}{A}{2}\gamma^{s}}\\
=\intM{\nablanorm{k}{A}{2}\partial_{t}\gamma^{s}}-2\sum_{j=1}^{p+1}\dbinom{p+1}{j}\intM{\inner{\nabla_{\oo{j}}\gamma^{s}\otimes\nabla_{\oo{k+p+1-j}}A,\nabla_{\oo{k+p+1}}A}}\\
+\sum_{i=0}^{2p+k}\intM{\nabla_{\oo{i}}\oo{\nabla_{\oo{2p+k-i}}A* A* A}*\nabla_{\oo{k}}A\,\gamma^{s}}.
\end{multline*}
\end{corollary}
\begin{proof}
Use Corollary \ref{EvolutionCorollary1}, Lemma \ref{EvolutionLemma1}, the product rule, and integrate by parts $(p+1)$ times.
\end{proof}

To estimate the intermediate terms in Corollary \ref{EvolutionCorollary2} we will use \eqref{GammaProp} and the interpolation inequality of Lemma \ref{NewLemma1} in the Appendix.

\begin{proposition}\label{EvolutionProposition1}
Let $f:\Sigma\times[0,T)\to\mathbb{R}^3$ satisfy \eqref{PolyharmonicIntro1}.
Let $k\in\mathbb N_0$, set $m=k+p+1$, and let $s\ge 2m$.
There exists $\varepsilon_0>0$ depending only on $p$ and a constant
$C=C(k,p,s)$ such that, if
\[
    \int_{[\gamma>0]} |A|^2\,d\mu \le \varepsilon_0,
\]
then for every $\delta\in(0,1]$,
\[
\frac{d}{dt}\int_\Sigma |\nabla_{(k)}A|^2\gamma^s\,d\mu
+
\left(2-\delta-C\|A\|^2_{2,[\gamma>0]}\right)
\int_\Sigma |\nabla_{(m)}A|^2\gamma^s\,d\mu
\le
C_\delta\,c_\gamma^{2m}\|A\|^2_{2,[\gamma>0]}.
\]
\end{proposition}
\begin{proof}
Set
\[
    m:=k+p+1,\qquad
    M_m:=\llll{\nabla_{(m)}A}_{2,\gamma^s}^{2},
    \qquad
    E:=\A{2}.
\]
We choose \(\varepsilon_0\leq 1\).  This is enough for the local
interpolation estimates \eqref{E:localinterp-positive-defect} and
\eqref{E:localinterp-critical}; no smallness constant depending on
\(k\) is required at this point.

By Corollary \ref{EvolutionCorollary2},
\begin{align}
\frac{d}{dt}\intM{\nablanorm{k}{A}{2}\gamma^s}
+2M_m
&=
\intM{\nablanorm{k}{A}{2}\partial_t\gamma^s}\\ \nonumber
& \quad -2\sum_{j=1}^{p+1}\binom{p+1}{j}
\intM{
\inner{\nabla_{(j)}\gamma^s\otimes\nabla_{(m-j)}A,\nabla_{(m)}A}}
\\ \nonumber
& \quad+
\sum_{i=0}^{2p+k}
\intM{
\nabla_{(i)}
\big(
\nabla_{(2p+k-i)}A*A*A
\big)
*
\nabla_{(k)}A\,\gamma^s}.
\label{E:EvolutionProposition1-start}
\end{align}
We estimate the three terms on the right-hand side.

The first term is controlled by Lemma \ref{Claim2}: for every
\(\eta\in(0,1]\),
\begin{equation}
\left|
\intM{\nablanorm{k}{A}{2}\partial_t\gamma^s}
\right|
\leq
\eta M_m+C_\eta\cgam^{2m}E .
\label{E:EvolutionProposition1-time}
\end{equation}

We next estimate the terms containing derivatives of \(\gamma^s\).
Fix \(j\in\{1,\ldots,p+1\}\).  From
\eqref{E:highcovgammapow},
\[
|\nabla_{(j)}\gamma^s|
\leq
C\cgam^j\gamma^{s-j}
+
C\sum_{\ell=1}^{j-1}\sum_{a=1}^{j-\ell}\sum_{b=1}^{a}
\cgam^{j-a}|P_b^{a-b}(A)|\gamma^{s-\ell}.
\]
The pure cut-off term gives
\[
\cgam^j
\intM{\nablanorm{m-j}{A}{}\nablanorm{m}{A}{}\gamma^{s-j}}
=
\cgam^j
\intM{|P_2^{2m-j}(A)|\gamma^{s-j}},
\]
which is of the form \eqref{E:localinterp-positive-defect} with
\(\theta=j\).  Hence
\[
\cgam^j
\intM{\nablanorm{m-j}{A}{}\nablanorm{m}{A}{}\gamma^{s-j}}
\leq
\eta M_m+C_\eta\cgam^{2m}E .
\]
For the remaining terms we combine \(P_b^{a-b}(A)\),
\(\nabla_{(m-j)}A\), and \(\nabla_{(m)}A\).  Since
\(a\leq j-\ell\), the defect
\[
    \theta:=j-a
\]
satisfies \(\theta\geq\ell\geq 1\), and the resulting product has the
form
\[
\cgam^\theta
P_{b+2}^{2m-(b+2)+2-\theta}(A)\gamma^{s-\ell}.
\]
As \(\ell\leq\theta\), this is again controlled by
\eqref{E:localinterp-positive-defect}.  Therefore, for each
\(j=1,\ldots,p+1\),
\begin{equation}
\left|
\intM{
\inner{\nabla_{(j)}\gamma^s\otimes\nabla_{(m-j)}A,\nabla_{(m)}A}}
\right|
\leq
\eta M_m+C_\eta\cgam^{2m}E .
\label{E:EvolutionProposition1-gammaterms}
\end{equation}

It remains to estimate the nonlinear evolution terms.  Define
\[
X_i:=
\intM{
\nabla_{(i)}
\big(
\nabla_{(2p+k-i)}A*A*A
\big)
*
\nabla_{(k)}A\,\gamma^s},
\qquad
i=0,\ldots,2p+k .
\]

First suppose \(i\leq p-1\).  Integrating by parts \(p-i-1\) times gives
\begin{align}
X_i
&=
\sum_{u=0}^{p-i-1}\sum_{v=0}^{i+u}
\intM{
\nabla_{(m)}A*
P_3^{m-v-2}(A)*
\nabla_{(v)}\gamma^s}.
\label{E:EvolutionProposition1-Xi-low}
\end{align}
The contribution with \(v=0\) is a critical product:
\[
\intM{P_4^{2m-2}(A)\gamma^s}.
\]
By \eqref{E:localinterp-critical},
\[
\left|
\intM{P_4^{2m-2}(A)\gamma^s}
\right|
\leq
C E M_m+C\cgam^{2m}E .
\]
For \(v\geq 1\), using \eqref{E:highcovgammapow} gives a pure cut-off
part
\[
\cgam^v
\intM{
\nabla_{(m)}A*P_3^{m-v-2}(A)\gamma^{s-v}},
\]
which is of the form \eqref{E:localinterp-positive-defect} with
\(\theta=v\).  The terms containing curvature from
\(\nabla_{(v)}\gamma^s\) have the form
\[
\cgam^{v-a}
\intM{
\nabla_{(m)}A*
P_{3+b}^{m-v-2+a-b}(A)
\gamma^{s-\ell}},
\qquad
1\leq\ell\leq v-a,
\]
and are again controlled by \eqref{E:localinterp-positive-defect}, with
\(\theta=v-a\).  Thus, for \(i\leq p-1\),
\begin{equation}
|X_i|
\leq
\eta M_m+C E M_m+C_\eta\cgam^{2m}E .
\label{E:EvolutionProposition1-Xi-low-est}
\end{equation}

Now let \(i\geq p\).  Integrating by parts \(p-1\) times, with no
integration by parts needed when \(p=1\), gives
\[
X_i
=
\intM{
\nabla_{(i-p+1)}
\big(
\nabla_{(2p+k-i)}A*A*A
\big)
*
\nabla_{(p-1)}
\big(
\nabla_{(k)}A\gamma^s
\big)}.
\]
After expanding the derivatives, the terms in which no derivative falls
on \(\gamma^s\) are schematic \(P_4^{2m-2}(A)\gamma^s\)-terms.  Moreover
no curvature factor contains more than \(m\) derivatives: the first
bracket contains at most
\[
2p+k-i+i-p+1=k+p+1=m
\]
derivatives on a single curvature factor, while the second bracket gives
at most
\[
k+p-1=m-2
\]
derivatives on its curvature factor.  Hence the no-cutoff terms are
bounded by \eqref{E:localinterp-critical}:
\[
\left|
\intM{P_4^{2m-2}(A)\gamma^s}
\right|
\leq
C E M_m+C\cgam^{2m}E .
\]
If at least one derivative falls on \(\gamma^s\), then
\eqref{E:highcovgammapow} produces terms of the schematic form
\[
\cgam^\theta
\intM{
P_r^{2m-r+2-\theta}(A)\gamma^{s-\ell}},
\qquad
1\leq \ell\leq \theta ,
\]
again with no factor differentiated more than \(m\) times.  These terms
are controlled by \eqref{E:localinterp-positive-defect}.  Therefore
\begin{equation}
|X_i|
\leq
\eta M_m+C E M_m+C_\eta\cgam^{2m}E
\label{E:EvolutionProposition1-Xi-high-est}
\end{equation}
for every \(i\geq p\).

Summing \eqref{E:EvolutionProposition1-Xi-low-est} and
\eqref{E:EvolutionProposition1-Xi-high-est} over
\(i=0,\ldots,2p+k\), and combining with
\eqref{E:EvolutionProposition1-time} and
\eqref{E:EvolutionProposition1-gammaterms}, we obtain
\[
\frac{d}{dt}\intM{\nablanorm{k}{A}{2}\gamma^s}
+2M_m
\leq
\big(\eta+C E\big)M_m+C_\eta\cgam^{2m}E .
\]
Choosing \(\eta=c(k,p,s)^{-1}\delta\), absorbing harmless numerical constants into
\(C\) and \(C_\delta\), gives
\[
\frac{d}{dt}\intM{\nablanorm{k}{A}{2}\gamma^s}
+
\big(2-\delta-C\A{2}\big)
\llll{\nabla_{(m)}A}_{2,\gamma^s}^{2}
\leq
C_\delta\cgam^{2m}\A{2}.
\]
This is the desired estimate.
\end{proof}

Now we convert Proposition \ref{EvolutionProposition1} to a local curvature estimate under a small concentration condition.
This result shows that small concentration on an open set of a certain size persists along the flow to at least an open strict subset.

\begin{proposition}\label{PN:estA}
Let $f:\Sigma\times[0,T)\rightarrow\mathbb{R}^3$ satisfy
\eqref{PolyharmonicIntro1}.  There are constants
$C<\infty$ and $\varepsilon_0>0$, depending only on $p$, such that if
\begin{equation}
\sup_{0\leq \tau\leq T^*}\int_{[\gamma>0]} |A|^2\,d\mu\leq\varepsilon_0,
\label{EvolutionProposition1,0}
\end{equation}
then for every $t\in[0,T^*]$,
\begin{equation}
\int_{[\gamma=1]} |A|^2\,d\mu
+ \int_0^t \int_{[\gamma=1]} |\nabla_{(p+1)}A|^2\,d\mu\,d\tau
\le
\left.\int_{[\gamma>0]} |A|^2\,d\mu\right|_{t=0}
    + C\,\cgam^{2(p+1)}\varepsilon_0 t.
\label{EQ:estA}
\end{equation}
\end{proposition}

\begin{proof}
Apply Proposition \ref{EvolutionProposition1} with $k=0$,
$m=p+1$, $s=2(p+1)$, and $\delta=1/2$.  After decreasing
$\varepsilon_0$ if necessary, the coefficient of
$\llll{\nabla_{(p+1)}A}_{2,\gamma^{2(p+1)}}^2$ is bounded below by a
positive constant depending only on $p$.  Hence
\[
\frac{d}{dt}\intM{|A|^2\gamma^{2(p+1)}}
+
\intM{|\nabla_{(p+1)}A|^2\gamma^{2(p+1)}}
\leq
C\cgam^{2(p+1)}\varepsilon_0 .
\]
Integrating this inequality over $[0,t]$ and using
$\gamma\equiv1$ on $[\gamma=1]$ gives \eqref{EQ:estA}.
\end{proof}

We do not wish to use Proposition \ref{EvolutionProposition1} directly to control the concentration of higher derivatives of curvature.
This is because the constant $C$ depends on $k$, so obtaining uniform estimates for $k$ larger and larger will drive $C$ larger and larger, and this in turn will require our smallness condition to tighten.
To avoid this,  we use the control given by Proposition \ref{PN:estA} and in particular \eqref{EQ:estA} to control a variety of curvature integrals in terms of their $L^1$-in-time norm.
Then, we derive a new evolution equation to take advantage of this control.
Finally, we combine these estimates together to prove a version of \ref{EvolutionProposition1} where the constant on the left hand side is universal.

\begin{lemma}\label{LM:goodPguys}
There are constants $\varepsilon_0>0$ and $C<\infty$, depending only on
$p$, such that if
\[
    \int_{[\gamma>0]}|A|^2\,d\mu\leq\varepsilon_0,
\]
then
\[
    \sum_{i=2}^{p+2} \int_{[\gamma=1]} P_i^{p+2-i}(A)^{\#2}\,d\mu
    \le C\llll{\nabla_{(p+1)}A}_{2,\gamma^{2(p+1)}}^2
    + C\,\cgam^{2(p+1)}\llll{A}_{2,[\gamma>0]}^2 .
\]
\end{lemma}

\begin{proof}
Because $\gamma\equiv1$ on $[\gamma=1]$, it is enough to estimate the
same integrals with the weight $\gamma^{2(p+1)}$.  Applying
Theorem \ref{AppendixTheorem1} with $n=p+2$ gives
\[
\sum_{i=2}^{p+2}\intM{\oo{P_i^{p+2-i}(A)}^{\#2}\gamma^{2(p+1)}}
\leq
C\A{2}
\left(
\llll{\nabla_{(p+1)}A}_{2,\gamma^{2(p+1)}}^2
+
\cgam^{2(p+1)}\A{2}
\right).
\]
The asserted estimate follows after decreasing $\varepsilon_0$ so that
$\A{2}\leq1$.
\end{proof}

We proceed by customising two $\varepsilon$-regularity style inequalities from Kuwert-Sch\"{a}tzle \cite{Kuwert2}; the second  is time-dependent. This is not only interesting in its own right, but is used later in Theorem \ref{T:interiorests} to prove interior estimates along the flow.  The first three authors proved a similar result for a general tensor $T$ in \cite{mccoy2015geometric}.

\begin{proposition}\label{Chapter3NewProp2}
Let $f:\Sigma^{2}\rightarrow\mathbb{R}^{3}$ be an immersion. If the small energy assumption 
\[
\intgamma{\norm{A}^{2}}\leq\varepsilon
\]
holds for $\varepsilon>0$ sufficiently small then for any $m\in\mathbb{N}_{0}$, there exists a constant $C>0$ that depends only on $m$ such that
\begin{equation}
\llll{\nabla_{\oo{m}}A}_{\infty,\cc{\gamma=1}}^{2\oo{m+2}}\leq C\A{2}\oo{\llll{\nabla_{\oo{m+2}}A}_{2,\gamma^{2\oo{m+2}}}^{2\oo{m+1}}+\cgam^{2\oo{m+1}\oo{m+2}}\A{2\oo{m+1}}}.\label{Chapter3NewProp2,0}
\end{equation}
\end{proposition}
\begin{proof}
Using the multiplicative Sobolev inequality from Theorem \ref{EvolutionTheorem1} with $\varphi=\nablanorm{m}{A}{}\gamma^{\frac{2m+3}{2}}$ gives
\begin{align}
&\llll{\nabla_{\oo{m}}A}_{\infty,\cc{\gamma=1}}^{6}\\
&\leq c\llll{\nabla_{\oo{m}}A}_{2,\gamma^{2m+3}}^{2}\\
& \qquad\cdot\Biggl(\intM{\nablanorm{m+1}{A}{4}\gamma^{2\oo{2m+3}}}+\cgam^{4}\intM{\nablanorm{m}{A}{4}\gamma^{2\oo{2m+1}}} +\intM{\nablanorm{m}{A}{4}H^{4}\gamma^{2\oo{2m+3}}}\Biggr)\\
&\leq c\llll{\nabla_{\oo{m}}A}_{2,\gamma^{2m}}^{2}\\
&\qquad \cdot\Biggl(\intM{\nablanorm{m+1}{A}{4}\gamma^{2\oo{2m+3}}}+\cgam^{4}\intM{\nablanorm{m}{A}{4}\gamma^{2\oo{2m+1}}} +\intM{\nablanorm{m}{A}{4}H^{4}\gamma^{2\oo{2m+3}}}\Biggr).\label{Chapter3NewProp2,1}
\end{align}
We estimate each of these terms separately. Firstly, we use Theorem \ref{MichaelSimon} with $u=\nablanorm{m+1}{A}{2}\gamma^{2m+3}$ to estimate the second term:
\begin{align}
&\intM{\nablanorm{m+1}{A}{4}\gamma^{2\oo{2m+3}}}\\
&\leq c\Bigl(\intM{\nablanorm{m+2}{A}{}\nablanorm{m+1}{A}{}\gamma^{2m+3}}+\cgam\llll{\nabla_{\oo{m+1}}A}_{2,\gamma^{2(m+1)}}^{2} +\intM{\nablanorm{m+1}{A}{2}\norm{H}\gamma^{2m+3}}\Bigr)^{2}\\
&\leq c\llll{\nabla_{\oo{m+1}}A}_{2,\gamma^{2(m+1)}}^{2}\\
&\qquad \cdot\Bigl(\llll{\nabla_{\oo{m+2}}A}_{2,\gamma^{2(m+2)}}^{2} +\cgam^{2}\llll{\nabla_{\oo{m+1}}A}_{2,\gamma^{2(m+1)}}^{2}+\intM{\nablanorm{m+1}{A}{2}\norm{A}^{2}\gamma^{2\oo{m+2}}}\Bigr)\\
&\leq c\llll{\nabla_{\oo{m+1}}A}_{2,\gamma^{2(m+1)}}^{2}\cdot\Bigl(\llll{\nabla_{\oo{m+2}}A}_{2,\gamma^{2(m+2)}}^{2}+\intM{\oo{P_{2}^{m+1}\oo{A}}^{\#2}\gamma^{2\oo{m+2}}}+\cgam^{2\oo{m+2}}\A{2}\Bigr)\\
&\leq c\llll{\nabla_{\oo{m+1}}A}_{2,\gamma^{2(m+1)}}^{2}\cdot\oo{\llll{\nabla_{\oo{m+2}}A}_{2,\gamma^{2(m+2)}}^{2}+\cgam^{2\oo{m+2}}\A{2}}.\label{Chapter3NewProp2,3}
\end{align}
Similarly, using Theorem \ref{MichaelSimon} with $u=\nablanorm{m}{A}{2}\gamma^{2m+1}$ gives
\begin{align}
&\cgam^{4}\intM{\nablanorm{m}{A}{4}\gamma^{2\oo{2m+1}}}\\
&\leq c\,\cgam^{4}\Bigl(\intM{\nablanorm{m+1}{A}{}\nablanorm{m}{A}{}\gamma^{2m+1}}+\cgam\llll{\nabla_{\oo{m}}A}_{2,\gamma^{2m}}^{2}+\intM{\nablanorm{m}{A}{2}\norm{H}\gamma^{2m+1}}\Bigr)^{2}\\
&\leq c\,\cgam^{4}\llll{\nabla_{\oo{m}}A}_{2,\gamma^{2m}}^{2}\cdot\Bigl(\llll{\nabla_{\oo{m+1}}A}_{2,\gamma^{2(m+1)}}^{2}+\cgam^{2}\llll{\nabla_{\oo{m}}A}_{2,\gamma^{2m}}^{2}+\intM{\nablanorm{m}{A}{2}\norm{A}^{2}\gamma^{2\oo{m+1}}}\Bigr)\\
&\leq c\,\cgam^{4}\llll{\nabla_{\oo{m}}A}_{2,\gamma^{2m}}^{2}\cdot\Bigl(\llll{\nabla_{\oo{m+1}}A}_{2,\gamma^{2(m+1)}}^{2}+\intM{\oo{P_{2}^{m}\oo{A}}^{\#2}\gamma^{2\oo{m+1}}}+\cgam^{2\oo{m+1}}\A{2}\Bigr)\\
&\leq c\,\cgam^{4}\llll{\nabla_{\oo{m}}A}_{2,\gamma^{2m}}^{2}\cdot\oo{\llll{\nabla_{\oo{m+1}}A}_{2,\gamma^{2(m+1)}}^{2}+\cgam^{2\oo{m+1}}\A{2}}\\
&\leq c\oo{\llll{\nabla_{\oo{m+2}}A}_{2,\gamma^{2(m+2)}}^{2}+\cgam^{2\oo{m+2}}\A{2}}\oo{\llll{\nabla_{\oo{m+1}}A}_{2,\gamma^{2(m+1)}}^{2}+\cgam^{2\oo{m+1}}\A{2}}.\label{Chapter3NewProp2,4}
\end{align}
For the last term in \eqref{Chapter3NewProp2,1} we use Theorem \ref{MichaelSimon} again, this time with $u=\nablanorm{m}{A}{2}H^{2}\gamma^{2m+3}$:
\begin{align}
&\intM{\nablanorm{m}{A}{4}H^{4}\gamma^{2\oo{2m+3}}}
\leq c\Bigl(\intM{\nablanorm{m+1}{A}{}\nablanorm{m}{A}{}\gamma^{2m+3}}+\intM{\nablanorm{m}{A}{2}\norm{\nabla H}\norm{H}\gamma^{2m+3}}\\
& \hspace*{5.5cm} +\cgam\intM{\nablanorm{m}{A}{2}H^{2}\gamma^{2\oo{m+1}}}
+\intM{\nablanorm{m}{A}{2}\norm{H}^{3}\gamma^{2m+3}}\Bigr)^{2}\\
&\leq c\intM{\nablanorm{m}{A}{2}\norm{A}^{2}\gamma^{2\oo{m+1}}}\cdot\Bigl(\intM{\nablanorm{m+1}{A}{2}\norm{A}^{2}\gamma^{2\oo{m+2}}}+\intM{\nablanorm{m}{A}{2}\norm{\nabla A}^{2}\gamma^{2\oo{m+2}}}\\
&\qquad \qquad +\cgam^{2}\intM{\nablanorm{m}{A}{2}\norm{A}^{2}\gamma^{2\oo{m+1}}}+\intM{\nablanorm{m}{A}{2}\norm{A}^{4}\gamma^{2\oo{m+2}}}\Bigr)\\
&\leq c\intM{\oo{P_{2}^{m}\oo{A}}^{\#2}\gamma^{2\oo{m+1}}}\cdot\Bigl(\sum_{i=2}^{3}\intM{\oo{P_{i}^{m+3-i}\oo{A}}^{\#2}\gamma^{2\oo{m+2}}}+\cgam^{2}\intM{\oo{P_{2}^{m}\oo{A}}^{\#2}\gamma^{2\oo{m+1}}}\Biggr)\\
&\leq c\A{4}\oo{\intM{\nablanorm{m+1}{A}{2}\gamma^{2\oo{m+1}}}+\cgam^{2\oo{m+1}}\A{2}}\\
&\qquad\cdot\Bigl(\llll{\nabla_{\oo{m+2}}A}_{2,\gamma^{2(m+2)}}^{2}+\cgam^{2}\llll{\nabla_{\oo{m+1}}A}_{2,\gamma^{2(m+1)}}^{2}+\cgam^{2\oo{m+2}}\A{2}\Bigr)\\
&\leq c\A{4}\oo{\llll{\nabla_{\oo{m+1}}A}_{2,\gamma^{2(m+1)}}^{2}+\cgam^{2\oo{m+1}}\A{2}}\\
&\qquad \cdot\oo{\llll{\nabla_{\oo{m+2}}A}_{2,\gamma^{2(m+2)}}^{2}+\cgam^{2\oo{m+2}}\A{2}}.\label{Chapter3NewProp2,5}
\end{align}
Substituting \eqref{Chapter3NewProp2,3},\eqref{Chapter3NewProp2,4} and \eqref{Chapter3NewProp2,5} into \eqref{Chapter3NewProp2,1}, and using Lemma \ref{NewLemma1} and Young's inequality multiple times: first twice with conjugate exponents $p=\frac{2m+1}{m+1},p^{*}=\frac{2m+1}{m}$ and then twice with $p=\frac{3\oo{m+1}}{m+2},p^{*}=\frac{3\oo{m+1}}{2m+1}$, gives
\begin{align}
&\llll{\nabla_{\oo{m}}A}_{\infty,\cc{\gamma=1}}^{6}
 \leq c\llll{\nabla_{\oo{m}}A}_{2,\gamma^{2m}}^{2}\cdot\oo{\llll{\nabla_{\oo{m+1}}A}_{2,\gamma^{2\oo{m+1}}}^{2}+\cgam^{2\oo{m+1}}\A{2}}\\
 &\hspace*{6.5cm}\cdot\oo{\llll{\nabla_{\oo{m+2}}A}_{2,\gamma^{2\oo{m+2}}}^{2}+\cgam^{2\oo{m+2}}\A{2}}\\
&\leq c\oo{\A{\frac{4}{m+2}}\llll{\nabla_{\oo{m+2}}A}_{2,\gamma^{2\oo{m+2}}}^{\frac{2m}{m+2}}+\cgam^{2m}\A{2}}\\
& \qquad \cdot \oo{\A{\frac{2}{m+2}}\llll{\nabla_{\oo{m+2}}A}_{2,\gamma^{2\oo{m+2}}}^{\frac{2\oo{m+1}}{m+2}}+\cgam^{2\oo{m+1}}\A{2}}
\cdot\oo{\llll{\nabla_{\oo{m+2}}A}_{2,\gamma^{2\oo{m+2}}}^{2}+\cgam^{2\oo{m+2}}\A{2}}\\
&\leq C\A{\frac{6}{m+2}}\Biggl(\llll{\nabla_{\oo{m+2}}A}_{2,\gamma^{2\oo{m+2}}}^{\frac{2\oo{2m+1}}{m+2}}+\cgam^{2\oo{m+1}}\A{\frac{2\oo{m+1}}{m+2}}\llll{\nabla_{\oo{m+2}}A}_{2,\gamma^{2\oo{m+2}}}^{\frac{2m}{m+2}}\\
&\qquad +\cgam^{2m}\A{\frac{2m}{m+2}}\llll{\nabla_{\oo{m+2}}A}_{2,\gamma^{2\oo{m+2}}}^{\frac{2\oo{m+1}}{m+2}}+\cgam^{2\oo{2m+1}}\A{\frac{2\oo{2m+1}}{m+2}}\Biggr)\\
& \qquad \cdot\oo{\llll{\nabla_{\oo{m+2}}A}_{2,\gamma^{2\oo{m+2}}}^{2}+\cgam^{2\oo{m+2}}\A{2}}\\
&\leq C\A{\frac{6}{m+2}}\oo{\llll{\nabla_{\oo{m+2}}A}_{2,\gamma^{2\oo{m+2}}}^{\frac{2\oo{2m+1}}{m+2}}+\cgam^{2\oo{2m+1}}\A{\frac{2\oo{2m+1}}{m+2}}}\\
&\qquad \cdot \oo{\llll{\nabla_{\oo{m+2}}A}_{2,\gamma^{2\oo{m+2}}}^{2}+\cgam^{2\oo{m+2}}\A{2}}\\
&\leq C\A{\frac{6}{m+2}}\cdot\Biggl(\llll{\nabla_{\oo{m+2}}A}_{2,\gamma^{2\oo{m+2}}}^{\frac{6\oo{m+1}}{m+2}}+\cgam^{2\oo{m+2}}\A{2}\llll{\nabla_{\oo{m+2}}A}_{2,\gamma^{2\oo{m+2}}}^{\frac{2\oo{2m+1}}{m+2}}\\
&\qquad \qquad +\cgam^{2\oo{2m+1}}\A{\frac{2\oo{2m+1}}{m+2}}\llll{\nabla_{\oo{m+2}}A}_{2,\gamma^{2\oo{m+2}}}^{2}+\cgam^{6\oo{m+1}}\A{\frac{6\oo{m+1}}{m+2}}\Biggr)\\
&\leq C\A{\frac{6}{m+2}}\cdot\oo{\llll{\nabla_{\oo{m+2}}A}_{2,\gamma^{2\oo{m+2}}}^{\frac{6\oo{m+1}}{m+2}}+\cgam^{6\oo{m+1}}\A{\frac{6\oo{m+1}}{m+2}}}.
\end{align}
where the coefficient $C>0$ on the right hand side depends only on $m$. Raising each side to the power $\oo{m+2}/3$ then finishes the proof.
\end{proof}

\begin{corollary}\label{InfinityCorollary}
Let $f:\Sigma^{2}\rightarrow\R^{3}$ be an immersion. If the small energy assumption 
\[
\intgamma{\norm{A}^{2}}\leq\varepsilon
\]
holds for $\varepsilon>0$ sufficiently small then there exists a universal constant $c>0$ such that
\begin{equation}
\llll{A}_{\infty,[\gamma=1]}^{4} \leq c\A{2}\big(||\nabla_{(2)}A||_{2,\gamma^{4}}^{2} + \cgam^{4}\A{2}\big).\label{InfinityCorollary:0}
\end{equation}
Moreover, for any $m\in\mathbb{N}$ there exists a constant $C>0$ which depends only on $m$ such that
\begin{equation}
\llll{A}_{\infty,[\gamma=1]}^{2(m+1)}\leq C \A{2m} \big(||\nabla_{(m+1)}A||_{2,\gamma^{2(m+1)}}^{2} + \cgam^{2(m+1)}\A{2}\big)\label{InfinityCorollary:1}
\end{equation}
\end{corollary}
\begin{proof}
Applying Proposition \ref{Chapter3NewProp2} with $m=0$ immediately proves \eqref{InfinityCorollary:0}, noting that the coefficient $c>0$ on the right hand side becomes universal since $m=0$ is fixed. 
To prove \eqref{InfinityCorollary:1}, we note that by  the interpolative inequalities of Lemma \ref{NewLemma1}, for any $m\in\mathbb{N}$ there exists a constant $C>0$ which depends only on $m$ such that
\[
||\nabla_{(2)}A||_{2,\gamma^{4}}^{2} \leq C \A{\frac{2(m-1)}{m+1}}\llll{\nabla_{(m+1)}A}_{2,\gamma^{2(m+1)}}^{\frac{4}{m+1}} + C\,\cgam^{4}\A{2}.
\]
Combining the above with \eqref{InfinityCorollary:0} (which was just proven) then yields
\begin{align}
\llll{A}_{\infty,[\gamma=1]}^{4} &\leq C\A{2}\big(\A{\frac{2(m-1)}{m+1}}\llll{\nabla_{(m+1)}A}_{2,\gamma^{2(m+1)}}^{\frac{4}{m+1}} + \cgam^{4}\A{2}\big)\\
&= C\A{\frac{4m}{m+1}}\big(\llll{\nabla_{(m+1)}A}_{2,\gamma^{2(m+1)}}^{\frac{4}{m+1}}+\cgam^{4}\A{\frac{4}{m+1}}\big)
\end{align}
for some constant $C>0$ which depends only on $m$. Raising both sides of the above inequality to the power $(m+1)/2$ then proves \eqref{InfinityCorollary:1}.
\end{proof}

Proposition \ref{Chapter3NewProp2} has the following fixed-order consequence, modelled on the local smoothing step of Kuwert--Sch\"atzle.  The smallness threshold is allowed to depend on the derivative order.

\begin{proposition}\label{NewProposition1}
Suppose that $f:\Sigma\times\left[0,T\right)\rightarrow\mathbb{R}^{3}$ satisfies \eqref{PolyharmonicIntro1}.  For each fixed $k\in\mathbb N_0$ there is an $\varepsilon_k>0$, depending only on $k$ and $p$ (and on the fixed cut-off profile), such that if
\[
\sup_{\cc{0,T^{*}}}\int_{\cc{\gamma>0}}{\norm{A}^{2}\,d\mu}\leq\varepsilon_k,
\]
then there is a constant $\tilde c_k$, depending only on $k,T^{*},\cgam,\varepsilon_k$ and
\[
\alpha_{0}\oo{k+2}:=\sum_{j=0}^{k+2}\llll{\nabla_{\oo{j}}A}_{2,\cc{\gamma>0}}^{2}\Big|_{t=0},
\]
such that
\begin{equation}
\llll{\nabla_{\oo{k}}A}_{\infty,\cc{\gamma=1}}^{2}\leq\tilde{c}_{k}.
\label{NewProposition1,1}
\end{equation}
\end{proposition}
\begin{proof}
The point is that $k$ is fixed.  Choose $\varepsilon_k$ smaller than the fixed-order smallness thresholds needed in Proposition \ref{EvolutionProposition1} for the derivative orders $k$ and $k+2$, and smaller than the smallness threshold in Proposition \ref{Chapter3NewProp2}.  Applying Proposition \ref{EvolutionProposition1} with a cut-off $\gamma_{0,1/2}$ and absorbing the top-order term gives
\begin{equation}
\llll{\nabla_{\oo{j}}A}_{2,\cc{\gamma\geq1/2}}^{2}
\leq
\llll{\nabla_{\oo{j}}A}_{2,\cc{\gamma>0}}^{2}\Big|_{t=0}
+C_j\,\cgam^{2\oo{j+p+1}}\varepsilon_k t
\leq
c_j\oo{\alpha_0(j),j,T^*,\cgam,\varepsilon_k}
\label{NewProposition1,2}
\end{equation}
for each $j\in\{k,k+2\}$.  Combining \eqref{NewProposition1,2} with Proposition \ref{Chapter3NewProp2}, using the cut-off $\gamma_{1/2,1}$, yields
\begin{multline*}
\llll{\nabla_{\oo{k}}A}_{\infty,\cc{\gamma=1}}^{2\oo{k+2}}
\leq
c\A{2}\oo{\llll{\nabla_{\oo{k+2}}A}_{2,\gamma_{1/2,1}}^{2\oo{k+1}}+
\cgam^{2\oo{k+1}\oo{k+2}}\A{2\oo{k+1}}}
\\
\leq
c\oo{\varepsilon_k,\alpha_0\oo{k+2},k,T^*,\cgam}.
\end{multline*}
Taking the $2(k+2)$-th root proves \eqref{NewProposition1,1}.
\end{proof}

\section{High order interpolation and Sobolev inequalities}
\label{S:interp}

\begin{lemma}\label{HigherOrderSobolevLemma1}
Let $l\in\mathbb{N}$, $f:\Sigma\rightarrow\mathbb{R}^{3}$ an immersion and $\gamma$ a cut-off function as in \eqref{GammaProp0}. Define  
\[
S_{l}:=\intM{|\Delta^{\frac{l}{2}}H|^{2}\gamma^{2l}}.
\]
Then if $\Ao{2}\leq\varepsilon_{0}$ for $\varepsilon_{0}>0$ sufficiently small, there exists an absolute constant $c>0$ such that 
\begin{equation}
S_{l}\leq c\oo{\Ao{\frac{2}{l+1}}S_{l+1}^{\frac{l}{l+1}}+\cgam^{2l}\Ao{2}}.\label{HigherOrderSobolevLemma1Eqn1}
\end{equation}
\end{lemma}
\begin{proof}
The proof is completed by combining induction and multiple applications of integration by parts. Note that combining \eqref{MultliplicativeSobolevLemma2Eqn2} and \eqref{MultliplicativeSobolevLemma2Eqn7,0} proves \eqref{HigherOrderSobolevLemma1Eqn1} for $l=1$. Next assume that the statement is true for $l=m\in\mathbb{N}$:
\begin{equation}
S_{m} \leq c\oo{\Ao{\frac{2}{m+1}}S_{m+1}^{\frac{m}{m+1}}+\cgam^{2m}\Ao{2}}.\label{HigherOrderSobolevLemma1Eqn2}
\end{equation}
Then integration by parts, together with the Cauchy-Schwarz and Young's inequalities and the inductive assumption \eqref{HigherOrderSobolevLemma1Eqn2} gives
\begin{align}
&S_{m+1} = \intM{|\Delta^{\frac{m+1}{2}}H|^{2}\gamma^{2(m+1)}}\\
= \intM{\Delta^{\frac{m}{2}}H * \Delta^{\frac{m+2}{2}}H * \gamma^{2(m+1)}} + \intM{\Delta^{\frac{m}{2}}H * \Delta^{\frac{m+1}{2}}H * \nabla\gamma * \gamma^{2m+1}}\\
&\leq c \, S_{m}^{\frac{1}{2}}S_{m+2}^{\frac{1}{2}} + \eta\,S_{m+1} + c\,\eta^{-1}\cgam^{2}S_{m}\\
&\leq c\oo{\Ao{\frac{2}{m+1}}S_{m+1}^{\frac{m}{m+1}}+\cgam^{2m}\Ao{2}}^{\frac{1}{2}} S_{m+2}^{\frac{1}{2}}\\
&\qquad + \eta\,S_{m+1} + c\,\eta^{-1}\cgam^{2}\oo{\Ao{\frac{2}{m+1}}S_{m+1}^{\frac{m}{m+1}}+\cgam^{2m}\Ao{2}},
\end{align}
which holds for any $\eta>0$. Next using Young's inequality with conjugate exponents $p=2(m+1)/m$ and $p^{*}=2(m+1)/(m+2)$ gives
\[
\Ao{\frac{1}{m+1}}S_{m+1}^{\frac{m}{2(m+1)}}S_{m+2}^{\frac{1}{2}} \leq \eta S_{m+1} + c\,\eta^{-1}\Ao{\frac{2}{m+2}}S_{m+2}^{\frac{m+1}{m+2}}.
\]
Similarly,
\[
\cgam^{m}\Ao{}S_{m+2}^{\frac{1}{2}}
\leq
c\oo{
\Ao{\frac{2}{m+2}}S_{m+2}^{\frac{m+1}{m+2}}
+
\cgam^{2(m+1)}\Ao{2}
}.\]
and
\[
c\,\eta^{-1}\cgam^{2} \Ao{\frac{2}{m+1}}S_{m+1}^{\frac{m}{m+1}} \leq \tilde{\eta}S_{m+1}+c\,\tilde{\eta}^{-1}\cgam^{2(m+1)}\Ao{2},
\]
where $\tilde{\eta}= \tilde{\eta}(\eta)>0$ is a new constant obtained from Young's inequality that is allowed to be as small as desired. Therefore
\[
(1-2\eta - \tilde{\eta})S_{m+1} \leq c\,(\eta^{-1}+\tilde{\eta}^{-1})\oo{\Ao{\frac{2}{m+2}}S_{m+2}^{\frac{m+1}{m+2}}+\cgam^{2(m+1)}\Ao{2}},
\]
and so choosing $\eta,\tilde{\eta}>0$ small enough proves the inductive step, and hence \eqref{HigherOrderSobolevLemma1Eqn1}.
\end{proof}

\begin{theorem}\label{HigherOrderSobolevTheorem1}
Let $l\in\mathbb{N},\,m\in\mathbb{N}_{0}$, and $f:\Sigma\rightarrow\mathbb{R}^{3}$ be an immersion and $\gamma$ a cut-off function as in \eqref{GammaProp0}. Then if $\Ao{2}\leq\varepsilon_{0}$ for $\varepsilon_{0}>0$ sufficiently small, there exists an absolute constant $c>0$ such that
\begin{equation}
S_{l} \leq c\oo{\Ao{\frac{2m}{l+m}}S_{l+m}^{\frac{l}{l+m}} + \cgam^{2l}\Ao{2}}.\label{HigherOrderSobolevTheorem1Eqn1}
\end{equation}
\end{theorem}
\begin{proof}
The proof is completed by combining induction and the results of Lemma \ref{HigherOrderSobolevLemma1}. The result is trivial for $m=0$. Moreover, for any $l\in\mathbb{N}$, the statement is true for $m=1$ by Lemma \ref{HigherOrderSobolevLemma1}. Assume that for any $l\in\mathbb{N}$ the statement is true for some $m=k\in\mathbb{N}_{0}$. That is, assume that for every $l\in\mathbb{N}$:
\begin{equation}
S_{l} \leq c\oo{\Ao{\frac{2k}{l+k}}S_{l+k}^{\frac{l}{l+k}} + \cgam^{2l}\Ao{2}}.\label{HigherOrderSobolevTheorem1Eqn2}
\end{equation}
Next, applying \eqref{HigherOrderSobolevLemma1Eqn1} with $l+k$ instead of $l$ gives
\begin{equation}
S_{l+k} \leq c\oo{\Ao{\frac{2}{l+k+1}}S_{l+k+1}^{\frac{l+k}{l+k+1}}+\cgam^{2(l+k)}\Ao{2}}\label{HigherOrderSobolevTheorem1Eqn3}
\end{equation}
Combining \eqref{HigherOrderSobolevTheorem1Eqn2} and \eqref{HigherOrderSobolevTheorem1Eqn3} then yields, for any $l\in\mathbb{N}$:
\begin{align}
S_{l} &\leq c \Ao{\frac{2k}{l+k}}\oo{\Ao{\frac{2}{l+k+1}}S_{l+k+1}^{\frac{l+k}{l+k+1}}+\cgam^{2(l+k)}\Ao{2}}^{\frac{l}{l+k}} + c\,\cgam^{2l}\Ao{2}\nonumber\\
&\leq c\oo{\Ao{\frac{2(k+1)}{l+k+1}}S_{l+k+1}^{\frac{l}{l+k+1}}+\cgam^{2l}\Ao{2}}, 
\end{align}
which is inequality \eqref{HigherOrderSobolevTheorem1Eqn1} with $m=k+1\in\mathbb{N}$. Therefore, by mathematical induction, inequality \eqref{HigherOrderSobolevTheorem1Eqn1} holds for all $l\in\mathbb{N},\,m\in\mathbb{N}_{0}$.
\end{proof}

\section{Proof of the local tracefree curvature estimate and the Gap Lemma}
\label{S:localest}

We dedicate this section to the proof of Theorems \ref{EpsilonRegularity} and
\ref{GapTheorem}. Both are consequences of the local tracefree curvature
estimate of \cite{mccoy2015geometric} together with the higher-order
interpolation inequality proved above.

\begin{proof}[Proof of Theorem \ref{EpsilonRegularity}]
Let $\gamma$ be a cut-off function satisfying \eqref{GammaProp0}.  By
\cite[Theorem 28]{mccoy2015geometric}, if
\[
    \int_{[\gamma>0]} |A^o|^2\,d\mu \leq \varepsilon_0,
\]
then
\begin{equation}
\llll{A^{o}}_{\infty,\cc{\gamma=1}}^{6}
\leq
c\llll{A^{o}}_{2,\cc{\gamma>0}}^{4}
\oo{
\llll{\nabla\Delta H}_{2,\gamma^{6}}^{2}
+
\cgam^{6}\llll{A^{o}}_{2,\cc{\gamma>0}}^{2}
}.
\label{EpsilonRegularityProof0}
\end{equation}
For $m\geq3$, apply Theorem \ref{HigherOrderSobolevTheorem1} with
$l=3$ and with $m-3$ in place of the interpolation parameter.  Recalling
that
\[
S_j=\int_\Sigma \left|\Delta^{\frac{j}{2}}H\right|^2\gamma^{2j}\,d\mu,
\]
we obtain
\[
\llll{\nabla\Delta H}_{2,\gamma^6}^{2}
=
S_3
\leq
c\oo{
\Ao{\frac{2(m-3)}{m}}
S_m^{\frac{3}{m}}
+
\cgam^6\Ao{2}
}.
\]
Substituting this estimate into \eqref{EpsilonRegularityProof0} gives
\[
\llll{A^{o}}_{\infty,\cc{\gamma=1}}^{6}
\leq
c\oo{
\Ao{\frac{6m-6}{m}}S_m^{\frac{3}{m}}
+
\cgam^6\Ao{6}
}.
\]
Since $m\geq3$, raising both sides to the power $m/3$ and using
\((a+b)^{m/3}\leq c(m)(a^{m/3}+b^{m/3})\) yields
\begin{equation}
\llll{A^{o}}_{\infty,\cc{\gamma=1}}^{2m}
\leq
c\Ao{2m-2}
\oo{
S_m+\cgam^{2m}\Ao{2}
}.
\label{EpsilonRegularityProof1}
\end{equation}
Equivalently,
\[
\llll{A^{o}}_{\infty,\cc{\gamma=1}}^{2m}
\leq
c\llll{A^{o}}_{2,\cc{\gamma>0}}^{2m-2}
\oo{
\llll{\Delta^{\frac{m}{2}}H}_{2,\gamma^{2m}}^{2}
+
\cgam^{2m}\llll{A^{o}}_{2,\cc{\gamma>0}}^{2}
}.
\]

Now choose $\gamma=\tilde{\gamma}\circ f$, where $\tilde\gamma$ is a
standard cut-off satisfying
\[
    \chi_{B_{\rho}\oo{x}}\leq\tilde{\gamma}\leq\chi_{B_{2\rho}\oo{x}},
    \qquad
    |D^j\tilde\gamma|\leq c_j\rho^{-j}.
\]
Then $\cgam\leq c\rho^{-1}$, $[\gamma=1]$ contains
$f^{-1}\oo{B_\rho\oo{x}}$, and $[\gamma>0]$ is contained in
$f^{-1}\oo{B_{2\rho}\oo{x}}$.  Therefore
\[
\llll{A^{o}}_{\infty,f^{-1}\oo{B_{\rho}\oo{x}}}^{2m}
\leq
c\,\llll{A^{o}}_{2,f^{-1}\oo{B_{2\rho}\oo{x}}}^{2m-2}
\oo{
\llll{\Delta^{\frac{m}{2}}H}_{2,f^{-1}\oo{B_{2\rho}\oo{x}}}^{2}
+
\rho^{-2m}\llll{A^{o}}_{2,f^{-1}\oo{B_{2\rho}\oo{x}}}^{2}
},
\]
which proves the first assertion.

For the ``in particular'' statement, assume $\Delta^pH\equiv c_0$.
Apply the estimate just proved with
\[
    m=2p+1.
\]
Then
\[
    \Delta^{\frac{2p+1}{2}}H
    =
    \nabla\Delta^pH
    \equiv 0.
\]
Hence
\[
\llll{A^{o}}_{\infty,f^{-1}\oo{B_{\rho}\oo{x}}}^{4p+2}
\leq
c\oo{p}
\rho^{-(4p+2)}
\llll{A^{o}}_{2,f^{-1}\oo{B_{2\rho}\oo{x}}}^{4p+2}.
\]
Taking the $(4p+2)$-th root gives
\[
\llll{A^{o}}_{\infty,f^{-1}\oo{B_{\rho}\oo{x}}}
\leq
c\oo{p}\rho^{-1}
\llll{A^{o}}_{2,f^{-1}\oo{B_{2\rho}\oo{x}}}.
\]
Using \eqref{EpsilonRegularity1} gives the stated
$c(p)\rho^{-1}\varepsilon_0^{1/2}$ bound.
\end{proof}

We are now ready to prove the Gap Lemma, Theorem \ref{GapTheorem}.

\begin{proof}[Proof of Theorem \ref{GapTheorem}]
Fix an arbitrary point $q\in\Sigma$.  For every $\rho>0$, apply the
``in particular'' part of Theorem \ref{EpsilonRegularity} with
\[
    x=f(q).
\]
The global smallness assumption \eqref{Smallness} implies
\[
    \C^o(f(q),2\rho)
    \leq
    \int_\Sigma |A^o|^2\,d\mu
    \leq
    \varepsilon_0,
\]
and so the theorem gives
\[
    |A^o|(q)
    \leq
    c(p)\rho^{-1}
    \llll{A^o}_{2,f^{-1}\oo{B_{2\rho}\oo{f(q)}}}
    \leq
    c(p)\rho^{-1}\llll{A^o}_{2}.
\]
Letting $\rho\nearrow\infty$ gives $|A^o|(q)=0$.  Since $q\in\Sigma$
was arbitrary, $A^o\equiv0$ on $\Sigma$.

Thus $f$ is totally umbilic.  In an orthonormal frame,
\[
\norm{A^{o}}^{2}
=
\frac{1}{2}\oo{\kappa_{1}-\kappa_{2}}^{2},
\]
so the two principal curvatures agree identically.  By the Codazzi
equations, the common principal curvature is constant.  Hence the image
of $f$ lies either in an affine plane or in a round sphere.  Since $f$ is
proper and $\Sigma$ is connected, the image is respectively a standard
flat plane or a standard round sphere.
\end{proof}

\section{Exponential improvement of the energy condition}

In this section we show that evolving any immersion with small initial umbilic energy via \eqref{PolyharmonicIntro1} drives the umbilic energy down towards zero exponentially fast.
Recall that this energy gives a measure of the `average' distance from a round sphere in $L^{2}$. We will first state a result of Li and Yau \cite{LiYau1}. This result will later allow us, under a similar small energy condition, to establish that $f\oo{\Sigma,t},t\in\left[0,T\right)$ is a one-parameter family of embeddings. 
\begin{theorem}[Li-Yau $\text{\cite[Theorem 6]{LiYau1}}$]\label{LiYauTheorem1}
If an immersion $f:\Sigma\rightarrow\mathbb{R}^{3}$ has the property that
\begin{equation}
\frac{1}{4}\intM{H^{2}}<8\, \pi,\label{LiYauTheorem1,1}
\end{equation}
then $f$ is an embedding.
\end{theorem}

\begin{lemma}\label{LemmaPreservedSphericityEstimates2}
Let $f:\Sigma\rightarrow\R^{3}$ be a closed immersion satisfying $\Aoc{2} \leq \varepsilon_{0}$ for $\varepsilon_{0}>0$ sufficiently small. Then there exists a universal constant $c>0$ such that the following estimate holds:
\begin{equation}
\intM{\nablanorm{p}{A^{o}}{2}}\leq c\intM{|\Delta^{\frac{p}{2}}H|^{2}}\leq c\Aoc{\frac{2}{p+1}}\oo{\intM{|\Delta^{\frac{p+1}{2}}H|^{2}}}^{\frac{p}{p+1}},\label{LemmaPreservedSphericityEstimates2,Eqn1}
\end{equation}
\end{lemma}
\begin{proof}
If $p=1$, then applying \eqref{MultliplicativeSobolevLemma2Eqn1} with
$\gamma\equiv1$ gives
\[
\intM{\norm{\nabla A^o}^{2}}
\leq
c\intM{\norm{\nabla H}^{2}}.
\]
Moreover, by Theorem \ref{HigherOrderSobolevTheorem1} with $l=1$ and
$m=1$, again with $\gamma\equiv1$,
\[
\intM{\norm{\nabla H}^{2}}
\leq
c\Aoc{}\oo{\intM{\norm{\Delta H}^{2}}}^{\frac{1}{2}}.
\]
Thus
\[
\intM{\norm{\nabla A^o}^{2}}
\leq
c\intM{\norm{\nabla H}^{2}}
\leq
c\Aoc{}\oo{\intM{\norm{\Delta H}^{2}}}^{\frac{1}{2}},
\]
which is precisely \eqref{LemmaPreservedSphericityEstimates2,Eqn1}
when $p=1$.

If $p=2$, we use the endpoint closed tracefree elliptic estimate
\begin{equation}
\intM{\nablanorm{2}{A^o}{2}}
+
\intM{\nablanorm{1}{A^o}{2}H^2}
+
\intM{\norm{P_4^2(A^o)}}
\leq
C\intM{|\Delta H|^2}.
\label{LemmaPreservedSphericityEstimates2,p2endpoint}
\end{equation}
This follows from the tracefree Codazzi identity
$\operatorname{div}A^o=\frac12\nabla H$, the Simons identity for $A^o$,
and the multiplicative Sobolev inequality; the terms containing
$A^o$ are absorbed using $\|A^o\|_2^2\leq\varepsilon_0$.  Consequently
\[
\intM{\nablanorm{2}{A^o}{2}}
\leq
C\intM{|\Delta H|^2}.
\]
The second inequality in \eqref{LemmaPreservedSphericityEstimates2,Eqn1}
for $p=2$ is the closed interpolation estimate
\[
\intM{|\Delta H|^2}
\leq
C\Aoc{\frac23}\left(\intM{|\nabla\Delta H|^2}\right)^{\frac23},
\]
which is Theorem \ref{HigherOrderSobolevTheorem1} with $l=2$ and
$m=1$.  Thus the lemma holds when $p=2$.

We therefore assume for the remainder of the proof that $p\geq3$.

Now, by multiple applications the interchange of covariant formula and the fact that in dimensions $n=2$, $R_{ijk}^{l}=\frac{1}{2}R\oo{g_{ik}\delta_{j}^{l}-g_{jk}\delta_{i}^{l}}$, we have
\begin{equation}
\Delta\nabla_{\oo{k}}T=\nabla_{\oo{k}}\Delta T+\sum_{i=0}^{k}\nabla_{\oo{i}}\oo{R*\nabla_{\oo{k-i}}T}.\label{LemmaPreservedSphericityEstimates2,Eqn3}
\end{equation} 
Using integration by parts, identity \eqref{LemmaPreservedSphericityEstimates2,Eqn3} with $T=A^{o},k=p-1$, as well as identity \eqref{MultliplicativeSobolevLemma2Eqn4}, it follows that
\begin{align}
&\intM{\nablanorm{p}{A^{o}}{2}}=-\intM{\inner{\nabla_{\oo{p-1}}A^{o},\Delta\nabla_{\oo{p-1}}A^{o}}}\\
&\quad=-\intM{\inner{\nabla_{\oo{p-1}}A^{o},\nabla_{\oo{p-1}}\Delta A^{o}}}+\sum_{i=0}^{p-1}\intM{\nabla_{\oo{p-1}}A^{o}*\nabla_{\oo{i}}\oo{R*\nabla_{\oo{p-1-i}}A^{o}}}\\
&\quad=-\intM{\inner{\nabla_{\oo{p-1}}A^{o},\nabla_{(p+1)}H}}+\sum_{i=0}^{p-1}\intM{\nabla_{\oo{p-1}}A^{o}*\nabla_{\oo{i}}\oo{R*\nabla_{\oo{p-1-i}}A^{o}}}\\
&\quad=\intM{\inner{\Delta\nabla_{\oo{p-2}}A^{o},\nabla_{\oo{p}}H}}+\sum_{i=0}^{p-1}\intM{\nabla_{\oo{p-1}}A^{o}*\nabla_{\oo{i}}\oo{R*\nabla_{\oo{p-1-i}}A^{o}}},
\end{align}
which implies that
\begin{equation}
\intM{\nablanorm{p}{A^{o}}{2}}\leq c\intM{\nablanorm{p}{H}{2}}+\sum_{i=0}^{p-1}\intM{\nabla_{\oo{p-1}}A^{o}*\nabla_{\oo{i}}\oo{R*\nabla_{\oo{p-1-i}}A^{o}}}\label{LemmaPreservedSphericityEstimates2,Eqn4}
\end{equation}
for some universal constant $c>0$. By following a process similar to that above, we have
\begin{multline}
\intM{\nablanorm{p}{H}{2}}=\intM{|\Delta^{\frac{p}{2}}H|^{2}}+\sum_{i=0}^{p-1}\intM{\nabla_{\oo{p-1}}H*\nabla_{\oo{i}}\oo{R*\nabla_{\oo{p-1-i}}H}}\\
+\sum_{i=0}^{p-2}\intM{\nabla_{\oo{p}}H*\nabla_{\oo{i}}\oo{R*\nabla_{\oo{p-2-i}}H}}.
\end{multline}
Combining the last two identities gives
\begin{multline}
\intM{\nablanorm{p}{A^{o}}{2}} \leq c\intM{|\Delta^{\frac{p}{2}}H|^{2}}+\sum_{i=0}^{p-1}\intM{\nabla_{\oo{p-1}}A^{o}*\nabla_{\oo{i}}\oo{R*\nabla_{\oo{p-1-i}}A^{o}}}\\
+\sum_{i=0}^{p-1}\intM{\nabla_{\oo{p-1}}H*\nabla_{\oo{i}}\oo{R*\nabla_{\oo{p-1-i}}H}}+\sum_{i=0}^{p-2}\intM{\nabla_{\oo{p}}H*\nabla_{\oo{i}}\oo{R*\nabla_{\oo{p-2-i}}H}}. \label{LemmaPreservedSphericityEstimates2,Eqn5}
\end{multline}
We need to estimate the extraneous terms on the right hand side of \eqref{LemmaPreservedSphericityEstimates2,Eqn5}. For the first summation we use the identity $R=\frac{1}{2}H^{2}-\norm{A^{o}}^{2}$ to establish the following estimate:
\begin{align}
&\sum_{i=0}^{p-1}\intM{\nabla_{\oo{p-1}}A^{o}*\nabla_{\oo{i}}\oo{R*\nabla_{\oo{p-1-i}}A^{o}}}\\
&\quad\leq\sum_{i=0}^{p-2}\intM{\nabla_{\oo{p-1}}A^{o}*\nabla_{\oo{i}}\oo{H^{2}\,\nabla_{\oo{p-1-i}}A^{o}}}+\intM{\norm{P_{4}^{2\oo{p-1}}\oo{A^{o}}}}\\
&\quad=\intM{\oo{\nabla_{\oo{p-1}}A^{o}}^{\#2}\,H^{2}}+\sum_{i=1}^{p-2}\intM{H\,\nabla_{\oo{p-1}}A^{o}*\nabla_{\oo{i}}H*\nabla_{\oo{p-1-i}}A^{o}}+\intM{\norm{P_{4}^{2\oo{p-1}}\oo{A^{o}}}}\\
&\quad\leq c\intM{\nablanorm{p-1}{A^{o}}{2}H^{2}}+\sum_{i=1}^{p-2}\intM{\nablanorm{p-1}{A^{o}}{}\nablanorm{p-1-i}{A^{o}}{}\nablanorm{i}{A^{o}}{}\norm{H}}+\intM{\norm{P_{4}^{2\oo{p-1}}\oo{A^{o}}}}\\
&\quad\leq c\intM{\nablanorm{p-1}{A^{o}}{2}H^{2}}+\intM{\norm{P_{4}^{2\oo{p-1}}\oo{A^{o}}}}\\
&\qquad+\sum_{i=1}^{p-2}\oo{\intM{\nablanorm{p-1}{A^{o}}{2}H^{2}}\cdot\intM{\nablanorm{p-1-i}{A^{o}}{2}\nablanorm{i}{A^{o}}{2}}}^{\frac{1}{2}}\\
&\quad\leq c\intM{\nablanorm{p-1}{A^{o}}{2}H^{2}}+\intM{\norm{P_{4}^{2\oo{p-1}}\oo{A^{o}}}}.
\end{align}
Here we have used the Cauchy-Schwarz inequality in the last step along with the fact that
\[
\intM{\nablanorm{p-1-i}{A^{o}}{2}\nablanorm{i}{A^{o}}{2}}\leq \intM{\norm{P_{4}^{2\oo{p-1}}\oo{A^{o}}}}.
\]
The other two summations in the right hand side of \eqref{LemmaPreservedSphericityEstimates2,Eqn5} are estimated in the same way. 
Therefore
\begin{equation}
\intM{\nablanorm{p}{A^{o}}{2}}\leq c\intM{|\Delta^{\frac{p}{2}}H|^{2}}+c\intM{\nablanorm{p-1}{A^{o}}{2}H^{2}}+\intM{\norm{P_{4}^{2\oo{p-1}}\oo{A^{o}}}}.\label{LemmaPreservedSphericityEstimates2,Eqn6}
\end{equation}
The last term can be estimated using the multiplicative Sobolev inequality (Proposition \ref{EvolutionProposition4}),  interpolation inequalities (Theorem \ref{HigherOrderSobolevTheorem1}) and Lemma \ref{NewLemma1}, as well as our estimate for the trace-free curvature, \eqref{EpsilonRegularityProof1}:
\begin{align}
\intM{\norm{P_{4}^{2\oo{p-1}}\oo{A^{o}}}}&\leq c\Aoi{2}\intM{\nablanorm{p-1}{A^{o}}{2}}\\
&\leq c\oo{\Aoc{\frac{2(p-1)}{p}}\oo{\intM{|\Delta^{\frac{p}{2}}H|^{2}}}^{\frac{1}{p}}}\cdot\oo{\Aoc{\frac{2}{p}}\oo{\intM{\nablanorm{p}{A^{o}}{2}}}^{\frac{p-1}{p}}}\\
&\leq c\Aoc{2}\oo{\intM{|\Delta^{\frac{p}{2}}H|^{2}}}^{\frac{1}{p}}\oo{\intM{\nablanorm{p}{A^{o}}{2}}}^{\frac{p-1}{p}}\\
&\leq c\Aoc{2}\oo{\intM{|\Delta^{\frac{p}{2}}H|^{2}}+\intM{\nablanorm{p}{A^{o}}{2}}}.\label{LemmaPreservedSphericityEstimates2,Eqn7}
\end{align} 
Here in the last step we have used Young's inequality with conjugate exponents $\alpha=p/\oo{p-1}$, $\alpha^{*}=p$.
For the penultimate term in \eqref{LemmaPreservedSphericityEstimates2,Eqn6}, we combine the results of Theorem \ref{TheoremXX} and Lemma \ref{AppendxLemma1} to give
\begin{multline}
\intM{\nablanorm{p-1}{A^{o}}{2}H^{2}}\leq c\intM{\norm{\nabla A^{o}}^{2}H^{2\oo{p-1}}}+\intM{\norm{A^{o}}^{2}H^{2p}}+c\intM{|\Delta^{\frac{p}{2}}H|^{2}}\\
\leq c\intM{\norm{\Delta H}^{2}H^{2\oo{p-2}}}+c\intM{|\Delta^{\frac{p}{2}}H|^{2}}
\leq c\intM{|\Delta^{\frac{p}{2}}H|^{2}}.\label{LemmaPreservedSphericityEstimates2,Eqn8}
\end{multline}
Substituting \eqref{LemmaPreservedSphericityEstimates2,Eqn7} and \eqref{LemmaPreservedSphericityEstimates2,Eqn8} in \eqref{LemmaPreservedSphericityEstimates2,Eqn6} then implies
\begin{equation}
\oo{1-c\Aoc{2}}\cdot\intM{\nablanorm{p}{A^{o}}{2}}\leq c\cdot\oo{1+\Aoc{2}}\intM{|\Delta^{\frac{p}{2}}H|^{2}}.
\end{equation}
Therefore if $\varepsilon_{0}>0$ is small enough we obtain
\begin{equation}
\intM{\nablanorm{p}{A^{o}}{2}}\leq c\intM{|\Delta^{\frac{p}{2}}H|^{2}}\leq c\Aoc{\frac{2}{p+1}}\oo{\intM{|\Delta^{\frac{p+1}{2}}H|^{2}}}^{\frac{p}{p+1}},\label{LemmaPreservedSphericityEstimates2,Eqn9}
\end{equation}
where we have used the interpolation inequality of Theorem \ref{HigherOrderSobolevTheorem1} once more in the last step.
\end{proof}

\begin{lemma}\label{LemmaPreservedSphericityEstimates1}
Let $f:\Sigma\rightarrow\R^{3}$ be a closed immersion satisfying $\Aoc{2} \leq \varepsilon_{0}$ for $\varepsilon_{0}>0$ sufficiently small. Then there exists a universal constant $c>0$ such that the following estimate holds:
\begin{equation}
\intM{\norm{\Delta^{\frac{p-1}{2}}\oo{H\norm{A^{o}}^{2}}}^{2}}\leq c\Aoc{2}\intM{|\Delta^{\frac{p+1}{2}}H|^{2}}.\label{LemmaPreservedSphericityEstimates1,Eqn1}
\end{equation}
\end{lemma}
\begin{proof}
We first treat the endpoint case $p=1$.  In this case the left hand side
of \eqref{LemmaPreservedSphericityEstimates1,Eqn1} is
\[
I:=\intM{\norm{H}^{2}\norm{A^o}^{4}}.
\]
Applying the Michael--Simon Sobolev inequality, Theorem
\ref{MichaelSimon}, to $u=\norm{H}\norm{A^o}^{2}$ yields
\begin{align}
I
&\leq
c\oo{
\intM{\norm{\nabla H}\norm{A^o}^{2}}
+
\intM{\norm{H}\norm{A^o}\norm{\nabla A^o}}
+
\intM{\norm{H}^{2}\norm{A^o}^{2}}
}^{2}.
\label{LemmaPreservedSphericityEstimates1,p1MS}
\end{align}
The three terms on the right are estimated using
\eqref{MultliplicativeSobolevLemma2Eqn1} and
\eqref{MultliplicativeSobolevLemma2Eqn3}, again with $\gamma\equiv1$.
Indeed,
\[
\intM{\norm{\nabla H}\norm{A^o}^{2}}
\leq
\oo{\intM{\norm{\nabla H}^{2}}}^{\frac{1}{2}}
\oo{\intM{\norm{A^o}^{4}}}^{\frac{1}{2}}
\leq
c\Aoc{}\intM{\norm{\nabla H}^{2}},
\]
while
\[
\intM{\norm{H}\norm{A^o}\norm{\nabla A^o}}
\leq
\oo{\intM{\norm{H}^{2}\norm{A^o}^{2}}}^{\frac{1}{2}}
\oo{\intM{\norm{\nabla A^o}^{2}}}^{\frac{1}{2}}
\leq
c\intM{\norm{\nabla H}^{2}},
\]
and
\[
\intM{\norm{H}^{2}\norm{A^o}^{2}}
\leq
c\intM{\norm{\nabla H}^{2}}.
\]
Therefore, since $\Aoc{2}$ is small,
\[
I\leq c\oo{\intM{\norm{\nabla H}^{2}}}^{2}.
\]
Finally, Theorem \ref{HigherOrderSobolevTheorem1} with $l=1$ and
$m=1$ gives
\[
\intM{\norm{\nabla H}^{2}}
\leq
c\Aoc{}\oo{\intM{\norm{\Delta H}^{2}}}^{\frac{1}{2}}.
\]
Hence
\[
\intM{\norm{H}^{2}\norm{A^o}^{4}}
\leq
c\Aoc{2}\intM{\norm{\Delta H}^{2}},
\]
which proves \eqref{LemmaPreservedSphericityEstimates1,Eqn1} when
$p=1$.

If $p=2$, then
\[
\Delta^{\frac{p-1}{2}}(H|A^o|^2)=\nabla(H|A^o|^2),
\]
and hence
\begin{align*}
\intM{|\nabla(H|A^o|^2)|^2}
&\leq
C\intM{|\nabla H|^2|A^o|^4}
+
C\intM{H^2|A^o|^2|\nabla A^o|^2}  \\
&\leq
C\|A^o\|_\infty^4\intM{|\nabla H|^2}
+
C\|A^o\|_\infty^2\intM{|\nabla A^o|^2H^2}.
\end{align*}
Applying Theorem \ref{EpsilonRegularity} with $m=3$ globally and using
Theorem \ref{HigherOrderSobolevTheorem1} gives
\[
\|A^o\|_\infty^2
\leq
C\Aoc{\frac43}\left(\intM{|\nabla\Delta H|^2}\right)^{\frac13},
\quad
\intM{|\nabla H|^2}
\leq
C\Aoc{\frac43}\left(\intM{|\nabla\Delta H|^2}\right)^{\frac13},
\]
and
\[
\intM{|\Delta H|^2}
\leq
C\Aoc{\frac23}\left(\intM{|\nabla\Delta H|^2}\right)^{\frac23}.
\]
Together with \eqref{LemmaPreservedSphericityEstimates2,p2endpoint},
these estimates imply
\[
\intM{|\nabla(H|A^o|^2)|^2}
\leq
C\Aoc{2}\intM{|\nabla\Delta H|^2},
\]
after decreasing $\varepsilon_0$ if necessary.  This is precisely
\eqref{LemmaPreservedSphericityEstimates1,Eqn1} for $p=2$.

We therefore assume for the remainder of the proof that $p\geq3$.
By the identity \eqref{Prelimiaries2} we have
\begin{equation}
\Delta^{\frac{p-1}{2}}\oo{H\norm{A^{o}}^{2}}
\leq
\Delta^{\frac{p-1}{2}}\norm{A^{o}}^{2}\cdot H
+
\norm{P_{3}^{p-1}\oo{A^{o}}}
\end{equation}
and therefore by the multiplicative Sobolev inequality of Proposition \ref{EvolutionProposition4} one has
\begin{align}
\intM{\norm{\Delta^{\frac{p-1}{2}}\oo{H\norm{A^{o}}^{2}}}^{2}}&\leq c\intM{\norm{\Delta^{\frac{p-1}{2}}\norm{A^{o}}^{2}}^{2}H^{2}}+\intM{\norm{P_{6}^{2\oo{p-1}}\oo{A^{o}}}}\\
&\leq c\intM{\norm{\Delta^{\frac{p-1}{2}}\norm{A^{o}}^{2}}^{2}H^{2}}+c\Aoi{4}\intM{\nablanorm{p-1}{A^{o}}{2}}.\label{LemmaPreservedSphericityEstimates1,Eqn2}
\end{align}
Next we estimate the first term on the right hand side of \eqref{LemmaPreservedSphericityEstimates1,Eqn2}. We treat the cases $p$ is even and $p$ is odd separately.  We use the following inequalities, which hold for $\varphi\in C^{3}$ under the small-energy assumption \eqref{Smallness}:
\begin{equation}
\intM{\nablanorm{2}{\varphi}{2}} + \intM{\norm{\nabla\varphi}^{2}H^{2}} \leq c \intM{\norm{\Delta \varphi}^{2}} + c\intM{\norm{\nabla A^{o}}^{2}}\intM{\norm{\nabla\varphi}^{2}}.\label{LemmaPreservedSphericityEstimates1,Eqn3}
\end{equation}
and
\begin{multline}
\intM{\nablanorm{3}{\varphi}{2}} + \intM{\nablanorm{2}{\varphi}{2}H^{2}}+\intM{\norm{\nabla\varphi}^{2}H^{4}}\\
\leq c\intM{\norm{\nabla\Delta\varphi}^{2}}+c\intM{\norm{\nabla H}^{2}}\cdot\intM{\norm{\Delta\varphi}^{2}}+c\intM{\norm{\Delta H}^{2}}\cdot\intM{\norm{\nabla\varphi}^{2}}.\label{LemmaPreservedSphericityEstimates1,Eqn4}
\end{multline}
The inequalities \eqref{LemmaPreservedSphericityEstimates1,Eqn3} and \eqref{LemmaPreservedSphericityEstimates1,Eqn4} are proven at the end of the appendix.  If $p$ is even with $p=2\oo{k+1},\,k\geq1$, then using \eqref{LemmaPreservedSphericityEstimates1,Eqn3} with $\varphi=\Delta^{k}\norm{A^{o}}^{2}$, followed by the multiplicative Sobolev inequality of Proposition \ref{EvolutionProposition4} gives
\begin{align}
\intM{\norm{\Delta^{\frac{p-1}{2}}\norm{A^{o}}^{2}}^{2}H^{2}}&=\intM{\norm{\nabla\varphi}^{2}H^{2}}\\
&\leq c\intM{\norm{\Delta\varphi}^{2}}+c\intM{\norm{\nabla A^{o}}^{2}}\cdot\intM{\norm{\nabla\varphi}^{2}}\\
&=c\intM{\norm{\Delta^{k+1}\norm{A^{o}}^{2}}^{2}}+c\intM{\norm{\nabla A^{o}}^{2}}\cdot\intM{\norm{\nabla\Delta^{k}\norm{A^{o}}^{2}}^{2}}\\
&\leq \intM{\norm{P_{4}^{4\oo{k+1}}\oo{A^{o}}}}+\intM{\norm{\nabla A^{o}}^{2}}\cdot\intM{\norm{P_{4}^{2\oo{2k+1}}\oo{A^{o}}}}\\
&\leq c\Aoi{2}\oo{\intM{\nablanorm{p}{A^{o}}{2}}+\intM{\norm{\nabla A^{o}}}^{2}\cdot\intM{\nablanorm{p-1}{A^{o}}{2}}}.
\end{align}
Using \eqref{MultliplicativeSobolevLemma2Eqn1} to estimate $\llll{\nabla A^{o}}_{2}^{2}$, along with the interpolation inequalities from Lemma \ref{NewLemma1} and Theorem \ref{HigherOrderSobolevTheorem1} as well as Young's inequality with conjugate exponents $\alpha=p/\oo{p-1},\alpha^{*}=p$, the last term on the right can be estimated as follows:

Using \eqref{MultliplicativeSobolevLemma2Eqn1} to estimate $\llll{\nabla A^{o}}_{2}^{2}$, the results of Lemma \ref{LemmaPreservedSphericityEstimates2}, along with the interpolation inequalities from Lemma \ref{NewLemma1} yields
\begin{align}
&\intM{\norm{\nabla A^{o}}}^{2}\cdot\intM{\nablanorm{p-1}{A^{o}}{2}}\\
&\qquad\leq c\oo{\Aoc{\frac{2p}{p+1}}\oo{\intM{|\Delta^{\frac{p+1}{2}}H|^{2}}}^{\frac{1}{p+1}}} \cdot \oo{\Aoc{\frac{2}{p}}\oo{\intM{\nablanorm{p}{A^{o}}{2}}}^{\frac{p-1}{p}}}\\
&\qquad\leq c\Aoc{\frac{2p}{p+1}+\frac{2}{p}}\oo{\intM{|\Delta^{\frac{p+1}{2}}H|^{2}}}^{\frac{1}{p+1}}\cdot \oo{\Aoc{\frac{2}{p+1}}\oo{\intM{|\Delta^{\frac{p+1}{2}}H|^{2}}}^{\frac{p}{p+1}}}^{\frac{p-1}{p}}\\
&\qquad= c\Aoc{\frac{2(p+2)}{p+1}}\oo{\intM{|\Delta^{\frac{p+1}{2}}H|^{2}}}^{\frac{p}{p+1}}.\label{LemmaPreservedSphericityEstimates1,Eqn5}
\end{align}
Combining this with the estimates of Lemma \ref{LemmaPreservedSphericityEstimates2} gives
Therefore
\begin{align}
\intM{\norm{\Delta^{\frac{p-1}{2}}\norm{A^{o}}^{2}}^{2}H^{2}}&\leq c\Aoi{2}\oo{\Aoc{\frac{2}{p+1}}+\Aoc{\frac{2(p+2)}{p+1}}}\cdot\oo{\intM{|\Delta^{\frac{p+1}{2}}H|^{2}}}^{\frac{p}{p+1}}\\
&\leq c\Aoi{2}\Aoc{\frac{2}{p+1}}\oo{\intM{|\Delta^{\frac{p+1}{2}}H|^{2}}}^{\frac{p}{p+1}},\label{LemmaPreservedSphericityEstimates1,Eqn6}
\end{align}
where we have used the results of Lemma \ref{LemmaPreservedSphericityEstimates2} to estimate $||\nabla_{(p)}A^{o}||_{2}^{2}$ and also used the fact that $p\geq1$ which implies that $\Aoc{\frac{2(p+2)}{p+1}} \leq c \Aoc{\frac{2}{p+1}}$.

Now suppose $p$ is odd.  Since the case $p=1$ has already been handled,
we may write $p=2k+1$ with $k\geq1$.  Applying
\eqref{LemmaPreservedSphericityEstimates1,Eqn4} with
$\varphi=\Delta^{k-1}\norm{A^{o}}^{2}$, followed by the multiplicative
Sobolev inequality of Proposition \ref{EvolutionProposition4}, gives
\begin{align}
&\intM{\norm{\Delta^{\frac{p-1}{2}}\norm{A^{o}}^{2}}^{2}H^{2}}
\leq
c\intM{\nablanorm{2}{\varphi}{2}H^{2}}
\nonumber\\
&\quad\leq
c\intM{\norm{\nabla\Delta\varphi}^{2}}
+
c\intM{\norm{\nabla H}^{2}}\cdot\intM{\norm{\Delta\varphi}^{2}}
+
c\intM{\norm{\Delta H}^{2}}\cdot\intM{\norm{\nabla\varphi}^{2}}
\nonumber\\
&\quad=
c\intM{\norm{\nabla\Delta^{k}\norm{A^{o}}^{2}}^{2}}
+
c\intM{\norm{\nabla H}^{2}}\cdot
  \intM{\norm{\Delta^{k}\norm{A^{o}}^{2}}^{2}}
+
c\intM{\norm{\Delta H}^{2}}\cdot
  \intM{\norm{\nabla\Delta^{k-1}\norm{A^{o}}^{2}}^{2}}
\nonumber\\
&\quad\leq
c\intM{\norm{P_{4}^{2p}\oo{A^{o}}}}
+
c\intM{\norm{\nabla H}^{2}}\cdot
  \intM{\norm{P_{4}^{2\oo{p-1}}\oo{A^{o}}}}
+
c\intM{\norm{\Delta H}^{2}}\cdot
  \intM{\norm{P_{4}^{2\oo{p-2}}\oo{A^{o}}}}
\nonumber\\
&\quad\leq
c\Aoi{2}\\
& \qquad  \cdot \Biggl(
\intM{\nablanorm{p}{A^{o}}{2}}
+
\intM{\norm{\nabla H}^{2}}\cdot
  \intM{\nablanorm{p-1}{A^{o}}{2}}
+
\intM{\norm{\Delta H}^{2}}\cdot
  \intM{\nablanorm{p-2}{A^{o}}{2}}
\Biggr).
\label{LemmaPreservedSphericityEstimates1,OddCaseIntermediate}
\end{align}
The last two products are estimated in the same way as
\eqref{LemmaPreservedSphericityEstimates1,Eqn5}, using Lemma
\ref{LemmaPreservedSphericityEstimates2} and the interpolation inequality
of Theorem \ref{HigherOrderSobolevTheorem1}.  Hence
\begin{equation}
\intM{\norm{\Delta^{\frac{p-1}{2}}\norm{A^{o}}^{2}}^{2}H^{2}}
\leq
c\Aoi{2}\Aoc{\frac{2}{p+1}}
\oo{\intM{|\Delta^{\frac{p+1}{2}}H|^{2}}}^{\frac{p}{p+1}}
\label{LemmaPreservedSphericityEstimates1,OddCase}
\end{equation}
whenever $p$ is odd and $p\geq3$.

Combining the even estimate \eqref{LemmaPreservedSphericityEstimates1,Eqn6}
with the odd estimate \eqref{LemmaPreservedSphericityEstimates1,OddCase},
we obtain, for every $p\geq2$,
\begin{equation}
\intM{\norm{\Delta^{\frac{p-1}{2}}\norm{A^{o}}^{2}}^{2}H^{2}}
\leq
c\Aoi{2}\Aoc{\frac{2}{p+1}}
\oo{\intM{|\Delta^{\frac{p+1}{2}}H|^{2}}}^{\frac{p}{p+1}}.
\label{LemmaPreservedSphericityEstimates1,Eqn7}
\end{equation}

Next, by combining the results of Lemma \ref{LemmaPreservedSphericityEstimates2} along with the interpolation inequalities from Lemma \ref{LemmaPreservedSphericityEstimates2} gives
\begin{equation}
\intM{\nablanorm{p-1}{A^{o}}{2}} \leq c\Aoc{\frac{2}{p}}\oo{\intM{\nablanorm{p}{A^{o}}{2}}}^{\frac{p-1}{p}} \leq c\Aoc{\frac{4}{p+1}}\oo{\intM{|\Delta^{\frac{p+1}{2}}H|^{2}}}^{\frac{p-1}{p+1}}.\label{LemmaPreservedSphericityEstimates1,Eqn8}
\end{equation}
Substituting \eqref{LemmaPreservedSphericityEstimates1,Eqn7} and \eqref{LemmaPreservedSphericityEstimates1,Eqn8} back into \eqref{LemmaPreservedSphericityEstimates1,Eqn2} and using the $L^{\infty}$ estimate for the trace-free curvature from \eqref{EpsilonRegularityProof1} gives
\begin{align*}
&\intM{\norm{\Delta^{\frac{p-1}{2}}\oo{H\norm{A^{o}}^{2}}}^{2}}
\leq c\intM{\norm{\Delta^{\frac{p-1}{2}}\norm{A^{o}}^{2}}^{2}H^{2}}+c\Aoi{4}\intM{\nablanorm{p-1}{A^{o}}{2}}\\
&\leq c\Aoi{2}\Aoc{\frac{2}{p+1}}\oo{\intM{|\Delta^{\frac{p+1}{2}}H|^{2}}}^{\frac{p}{p+1}} + c\Aoi{4}\Aoc{\frac{4}{p+1}}\oo{\intM{|\Delta^{\frac{p+1}{2}}H|^{2}}}^{\frac{p-1}{p+1}}\\
&\leq c\Aoc{\frac{2p}{p+1}+\frac{2}{p+1}}\oo{\intM{|\Delta^{\frac{p+1}{2}}H|^{2}}}^{\frac{1}{p+1} + \frac{p}{p+1}} + c\Aoc{\frac{4p}{p+1}+\frac{4}{p+1}}\oo{\intM{|\Delta^{\frac{p+1}{2}}H|^{2}}}^{\frac{2}{p+1}+\frac{p-1}{p+1}}\\
&\leq c \Aoc{2}\oo{1+\Aoc{2}}\intM{|\Delta^{\frac{p+1}{2}}H|^{2}}
\leq c\Aoc{2}\intM{|\Delta^{\frac{p+1}{2}}H|^{2}},
\end{align*}
which completes the proof.
\end{proof}

We are able to show that if initially small, the total tracefree curvature decreases under \eqref{PolyharmonicIntro1}. We refer to this as ``preserved almost-sphericity''.

\begin{theorem}[Preserved almost-sphericity]\label{T:preservealmostsphere}
Let $f:\Sigma\times[0,T)\rightarrow\mathbb{R}^3$ satisfy
\eqref{PolyharmonicIntro1}, where $\Sigma$ is connected and closed.  There exists an $\varepsilon_0>0$, depending
only on $p$, such that if
\begin{equation}
\intM{|A^o|^2}\Big|_{t=0}\leq\varepsilon_0<8\pi,
\label{GlobalEnergy1}
\end{equation}
then, for every $t<T$,
\begin{equation}
\frac{d}{dt}\intM{|A^o|^2}
\leq
-\frac{1}{2}\intM{|\Delta^{\frac{p+1}{2}}H|^2}.
\label{Preservedsmallnesstheorem1,00}
\end{equation}
In particular, by Theorem \ref{LiYauTheorem1} (and Gauss-Bonnet), each
$\Sigma_t=f(\Sigma,t)$ is embedded.
\end{theorem}

\begin{proof}
Let
\[
    I:=\left\{t\in[0,T):\intM{|A^o|^2}(\tau)\leq2\varepsilon_0
    \text{ for all }\tau\in[0,t]\right\}.
\]
By continuity, $I$ is a nonempty interval containing $0$.  We prove the
estimate on $I$ and then close the bootstrap.

Using Gauss--Bonnet and $|A^o|^2=\frac{1}{2}H^2-2K$, together with the
evolution equation for $H$, gives
\begin{equation}
\frac{d}{dt}\intM{|A^o|^2}
=\frac{1}{2}\frac{d}{dt}\intM{H^2}
=(-1)^p\intM{H\Delta^{p+1}H}
+(-1)^p\intM{H\Delta^pH\,|A^o|^2}.
\end{equation}
Integrating by parts $p+1$ times in the first integral and $p-1$ times
in the second gives
\begin{equation}
\frac{d}{dt}\intM{|A^o|^2}
+\intM{|\Delta^{\frac{p+1}{2}}H|^2}
\leq
C\left(
\intM{|\Delta^{\frac{p+1}{2}}H|^2}
\intM{\left|\Delta^{\frac{p-1}{2}}(H|A^o|^2)\right|^2}
\right)^{1/2}.
\label{Preservedsmallnesstheorem1,1}
\end{equation}
By Lemma \ref{LemmaPreservedSphericityEstimates1},
\[
\intM{\left|\Delta^{\frac{p-1}{2}}(H|A^o|^2)\right|^2}
\leq
C\left(\intM{|A^o|^2}\right)
\intM{|\Delta^{\frac{p+1}{2}}H|^2}.
\]
Substituting this into \eqref{Preservedsmallnesstheorem1,1} gives, on
$I$,
\[
\frac{d}{dt}\intM{|A^o|^2}
+\left(1-C\sqrt{2\varepsilon_0}\right)
\intM{|\Delta^{\frac{p+1}{2}}H|^2}
\leq0.
\]
Choose $\varepsilon_0$ so small that $C\sqrt{2\varepsilon_0}\leq1/2$.
Then \eqref{Preservedsmallnesstheorem1,00} holds on $I$, and in
particular $\int |A^o|^2$ is nonincreasing on $I$.  Therefore
$\int |A^o|^2\leq\varepsilon_0$ on $I$, so the bootstrap interval is
both open and closed in $[0,T)$.  Hence $I=[0,T)$ and the estimate holds
for all $t<T$.

Since $\int|A^o|^2<8\pi$, the identity
$\int |A^o|^2=2W-4\pi\chi(\Sigma)$ together with the Willmore
inequality forces $\Sigma$ to have genus zero.  Hence
$W=\frac14\int H^2=4\pi+\frac12\int |A^o|^2<8\pi$, and Theorem
\ref{LiYauTheorem1} implies that the immersion is embedded at every time.
\end{proof}

\begin{lemma}[Global tracefree coercivity]\label{L:global-tracefree-coercivity}
Let $f:\Sigma\to\mathbb R^3$ be a connected
closed embedded immersion satisfying
\[
    \int_\Sigma |A^o|^2\,d\mu\leq\varepsilon_0
\]
with $\varepsilon_0$ sufficiently small.  Then
\begin{equation}
\int_\Sigma |A^o|^2\,d\mu
\leq
C |\Sigma|^{p+1}
\int_\Sigma |\Delta^{\frac{p+1}{2}}H|^2\,d\mu,
\label{E:coercive-tracefree}
\end{equation}
where $C=C(p)$.
\end{lemma}

\begin{proof}
It is enough to prove the estimate after normalising $|\Sigma|=1$.
For sufficiently small $\varepsilon_0$, Gauss--Bonnet, the Willmore
inequality and the Li--Yau inequality imply that $\Sigma$ is an embedded
topological sphere.  We first prove the first-order estimate
\begin{equation}
    \int_\Sigma |A^o|^2\,d\mu
    \leq C\int_\Sigma |\nabla H|^2\,d\mu .
\label{E:first-order-tracefree-coercivity}
\end{equation}
Indeed, on a sphere there are no nontrivial tracefree Codazzi tensors.
Equivalently, the operator
\[
    B\longmapsto \operatorname{div} B
\]
from tracefree symmetric two-tensors to one-forms has trivial kernel on
the tracefree Codazzi subspace.  The standard elliptic estimate for this
first-order operator therefore gives
\[
    \|B\|_{L^2}\leq C\|\operatorname{div}B\|_{L^2}
\]
for tracefree tensors on the round sphere.  

The required elliptic estimate for tracefree symmetric two-tensors is standard:
on a two-sphere, the operator
\[
    \operatorname{div}: \Gamma(S^2_0T^*\Sigma)
    \longrightarrow \Gamma(T^*\Sigma)
\]
is elliptic with trivial kernel.  Equivalently, if \(B\) is tracefree and symmetric,
then
\[
    \|B\|_{L^2} \leq C\|\operatorname{div}B\|_{L^2}.
\]
See Christodoulou \cite[Section 5.4, Lemma 5.6]{ChristodoulouBlackHoles};
the vanishing of the kernel on \(S^2\) is also the classical fact that there are no
nontrivial tracefree Codazzi tensors on the sphere, see
\cite[Corollary 3.7]{FoxEinsteinLikeSurfaces}.

The same estimate holds with a
uniform constant for metrics induced by embedded almost-umbilic spheres
with $\|A^o\|_2$ sufficiently small.  One way to see the uniformity is by
contradiction: if the constants degenerated, the quantitative
almost-umbilical compactness theorem of De Lellis--M\"uller
\cite{DeLellisMueller2005} would, after
translation, scaling, and reparametrisation, give convergence to a round
sphere, while the normalised tensors would converge weakly to a nonzero
tracefree divergence-free tensor on $S^2$, equivalently a nonzero
holomorphic quadratic differential; this is impossible on $S^2$.
Applying the estimate to $B=A^o$ and using
$\operatorname{div}A^o=\frac12\nabla H$ proves
\eqref{E:first-order-tracefree-coercivity}.

It remains to pass from $\nabla H$ to the top-order quantity.  On the
area-normalised surface, the closed-surface interpolation inequalities and
the smallness of $\|A^o\|_2$ give
\[
    \int_\Sigma |\nabla H|^2\,d\mu
    \leq
    C\int_\Sigma |\Delta^{\frac{p+1}{2}}H|^2\,d\mu .
\]
Rescaling back to arbitrary area gives the factor $|\Sigma|^{p+1}$ and
hence \eqref{E:coercive-tracefree}.
\end{proof}

\begin{theorem}[Exponential decay of umbilic energy]\label{T:expdecayenergy}
Let $f:\Sigma\times[0,T)\rightarrow\mathbb{R}^3$ satisfy
\eqref{PolyharmonicIntro1}, where $\Sigma$ is connected and closed.  After decreasing the constant
$\varepsilon_0$ in Theorem \ref{T:preservealmostsphere}, if
\eqref{GlobalEnergy1} holds, then there is a constant
$\delta>0$, depending only on $p$ and on the initial area, such that, for every $t<T$, 
\begin{equation}
\intM{|A^o|^2}(t)
\leq
\left.\intM{|A^o|^2}\right|_{t=0}e^{-\delta t}
\label{TheoremExponentialDecayEqn1}
\end{equation}
\end{theorem}

\begin{proof}
By Theorem \ref{T:preservealmostsphere},
\begin{equation}
\frac{d}{dt}\intM{|A^o|^2}
\leq
-\frac{1}{2}\intM{|\Delta^{\frac{p+1}{2}}H|^2}.
\label{TheoremExponentialDecayEqn2}
\end{equation}
By Lemma \ref{L:global-tracefree-coercivity}, and using that area is
nonincreasing along the flow,
\[
\intM{|A^o|^2}
\leq
C |\Sigma_0|^{p+1}
\intM{|\Delta^{\frac{p+1}{2}}H|^2}.
\]
Combining
\eqref{TheoremExponentialDecayEqn2} and \eqref{E:coercive-tracefree}
yields
\[
\frac{d}{dt}\intM{|A^o|^2}
+\frac{1}{C|\Sigma_0|^{p+1}}\intM{|A^o|^2}
\leq0.
\]
Gronwall's inequality proves \eqref{TheoremExponentialDecayEqn1} with
$\delta=(C|\Sigma_0|^{p+1})^{-1}$.
\end{proof}

\section{Proof of the Lifespan Theorem}
\label{S:lifespan}

\begin{lemma}[Finite-order continuation under small concentration]\label{L:finite-order-continuation}
Let
\[
    K_{\rm cont}:=2p+6.
\]
There is a constant $\varepsilon_{\rm cont}>0$, depending only on $p$, with the following property.  Let $f:\Sigma\times[0,T)\to\mathbb R^3$ be a maximal smooth solution on a compact surface, with $T<\infty$.  If, for some $\rho>0$,
\begin{equation}
    \sup_{0\leq t<T}\C(t,2\rho)\leq\varepsilon_{\rm cont},
\label{E:finite-cont-smallness}
\end{equation}
then the solution extends smoothly past $T$.
\end{lemma}

\begin{proof}
By scaling we take $\rho=1$.  For every $y\in\mathbb R^3$ choose a standard cut-off $\gamma_y$ which is identically equal to $1$ on $B_1(y)$ and supported in $B_2(y)$, with cut-off constants independent of $y$.  We choose $\varepsilon_{\rm cont}$ no larger than the fixed-order thresholds in Proposition \ref{NewProposition1} for $k=0,\ldots,K_{\rm cont}$, for this cut-off profile.  Then \eqref{E:finite-cont-smallness} and Proposition \ref{NewProposition1} give for $k=0,\ldots,K_{\rm cont}$,
\[
    \sup_{0\leq t<T}\sup_\Sigma |\nabla_{(k)}A|(t)
    \leq C_k(T,f_0)<\infty,
\]
Indeed, the constants are uniform in $y$, because
\[
    \sup_{y\in\mathbb R^3}\sum_{j=0}^{K_{\rm cont}+2}
    \left.\int_{f_0^{-1}(B_2(y))}|\nabla_{(j)}A|^2\,d\mu\right|_{t=0}<\infty
\]
by compactness of the initial surface.

The boundedness of $A$ gives uniform equivalence of the metrics on every finite time interval ending at $T$, and the bounds for $\nabla_{(k)}A$ control the coordinate derivatives of the metric, Christoffel symbols, and immersion up to order $2p+4$.  Since the normal velocity
\[
    \partial_t f=(-1)^{p+1}\Delta^pH\,\nu
\]
is controlled by derivatives of $A$ of order at most $2p$, the maps $f(\cdot,t)$ converge as $t\nearrow T$ in a sufficiently high $C^{2p+4,\alpha}$ topology to an immersion $f_T$.  The standard short-time existence and continuation theorem for the DeTurck-gauged form of \eqref{PolyharmonicIntro1} applies to $f_T$ and extends the solution past $T$; parabolic smoothing then gives a smooth extension.  This contradicts maximality.
\end{proof}

\begin{proof}[Proof of Theorem \ref{LifespanTheorem}]
We prove the theorem by the usual localisation and continuity argument.
The quantity $\int |A|^2$ is scale-invariant on surfaces, and the flow is
invariant under the parabolic rescaling
\[
    \tilde f(x,t)=\rho^{-1}f(x,\rho^{2(p+1)}t).
\]
It is therefore enough to prove the assertion for $\rho=1$.

For $t<T$ define
\[
\Kappa(t):=\sup_{x\in\mathbb{R}^3}
\int_{f^{-1}(B_1(x))}|A|^2\,d\mu .
\]
Let $\varepsilon_{\rm evol}$ be the smallness constant in Proposition
\ref{PN:estA}.  Let $C_\Kappa$ be the finite covering constant for
covering a ball of radius $2$ by finitely many balls of radius $1$.
Let $\varepsilon_{\rm cont}$ be the threshold from Lemma
\ref{L:finite-order-continuation}.  Choose $\varepsilon_{\mathrm{life}}>0$ so small that
\[
    3C_\Kappa\varepsilon_{\mathrm{life}}
    \leq \varepsilon_{\rm evol},
    \qquad
    2C_\Kappa^2\varepsilon_{\mathrm{life}}
    \leq \varepsilon_{\rm cont}.
\]
Now fix $0<\varepsilon\leq\varepsilon_{\mathrm{life}}$ and assume
$\Kappa(0)\leq\varepsilon$.
Set
\[
    t_0:=\sup\left\{t<T:\Kappa(\tau)\leq3C_\Kappa\varepsilon
    \text{ for every }\tau\in[0,t]\right\}.
\]
By continuity, $t_0>0$.  For any $x\in\mathbb{R}^3$ choose a cut-off
$\gamma_x$ with $\gamma_x\equiv1$ on $B_1(x)$, supported in $B_2(x)$,
and with $c_{\gamma_x}\leq C$.  The finite covering argument and the
definition of $t_0$ imply that, for $t\leq t_0$,
\[
    \int_{[\gamma_x>0]}|A|^2\,d\mu
    \leq 3C_\Kappa\varepsilon
    \leq \varepsilon_{\rm evol}.
\]
Proposition \ref{PN:estA} therefore gives
\[
\int_{f^{-1}(B_1(x))}|A|^2\,d\mu(t)
\leq
\int_{f^{-1}(B_2(x))}|A|^2\,d\mu(0)+C\varepsilon t
\leq
C_\Kappa\varepsilon+C\varepsilon t .
\]
Taking the supremum over $x$ yields
\begin{equation}
    \Kappa(t)\leq C_\Kappa\varepsilon+C\varepsilon t
    \qquad\text{for }0\leq t\leq t_0.
\label{E:lifespan-kappa-bound}
\end{equation}
Choose
\[
    \lambda:=\frac{C_\Kappa}{C}.
\]
Then \eqref{E:lifespan-kappa-bound} implies
$\Kappa(t)\leq2C_\Kappa\varepsilon$ for all
$t\leq\min\{t_0,\lambda\}$.  Hence the defining inequality for $t_0$
cannot fail before time $\lambda$.

It remains to exclude the endpoint possibility $T<\lambda$.  In that case $t_0=T$ and the preceding bound holds for every $t<T$.  Covering every ball of radius $2$ by at most $C_\Kappa$ balls of radius $1$ gives
\[
    \C(t,2)\leq C_\Kappa\Kappa(t)
    \leq 2C_\Kappa^2\varepsilon
    \leq \varepsilon_{\rm cont}
    \qquad\text{for all }t<T.
\]
Lemma \ref{L:finite-order-continuation} then extends the solution past $T$, contradicting maximality.  Therefore $T\geq\lambda$.

Returning to a general radius $\rho>0$ gives
\[
    T\geq \lambda\rho^{2(p+1)}.
\]
Moreover, for every
$t\leq\min\{T,\lambda\rho^{2(p+1)}\}$ the same argument gives
\[
\sup_{x\in\mathbb{R}^3}
\int_{f^{-1}(B_\rho(x))}|A|^2\,d\mu(t)
\leq C\varepsilon,
\]
after absorbing the covering constants into $C$.  This proves
\eqref{lifespan1} and \eqref{lifespan2}, with
$C_{\mathrm{life}}$ chosen larger than both $\lambda^{-1}$ and the last
constant.
\end{proof}

\section{Construction of the Blowup and Global Analysis of the Flow}
\label{S:blowupglobal}

We now use the lifespan theorem and the interior estimates to prove the
global existence alternative and the convergence statement for initially
almost-umbilic data.

We begin by stating and proving Theorem \ref{T:interiorests}. Here we assume a bound on the concentration of curvature in the ball $f^{-1}(B_{2\rho}(x))$ on some time interval that is a proper subset of $[0,T)$. We use this assumption to establish a universal bound on the $L^{\infty}-$norm of any derivative of curvature on the smaller ball $f^{-1}(B_{\rho}(x))$.

\begin{theorem}[Interior Estimates]\label{T:interiorests}
Let $m\in\mathbb N_0$.  There exist constants
$\varepsilon_m>0$, $c_0>0$, and $c_m<\infty$, depending only on
$m$ and $p$, with the following property.  Suppose
$f:\Sigma\times[0,T)\to\mathbb R^3$ satisfies
\eqref{PolyharmonicIntro1}.  Let $x_0\in\mathbb R^3$, $\rho>0$, and
$0<T^*<T$ be such that
\begin{equation}
\sup_{0<t\leq T^*}
\int_{f^{-1}\oo{B_{2\rho}\oo{x_0}}}|A|^2\,d\mu
\leq \varepsilon_m,
\qquad
T^*\leq c_0\rho^{2\oo{p+1}} .
\label{InteriorEstimatesEqn1}
\end{equation}
Then, for every $t\in(0,T^*]$,
\begin{equation}
\llll{\nabla_{\oo{m}}A}_{2,f^{-1}\oo{B_{\rho}\oo{x_0}}}
\leq
c_m\sqrt{\varepsilon_m}\,
t^{-\frac{m}{2\oo{p+1}}},
\label{InteriorEstimatesEqn01}
\end{equation}
and
\begin{equation}
\llll{\nabla_{\oo{m}}A}_{\infty,f^{-1}\oo{B_{\rho}\oo{x_0}}}
\leq
c_m\sqrt{\varepsilon_m}\,
t^{-\frac{m+1}{2\oo{p+1}}}.
\label{InteriorEstimatesEqn02}
\end{equation}
\end{theorem}

\begin{proof}
We first prove the local $L^2$-estimate.  Fix an integer
$N\in\mathbb N_0$; later we take $N=m$.  Let
$T_0\in(0,T^*]$ be arbitrary.  Choose a cutoff
$\gamma=\tilde\gamma\circ f$ such that
\[
    0\leq \gamma\leq 1,\qquad
    \gamma\equiv 1 \text{ on } f^{-1}\oo{B_\rho\oo{x_0}},
    \qquad
    [\gamma>0]\subset f^{-1}\oo{B_{2\rho}\oo{x_0}},
\]
and
\[
    \cgam\leq C\rho^{-1}.
\]
The constant $C$ here is universal and depends only on the fixed cutoff
profile.

For $j\in\{0,\ldots,N\}$ set
\[
    s_j:=2(j+p+1),
\]
and define
\[
    E_j(t):=\llll{\nabla_{(j)}A}_{2,\gamma^{s_j}}^2,
    \qquad
    D_j(t):=\llll{\nabla_{(j+p+1)}A}_{2,\gamma^{s_j}}^2 .
\]
Since $[\gamma>0]\subset f^{-1}\oo{B_{2\rho}\oo{x_0}}$, the hypothesis
\eqref{InteriorEstimatesEqn1} implies
\[
    E_0(t)\leq \varepsilon_m
    \qquad\text{for }0<t\leq T_0.
\]

We choose $\varepsilon_m$ no larger than the smallness constant in
Proposition \ref{EvolutionProposition1}, and also so small that the
coefficient of $D_j$ in Proposition \ref{EvolutionProposition1} is
bounded below by a positive constant for every $j=0,\ldots,N+2$.  Applying
Proposition \ref{EvolutionProposition1} with $k=j$, $s=s_j$, and with
$\delta$ fixed, we obtain
\begin{equation}
    E_j'(t)+D_j(t)
    \leq
    C_N\cgam^{2(j+p+1)}\varepsilon_m
    \leq
    C_N\varepsilon_m\rho^{-2(j+p+1)}
\label{InteriorEstimatesL2Evolution}
\end{equation}
for $j=0,\ldots,N+2$.  The constant $C_N$ depends only on $N$ and $p$.

Next we record the interpolation step needed in the induction.  For
$j\geq1$, Lemma \ref{NewLemma1}, applied to
$T=\nabla_{(j-1)}A$ with $k=1$ and interpolation gap $p+1$, gives
\begin{multline}
E_j
=
\llll{\nabla_{(j)}A}_{2,\gamma^{s_j}}^2
\leq
C_N
\oo{
\llll{\nabla_{(j-1)}A}_{2,\gamma^{s_j-2}}^2
}^{\frac{p+1}{p+2}}
\oo{
\llll{\nabla_{(j+p+1)}A}_{2,\gamma^{s_j+2(p+1)}}^2
}^{\frac{1}{p+2}}\\
+
C_N\cgam^2
\llll{\nabla_{(j-1)}A}_{2,\gamma^{s_j-2}}^2
\leq
C_N E_{j-1}^{\frac{p+1}{p+2}}D_j^{\frac{1}{p+2}}
+
C_N\rho^{-2}E_{j-1}.
\label{InteriorEstimatesInterpolation}
\end{multline}
Here we used $s_j-2=s_{j-1}$ and $0\leq\gamma\leq1$.

Set
\[
    a:=\frac{p+1}{p+2},
    \qquad
    b:=\frac{1}{p+2},
    \qquad
    \lambda:=\frac{p+2}{p+1}.
\]
For $0\leq t\leq T_0$ define
\[
    \beta(t):=\frac{t}{T_0},
    \qquad
    \beta_j(t):=\beta(t)^{\lambda j},
    \qquad
    F_j(t):=\beta_j(t)E_j(t),
\]
with the convention $\beta_0\equiv1$.  Then, for $j\geq1$,
\[
    0\leq\beta_j'(t)
    \leq
    C_NT_0^{-1}\beta_{j-1}(t)^a\beta_j(t)^b,
    \qquad
    0\leq\beta_j'(t)
    \leq
    C_NT_0^{-1}\beta_{j-1}(t).
\]
Indeed, this is precisely why the exponent $\lambda=(p+2)/(p+1)$ is
chosen.

We claim that, for $j=0,\ldots,N$ and $0\leq t\leq T_0$,
\begin{equation}
    F_j(t)
    \leq
    C_N\varepsilon_m T_0^{-\frac{j}{p+1}}.
\label{InteriorEstimatesEqn3}
\end{equation}
For $j=0$ this follows immediately from the concentration hypothesis.
Assume that \eqref{InteriorEstimatesEqn3} holds for $j-1$.  Using
\eqref{InteriorEstimatesL2Evolution},
\eqref{InteriorEstimatesInterpolation}, and the estimates for
$\beta_j'$, we obtain
\begin{multline*}
F_j'
=
\beta_j'E_j+\beta_jE_j'
\leq
\beta_j'E_j-\beta_jD_j
+
C_N\beta_j\varepsilon_m\rho^{-2(j+p+1)}
\\
\leq
C_NT_0^{-1}
F_{j-1}^{a}
\oo{\beta_jD_j}^{b}
+
C_NT_0^{-1}\rho^{-2}F_{j-1}
-\beta_jD_j
+
C_N\varepsilon_m\rho^{-2(j+p+1)}.
\end{multline*}
Young's inequality gives
\[
    C_NT_0^{-1}
    F_{j-1}^{a}
    \oo{\beta_jD_j}^{b}
    \leq
    \frac12\beta_jD_j
    +
    C_NT_0^{-\frac{p+2}{p+1}}F_{j-1}.
\]
Moreover, since $T_0\leq T^*\leq c_0\rho^{2(p+1)}$,
\[
    T_0^{-1}\rho^{-2}
    \leq
    C T_0^{-\frac{p+2}{p+1}},
    \qquad
    \rho^{-2(j+p+1)}
    \leq
    C T_0^{-\frac{j+p+1}{p+1}}.
\]
Therefore
\[
F_j'
\leq
C_NT_0^{-\frac{p+2}{p+1}}F_{j-1}
+
C_N\varepsilon_mT_0^{-\frac{j+p+1}{p+1}}.
\]
Using the induction hypothesis,
\[
F_j'
\leq
C_N\varepsilon_mT_0^{-\frac{j+p+1}{p+1}}.
\]
Since $\beta_j(0)=0$ for $j\geq1$, we have $F_j(0)=0$, and hence
\[
F_j(t)
\leq
C_N\varepsilon_mT_0^{-\frac{j+p+1}{p+1}}t
\leq
C_N\varepsilon_mT_0^{-\frac{j}{p+1}}.
\]
This proves \eqref{InteriorEstimatesEqn3}.

Taking $j=N$ and $t=T_0$, we have $\beta_N(T_0)=1$, so
\[
    E_N(T_0)
    \leq
    C_N\varepsilon_mT_0^{-\frac{N}{p+1}}.
\]
Since $\gamma\equiv1$ on $f^{-1}\oo{B_\rho\oo{x_0}}$,
\[
\llll{\nabla_{(N)}A}_{2,f^{-1}\oo{B_\rho\oo{x_0}}}^2
\leq
C_N\varepsilon_mT_0^{-\frac{N}{p+1}}.
\]
As $T_0\in(0,T^*]$ was arbitrary, this proves
\eqref{InteriorEstimatesEqn01} with $N=m$.

It remains to prove the $L^\infty$-estimate.  Fix
$T_0\in(0,T^*]$.  Choose a second cutoff $\hat\gamma$ such that
\[
    \hat\gamma\equiv1
    \text{ on } f^{-1}\oo{B_\rho\oo{x_0}},
    \qquad
    [\hat\gamma>0]\subset f^{-1}\oo{B_{\frac32\rho}\oo{x_0}},
    \qquad
    \hat\cgam\leq C\rho^{-1}.
\]
Repeating the $L^2$ argument above with a cutoff which is identically one
on $f^{-1}\oo{B_{\frac32\rho}\oo{x_0}}$ and supported in
$f^{-1}\oo{B_{2\rho}\oo{x_0}}$ gives
\[
\llll{\nabla_{(m+2)}A}_{2,f^{-1}\oo{B_{\frac32\rho}\oo{x_0}}}^2
\leq
C_m\varepsilon_mT_0^{-\frac{m+2}{p+1}}.
\]
Hence
\[
\llll{\nabla_{(m+2)}A}_{2,\hat\gamma^{2(m+2)}}^2
\leq
C_m\varepsilon_mT_0^{-\frac{m+2}{p+1}}.
\]
Applying Proposition \ref{Chapter3NewProp2} with the cutoff $\hat\gamma$
and using
\[
    \int_{[\hat\gamma>0]}|A|^2\,d\mu\leq\varepsilon_m,
    \qquad
    \hat\cgam\leq C\rho^{-1},
\]
we find
\begin{align*}
\llll{\nabla_{(m)}A}_{\infty,f^{-1}\oo{B_\rho\oo{x_0}}}^{2(m+2)}
&\leq
C_m\varepsilon_m
\left(
\llll{\nabla_{(m+2)}A}_{2,\hat\gamma^{2(m+2)}}^{2(m+1)}
+
\rho^{-2(m+1)(m+2)}\varepsilon_m^{m+1}
\right)
\\
&\leq
C_m\varepsilon_m
\left(
\oo{\varepsilon_mT_0^{-\frac{m+2}{p+1}}}^{m+1}
+
\rho^{-2(m+1)(m+2)}\varepsilon_m^{m+1}
\right).
\end{align*}
Since $T_0\leq c_0\rho^{2(p+1)}$,
\[
    \rho^{-2(m+1)(m+2)}
    \leq
    C_mT_0^{-\frac{(m+1)(m+2)}{p+1}}.
\]
Therefore
\[
\llll{\nabla_{(m)}A}_{\infty,f^{-1}\oo{B_\rho\oo{x_0}}}^{2(m+2)}
\leq
C_m\varepsilon_m^{m+2}
T_0^{-\frac{(m+1)(m+2)}{p+1}}.
\]
Taking the $2(m+2)$-th root gives
\[
\llll{\nabla_{(m)}A}_{\infty,f^{-1}\oo{B_\rho\oo{x_0}}}
\leq
C_m\sqrt{\varepsilon_m}\,
T_0^{-\frac{m+1}{2(p+1)}}.
\]
Again $T_0\in(0,T^*]$ was arbitrary, so this proves
\eqref{InteriorEstimatesEqn02}.
\end{proof}

\begin{theorem}[Concentration alternative]\label{T:concentration-alternative}
Let $f:\Sigma\times[0,T)\rightarrow\mathbb{R}^3$ be a maximal solution of
\eqref{PolyharmonicIntro1} with $T<\infty$.  Then there is a constant
$\varepsilon_0>0$, depending only on $p$, such that for every
$\rho>0$ with $\rho^{2(p+1)}\leq T$,
\[
    \limsup_{t\nearrow T}\sup_{x\in\mathbb{R}^3}
    \int_{f^{-1}(B_\rho(x))}|A|^2\,d\mu
    \geq \varepsilon_0 .
\]
\end{theorem}

\begin{proof}
This is the contrapositive of Theorem \ref{LifespanTheorem}.  If the
limsup were smaller than the lifespan threshold for some $\rho$, then
starting the flow at a time sufficiently close to $T$ would extend the
solution for a further definite time comparable to $\rho^{2(p+1)}$,
contradicting maximality of $T$.
\end{proof}

\begin{lemma}[Scale-invariant local area bound]\label{L:local-area-bound}
Under the hypotheses of Theorem \ref{T:globalalmostsphere}, after decreasing the smallness threshold if necessary, there is a constant $C<\infty$ depending only on the initial data such that
\begin{equation}
    \frac{\mu_t\oo{f_t^{-1}(B_\sigma(y))}}{\sigma^2}
    \leq C
\label{E:local-area-bound}
\end{equation}
for every $t<T$, every $y\in\mathbb R^3$, and every $\sigma>0$.
Consequently the same estimate holds for every parabolic rescaling of the flow, with the same constant.
\end{lemma}

\begin{proof}
The small tracefree-curvature hypothesis, Gauss--Bonnet and the Willmore inequality imply that $\Sigma$ is a topological sphere, and Theorem \ref{T:preservealmostsphere} preserves the smallness of $\int |A^o|^2$.  Hence
\[
    \int_\Sigma H^2\,d\mu
    =16\pi+2\int_\Sigma |A^o|^2\,d\mu
    \leq C.
\]
Simon's local area estimate, in the form used by Kuwert--Sch\"atzle \cite[Lemma 4.1]{Kuwert1}, gives for a proper immersed surface
\[
    \frac{\mu\oo{f^{-1}(B_\sigma(y))}}{\sigma^2}
    \leq
    C\left(
        \frac{\mu\oo{f^{-1}(B_R(y))}}{R^2}
        +\int_{f^{-1}(B_R(y))}H^2\,d\mu
    \right)
\]
for $0<\sigma\leq R$.  Letting $R\rightarrow\infty$ on the closed surface proves \eqref{E:local-area-bound}.  The estimate is scale-invariant.
\end{proof}

\begin{theorem}[Global existence for almost spheres]\label{T:globalalmostsphere}
Let $f_0:\Sigma\rightarrow\mathbb{R}^3$ be a smooth immersion of a connected closed surface.  Choose the orientation so that the signed enclosed volume $V_0$ is positive.  If
\[
    \int_\Sigma |A^o|^2\,d\mu\leq\varepsilon_0,
\]
where $\varepsilon_0>0$ is sufficiently small depending only on $p$.
Then the solution of \eqref{PolyharmonicIntro1} with initial data $f_0$
exists smoothly for all time.
\end{theorem}

\begin{proof}
Assume, to the contrary, that the maximal time $T$ is finite.  Choose a number $\varepsilon_*$ smaller than the thresholds in the lifespan theorem, the interior estimates, and the gap theorem.  By the concentration alternative, $\rho^{\varepsilon_*}(t)\to0$ as $t\nearrow T$.  Choose record-low times $t_i\nearrow T$ and set
\[
    r_i:=\rho^{\varepsilon_*}(t_i)\downarrow0,
    \qquad
    \rho^{\varepsilon_*}(t)\geq r_i\quad\text{for }0\leq t\leq t_i .
\]
By continuity of the concentration function in the radius, this gives, for all $t\leq t_i$,
\begin{equation}
    \C(t,r_i)\leq\varepsilon_*.
\label{E:blowup-pretime-small}
\end{equation}
At time $t_i$ the scale $r_i$ is maximal nonconcentrating, so the larger scale $2r_i$ is concentrating.  Hence there are centres $x_i\in\mathbb R^3$ such that
\begin{equation}
    \int_{f^{-1}(B_{2r_i}(x_i))}|A|^2\,d\mu(t_i)
    \geq \frac12\varepsilon_* .
\label{E:blowup-normalisation}
\end{equation}
Define the rescaled flows
\[
    f_i(q,\tau):=r_i^{-1}\big(f(q,t_i+r_i^{2(p+1)}\tau)-x_i\big),
    \qquad
    \tau\in[-r_i^{-2(p+1)}t_i,\,r_i^{-2(p+1)}(T-t_i)).
\]
By \eqref{E:blowup-pretime-small}, the rescaled flows have uniformly small curvature concentration on unit balls for $\tau<0$.  The lifespan theorem applied at time $t_i$ also gives a uniform forward time interval $[0,c_0]$ on which the unit-scale concentration is bounded by a fixed multiple of $\varepsilon_*$.  Choosing $\varepsilon_*$ sufficiently small, Theorem \ref{T:interiorests} gives uniform bounds for all derivatives of curvature on compact subsets of space-time.  The local area bound of Lemma \ref{L:local-area-bound} supplies the area hypothesis in the Langer--Breuning compactness theorem, Theorem \ref{KuwertSchatzleCompactnessTheorem}.  Thus a subsequence converges locally smoothly, after reparametrisation, to a proper solution
\[
    \hat f:\hat\Sigma\times(-\infty,c_0]\rightarrow\mathbb R^3
\]
of \eqref{PolyharmonicIntro1}.  The normalisation \eqref{E:blowup-normalisation} passes to the limit and gives
\begin{equation}
    \int_{\hat f^{-1}(B_2(0))}|\hat A|^2\,d\hat\mu
    \geq \frac12\varepsilon_* .
\label{E:blowup-limit-nonflat}
\end{equation}

We next show that the limit is stationary.  The dissipation estimate in Theorem \ref{T:preservealmostsphere} gives
\[
    \int_0^T\int_\Sigma |\Delta^{\frac{p+1}{2}}H|^2\,d\mu\,dt<\infty.
\]
This spacetime integral is invariant under the above parabolic rescaling.  Therefore, for every compact interval $I\subset(-\infty,c_0]$ and every compact spatial set in the limit,
\[
\int_I\int |\Delta_i^{\frac{p+1}{2}}H_i|^2\,d\mu_i\,d\tau
=\int_{t_i+r_i^{2(p+1)}I}\int |\Delta^{\frac{p+1}{2}}H|^2\,d\mu\,dt
\longrightarrow0,
\]
because the corresponding original time intervals shrink to $T$.  Passing to the smooth limit gives
\[
    \Delta^{\frac{p+1}{2}}\hat H\equiv0.
\]
If $p$ is odd this says $\Delta^{(p+1)/2}\hat H=0$; if $p$ is even it says $\nabla\Delta^{p/2}\hat H=0$.  In either case $\Delta^p\hat H\equiv0$ on each connected component, so the blowup limit is stationary.

The tracefree curvature smallness is scale-invariant and nonincreasing by Theorem \ref{T:preservealmostsphere}; hence every compact subset of the limit satisfies
\[
    \int |\hat A^o|^2\,d\hat\mu\leq\varepsilon_0.
\]
The gap theorem, Theorem \ref{GapTheorem}, implies that each connected component of the limit is either a plane or a round sphere.  The component carrying \eqref{E:blowup-limit-nonflat} cannot be a plane.

It remains to exclude a compact round-sphere component.  The signed enclosed volume is preserved by Lemma \ref{L:volpreserve}, and the area is nonincreasing, so the isoperimetric inequality gives the uniform lower bound
\begin{equation}
    |\Sigma_t|^3\geq36\pi V_0^2>0.
\label{E:global-area-lower-bound}
\end{equation}
If the concentrated component of the rescaled limit were compact, proper smooth convergence gives a large ball $B_R(0)$ and an annular neighbourhood separating that component from infinity.  For all large $i$, the corresponding component of the rescaled surface inside $B_R(0)$ is compact without boundary.  Since the original surface is connected, this component is the whole surface.  Undoing the scaling then gives
\[
    |\Sigma_{t_i}|=r_i^2\,|\Sigma_i(0)|\longrightarrow0,
\]
contradicting \eqref{E:global-area-lower-bound}.  The concentrated component is therefore neither a plane nor a compact round sphere, contradicting the gap theorem.  Thus $T=\infty$.
\end{proof}

\begin{proposition}[Uniform curvature bounds]\label{P:uniform-curvature-bounds}
Under the hypotheses of Theorem \ref{T:globalalmostsphere}, all
covariant derivatives of curvature are uniformly bounded along the flow:
for every $m\in\mathbb N_0$ there is a constant $C_m<\infty$ such that
\[
    \sup_{\Sigma_t}|\nabla_{(m)}A|\leq C_m
    \qquad\text{for all }t\geq0.
\]
\end{proposition}

\begin{proof}
Suppose first that $|A|$ is not uniformly bounded.  Choose times $t_i\to\infty$ so that
\[
    Q_i:=\max_{\Sigma\times[0,t_i]}|A|\longrightarrow\infty,
\]
and choose $p_i\in\Sigma$ with $|A|(p_i,t_i)=Q_i$.  This is the standard Hamilton point-picking choice; it gives
\[
    |A|(\cdot,t)\leq Q_i\qquad\text{for }0\leq t\leq t_i.
\]
Define the curvature-scale rescalings
\[
    f_i(q,\tau):=Q_i\big(f(q,t_i+Q_i^{-2(p+1)}\tau)-f(p_i,t_i)\big),
    \qquad
    \tau\in[-Q_i^{2(p+1)}t_i,0].
\]
Then $|A_i|\leq1$ for $\tau\leq0$ and $|A_i|(p_i,0)=1$.

By the scale-invariant local area bound, $\mu_i(B_R(y))\leq C R^2$ for all $R>0$.  Choosing a small fixed $r_0>0$ gives
\[
    \int_{f_i^{-1}(B_{2r_0}(y))}|A_i|^2\,d\mu_i
    \leq C r_0^2,
\]
which is below the interior-estimate threshold if $r_0$ is sufficiently small.  Theorem \ref{T:interiorests}, applied on balls of radius $r_0$, yields uniform bounds for all curvature derivatives on compact subsets of space-time.  Together with the local area bound and Theorem \ref{KuwertSchatzleCompactnessTheorem}, a subsequence converges locally smoothly to a proper ancient limit with
\[
    |\hat A|(\hat p,0)=1.
\]

The total dissipation on $[0,\infty)$ is finite by Theorem \ref{T:preservealmostsphere}.  Since the rescaled spacetime dissipation is invariant and the original time intervals corresponding to any compact $\tau$-interval drift to infinity, the limit satisfies $\Delta^{(p+1)/2}\hat H\equiv0$, and hence $\Delta^p\hat H\equiv0$ on each component.  The preserved smallness of $\int |A^o|^2$ allows us to apply the gap theorem.  Therefore each component is a plane or a round sphere.  The normalisation $|\hat A|(\hat p,0)=1$ excludes a plane.  A compact round-sphere component is excluded exactly as in the proof of Theorem \ref{T:globalalmostsphere}: proper convergence makes the corresponding approximating component closed, connectedness makes it the whole original surface, and undoing the scaling gives $|\Sigma_{t_i}|=O(Q_i^{-2})\to0$, contradicting \eqref{E:global-area-lower-bound}.  This contradiction proves a uniform bound for $|A|$.

Once $|A|$ is uniformly bounded, the local area bound implies that, on sufficiently small extrinsic balls, the $L^2$-curvature concentration is below the threshold in Theorem \ref{T:interiorests}, uniformly in time.  Applying the interior estimates on a finite cover at each time and using the uniformity of the constants gives uniform bounds for every covariant derivative of $A$.
\end{proof}

\begin{lemma}\label{L:volpreserve}
The signed enclosed volume is preserved along \eqref{PolyharmonicIntro1}.
\end{lemma}

\begin{proof}
With the chosen orientation, the first variation of signed enclosed
volume under a normal velocity $F\nu$ is a constant multiple of
$\int_\Sigma F\,d\mu$.  For \eqref{PolyharmonicIntro1},
$F=(-1)^{p+1}\Delta^pH$, and since $\Sigma$ is closed,
\[
    \int_\Sigma \Delta^pH\,d\mu=0.
\]
Thus the signed enclosed volume is independent of time.
\end{proof}

\begin{theorem}[Smooth convergence]\label{T:smoothconvergence}
Under the hypotheses of Theorem \ref{T:globalalmostsphere}, the flow
exists for all time and converges smoothly, after reparametrisation, to a round sphere enclosing the same signed volume as the initial immersion.
\end{theorem}

\begin{proof}
By Theorem \ref{T:expdecayenergy},
\[
    \int_\Sigma |A^o|^2\,d\mu(t)\leq Ce^{-\delta t}.
\]
The uniform curvature bounds from Proposition \ref{P:uniform-curvature-bounds}, interpolation, and the evolution equations imply exponential decay of all covariant derivatives of $A^o$.  In particular, using $\nabla H=2\operatorname{div}A^o$,
\begin{equation}
    \|\Delta^pH\|_{\infty}\leq Ce^{-\delta t}.
\label{E:velocity-decay}
\end{equation}
Thus the normal velocity is integrable in time:
\[
    \|\partial_t f\|_\infty
    =\|\Delta^pH\|_\infty
    \leq Ce^{-\delta t}.
\]
Consequently the images cannot drift by unbounded translations; for $t_2>t_1$,
\[
    \sup_{q\in\Sigma}|f(q,t_2)-f(q,t_1)|
    \leq C e^{-\delta t_1}.
\]

Every sequence $t_i\rightarrow\infty$ has a subsequence for which $f(\cdot,t_i)$ converges smoothly, after reparametrisation, to a totally umbilic closed surface.  The limit is a round sphere.  The preceding velocity estimate shows that the centres of any such subsequential limiting spheres form a Cauchy family, so the centre is unique.  Since the signed enclosed volume is preserved by Lemma \ref{L:volpreserve}, the radius is also unique.  Therefore the whole flow converges smoothly, after reparametrisation, to the round sphere of radius $(3V_0/4\pi)^{1/3}$ and some fixed centre $x_0\in\mathbb R^3$.
\end{proof}

\section{Appendix: Further Interpolation and Sobolev Inequalities, Selected Estimates}
\label{A:appendix}

\begin{theorem}[Michael-Simon \cite{michael1973sobolev}]\label{MichaelSimon}
Let $\Sigma^{m}\,(m\geq2)$ be a sufficiently smooth immersed submanifold in $\R^{n}$, and $u\in C_{c}^{\infty}\oo{\Sigma}$ be a non-negative smooth function with compact support. Then, for a universal, bounded constant $\cmss>0$ that depends only on $m$,
\[
\intM{u^{m/(m-1)}}\leq\cmss\oo{\intM{\norm{\nabla u}+|\vec{H}|u}}^{m/(m-1)}.     
\] 
Here $\vec{H}$ denotes the mean curvature vector. In particular, for surfaces immersed in Euclidean space, there exists a bounded universal constant $\cmss>0$ such that
\[
\intM{u^{2}} \leq \cmss\oo{\intM{\norm{\nabla u}}+\intM{\norm{u}|H| }}^{2}
\]
\end{theorem}

\begin{theorem}[Kuwert--Sch\"{a}tzle multiplicative Sobolev inequality (\cite{Kuwert2}, Theorem 5.6)]\label{EvolutionTheorem1}
Let $f:\Sigma\rightarrow\mathbb{R}^{3}$ be a smooth immersion. For
$2<p\leq\infty$, $0\leq\beta\leq\infty$ and $0<\alpha\leq1$, where
$\frac{1}{\alpha}=\oo{\frac{1}{2}-\frac{1}{p}}\beta+1$, there is a bounded constant $c>0$
depending only on $p$ and $\beta$ such that for all $u\in C_{c}^{1}\oo{\Sigma}$,
\begin{equation*}
\llll{u}_{\infty}\leq c\llll{u}_{\beta}^{1-\alpha}\oo{\llll{\nabla u}_{p}+\llll{Hu}_{p}}^{\alpha} \mbox{.}\label{EvolutionTheorem1,1}
\end{equation*}
\end{theorem}

The following theorem is a classical result originally due to Codazzi. Recall an immersion $\Sigma$ is said to be \emph{umbilic} at a point $p$ if the trace-free second fundamental form is equal to zero at $p$. If the trace-free second fundamental form is identically zero at all points of $\Sigma$, then $\Sigma$ is said to be an \emph{umbilic immersion}.
\begin{theorem}\label{SphereCodazzi}
Let $f:\Sigma^{2}\rightarrow\mathbb{R}^{3}$ be a connected regular, umbilic immersion. Then $f$ is a plane or sphere.
\end{theorem}

\begin{proposition}[Kuwert--Sch\"{a}tzle product interpolation (\cite{Kuwert2}, Corollary 5.5)]\label{EvolutionProposition4}
Let $0\leq i_{1},\dots,i_{r}\leq k,\sum_{j=1}^{r}i_{j}=2\, k$ and $s\geq2\, k$. Then for any tensor $T$ defined over a smooth immersed surface $f:\Sigma\rightarrow\mathbb{R}^{n}$ we have
\begin{equation*}
\norm{\intM{\nabla_{i_{1}}T*\cdots*\nabla_{i_{r}}T\gamma^{s} }}
\leq c\llll{T}_{\infty,[\gamma>0]}^{r-2}\oo{\intM{\nablanorm{k}{T}{2}\gamma^{s}}+\cgam^{2 k}\llll{T}_{2,\cc{\gamma>0}}^{2}},
\end{equation*}
where $c=c(k,r,s)$.
\end{proposition}

\begin{lemma}[Kuwert--Sch\"{a}tzle derivative interpolation (\cite{Kuwert2}, Corollary 5.3)] \label{EvolutionLemma3}
Let $2\leq p<\infty$, $k,m\in\mathbb{N}$ and $s\geq k\, p$.
Then for any $\delta>0$, there exists a constant $c_{\delta}>0$, depending only on $p,\,s$, $k$ and $\delta$, such that for any smooth tensor $T$,
\begin{align*}
  \oo{\intM{\nablanorm{k}{T}{p}\gamma^{s} }}^{\frac{1}{p}}
\leq\delta\, \cgam^{-1}\oo{\intM{\nablanorm{k+1}{T}{p}\gamma^{s+p} }}^{\frac{1}{p}}+c_{\delta}\, \cgam^{k}\oo{\int_{\cc{\gamma>0}}{\norm{T}^{p}}\gamma^{s-k\, p}\,d\mu}^{\frac{1}{p}} \mbox{.}
\end{align*}
\end{lemma}

\begin{lemma}\label{NewLemma1}
Suppose $T$ is a tensor field, let $k\in\mathbb{N}$ and $s\geq 2k$.
Then there exists a constant $C=C(k,s)$ such that
\begin{equation}
\intM{\nablanorm{k}{T}{2}\gamma^{s}}
\leq
C\oo{\intM{\norm{T}^{2}\gamma^{s-2k}}}^{\frac{1}{k+1}}
 \oo{\intM{\nablanorm{k+1}{T}{2}\gamma^{s+2}}}^{\frac{k}{k+1}}
+ C\,\cgam^{2k}\intM{\norm{T}^{2}\gamma^{s-2k}}.
\label{NewLemma1:1}
\end{equation}
Moreover, for any $k,m\in\mathbb{N}$ and $s\geq2k$,
\begin{multline}
\intM{\nablanorm{k}{T}{2}\gamma^{s}}
\leq
C\oo{\intM{\norm{T}^{2}\gamma^{s-2k}}}^{\frac{m}{k+m}}
 \oo{\intM{\nablanorm{k+m}{T}{2}\gamma^{s+2m}}}^{\frac{k}{k+m}}\\
+ C\,\cgam^{2k}\intM{\norm{T}^{2}\gamma^{s-2k}},
\label{NewLemma1:2}
\end{multline}
where $C=C(k,m,s)$ in \eqref{NewLemma1:2}.
\end{lemma}

\begin{proof}
We first prove the elementary two-step estimate.  If $j\geq1$ and
$\sigma\geq2j$, then integration by parts gives
\begin{align}
\llll{\nabla_{(j)}T}_{2,\gamma^\sigma}^{2}
&\leq
C\llll{\nabla_{(j-1)}T}_{2,\gamma^{\sigma-2}}
 \llll{\nabla_{(j+1)}T}_{2,\gamma^{\sigma+2}}
+C\cgam
\llll{\nabla_{(j-1)}T}_{2,\gamma^{\sigma-2}}
 \llll{\nabla_{(j)}T}_{2,\gamma^\sigma}
\nonumber\\
&\leq
C\llll{\nabla_{(j-1)}T}_{2,\gamma^{\sigma-2}}
 \llll{\nabla_{(j+1)}T}_{2,\gamma^{\sigma+2}}
+\frac12\llll{\nabla_{(j)}T}_{2,\gamma^\sigma}^{2}
+C\cgam^2\llll{\nabla_{(j-1)}T}_{2,\gamma^{\sigma-2}}^{2}.
\end{align}
Iterating Lemma \ref{EvolutionLemma3} on the last term, with the powers
of $\gamma$ decreasing by two at each step, yields
\begin{equation}
\llll{\nabla_{(j)}T}_{2,\gamma^\sigma}^{2}
\leq
C\llll{\nabla_{(j-1)}T}_{2,\gamma^{\sigma-2}}
 \llll{\nabla_{(j+1)}T}_{2,\gamma^{\sigma+2}}
+C\cgam^{2j}\llll{T}_{2,\gamma^{\sigma-2j}}^{2}.
\label{NewLemma1:3}
\end{equation}

Now fix $k$ and put
\[
    X_j:=\llll{\nabla_{(j)}T}_{2,\gamma^{s-2(k-j)}}^{2},
    \qquad j=0,\ldots,k+m .
\]
Thus $X_0=\llll{T}_{2,\gamma^{s-2k}}^2$, $X_k$ is the left hand side of
\eqref{NewLemma1:2}, and
$X_{k+m}=\llll{\nabla_{(k+m)}T}_{2,\gamma^{s+2m}}^2$.
Applying \eqref{NewLemma1:3} with
$\sigma=s-2(k-j)$ gives, for $1\leq j\leq k+m-1$,
\begin{equation}
    X_j\leq C X_{j-1}^{1/2}X_{j+1}^{1/2}
       +C\cgam^{2j}X_0 .
\label{NewLemma1:discrete}
\end{equation}
The standard discrete interpolation lemma obtained by inducting on the
number of intermediate indices in \eqref{NewLemma1:discrete} gives
\begin{equation}
    X_j
    \leq
    C X_0^{1-j/N}X_N^{j/N}+C\cgam^{2j}X_0,
    \qquad 0\leq j\leq N,
\label{NewLemma1:discrete-result}
\end{equation}
whenever $N\geq1$.  Indeed, the error terms are stable under the
induction because the inequality
$\cgam^{2j}\leq C(\cgam^{2(j-1)}\cgam^{2(j+1)})^{1/2}$ and Young's
inequality absorbs the resulting powers.
Taking $N=k+m$ and $j=k$ in \eqref{NewLemma1:discrete-result} proves
\eqref{NewLemma1:2}.  The special case $m=1$ is exactly
\eqref{NewLemma1:1}.
\end{proof}

\begin{lemma}[Simon local area bound]\label{L:simon-local-area}
This is the local area estimate used by Kuwert--Sch\"{a}tzle \cite[Lemma 4.1]{Kuwert1}.
Let $f:\Sigma\to\mathbb R^3$ be a proper immersion.  There is a
universal constant $C<\infty$ such that, for $0<\sigma\leq R<\infty$,
\[
    \frac{\mu(f^{-1}(B_\sigma(x)))}{\sigma^2}
    \leq
    C\left(
        \frac{\mu(f^{-1}(B_R(x)))}{R^2}
        +\int_{f^{-1}(B_R(x))}|H|^2\,d\mu
    \right).
\]
In particular, if $\Sigma$ is closed, then letting $R\to\infty$ gives
\[
    \mu(f^{-1}(B_\sigma(x)))
    \leq C\sigma^2\int_\Sigma |H|^2\,d\mu .
\]
\end{lemma}

Next we state Langer's compactness theorem, in the local form used by Breuning \cite{Breuning1} and Kuwert--Sch\"{a}tzle \cite[Theorem 4.2]{Kuwert1}.

\begin{theorem}[Langer's compactness theorem \cite{langer1984compactness}]\label{KuwertSchatzleCompactnessTheorem}
Let $f_{j}:\Sigma_{j}\rightarrow\mathbb{R}^{m\geq3}$ be a sequence of proper immersions, where each $\Sigma_{j}$ is a surface without boundary. For $R>0$ define
\[
\Sigma_{j}(R):=\left\{p\in\Sigma_{j}:|f_{j}(p)| < R\right\} = \Sigma_{j}\cap f_{j}^{-1}(B_{R}).
\]
Assume the bounds
\[
  \begin{cases}
    \mu_{j}(\Sigma_{j}(R)) \leq c(R)\,\,\text{for any}\,\,R>0,\,\,\text{and}\\
    \llll{\nabla_{(k)}A}_{\infty}\leq c(k)\,\,\text{for any}\,\,k\in\mathbb{N}_{0} 
  \end{cases}
\]
hold. Then there is a proper immersion $\tilde{f}:\tilde{\Sigma}\rightarrow\R^{3}$ (where $\tilde{\Sigma}$ is also a $2-$manifold without boundary) such that after passing to a subsequence one has
\begin{equation}
f_{j}\circ \phi_{j} = \tilde{f} + u_{j}\,\,\text{on}\,\,\tilde{\Sigma}\cap \tilde{f}^{-1}(B_{j})\label{KuwertSchatzleCompactnessTheoremEqn1}
\end{equation}
satisfying the following properties
\[
  \begin{cases}
    \phi_{j}:\tilde{\Sigma}(j)\rightarrow U_{j}\subset\Sigma_{j}\,\,\text{is a diffeomorphism},\\
    \Sigma_{j}(R)\subset U_{j}\,\,\text{if}\,\,j \geq j(R),\\
    u_{j}\in C^{\infty}(\tilde{\Sigma},\R^{3})\,\,\text{is normal along}\,\,\tilde{f}(\tilde{\Sigma}),\,\,\text{and}\\
    ||\tilde{\nabla}_{(k)}u_{j}||_{\infty,\Sigma_{j}}\rightarrow 0 \,\, \text{as}\,\,j\rightarrow\infty\,\,\text{for any}\,\,k\in\mathbb{N}_{0}.
  \end{cases}
\]
Here $\tilde{\nabla}$ refers to the Levi-Civita connection for the limit immersion.
\end{theorem}

The next lemma is similar to Lemma 23 from \cite{mccoy2015geometric}.
\begin{lemma}\label{MultliplicativeSobolevLemma2}
Suppose $f:\Sigma\rightarrow\mathbb{R}^{3}$ is an immersion satisfying
\[
\int_{[\gamma>0]}{\norm{A^{o}}^{2}\,d\mu} \leq \varepsilon_{0}
\]
for $\varepsilon_{0}>0$ sufficiently small, then for a universal constant $c>0$,
\begin{equation}
\intM{\norm{\nabla A^{o}}^{2}\gamma^{2}} + \intM{\norm{A^{o}}^{2}H^{2}\gamma^{2}} \leq c\intM{\norm{\nabla H}^{2}\gamma^{2}} +  c\,\cgam^{2}\Ao{2}, \label{MultliplicativeSobolevLemma2Eqn1}
\end{equation}
\begin{equation}
\intM{\nablanorm{2}{H}{2}\gamma^{4}} + \intM{\norm{\nabla H}^{2}H^{2}\gamma^{4}} \leq c \intM{\norm{\Delta H}^{2}\gamma^{4}} + c\,\cgam^{4}\Ao{2}, \label{MultliplicativeSobolevLemma2Eqn2}
\end{equation}
and
\begin{equation}
\intM{\norm{A^{o}}^{4}\gamma^{2}} \leq c\Ao{2}\oo{\intM{\norm{\nabla H}^{2}\gamma^{2}} + c\,\cgam^{2}\Ao{2}}.\label{MultliplicativeSobolevLemma2Eqn3}
\end{equation}
\end{lemma}
\begin{proof}
Firstly using the results of Lemma 22 from \cite{mccoy2015geometric}, the following identities hold:
\begin{equation}
\Delta A^{o} = S^{o}\left(\nabla_{(2)}H\right) + \frac{1}{2}H^{2}A^{o} - \norm{A^{o}}^{2}A^{o}\label{MultliplicativeSobolevLemma2Eqn4}
\end{equation}
and
\begin{equation}
\Delta\nabla H = \nabla\Delta H + \frac{1}{4}\nabla H\,\oo{H^{2}-2\norm{A^{o}}^{2}}.\label{MultliplicativeSobolevLemma2Eqn5}
\end{equation}
Here $S^{o}\oo{\nabla_{(2)}H} = \nabla_{(2)}H - \frac{1}{2}g\cdot\Delta H$ denotes the symmetric tracefree part of $\nabla_{(2)}H$. We aim to prove \eqref{MultliplicativeSobolevLemma2Eqn1} first. Applying integration by parts once, the identity \eqref{MultliplicativeSobolevLemma2Eqn4} and integration by parts again then yields
\begin{align}
&\intM{\norm{\nabla A^{o}}^{2}\gamma^{2}}
= -\intM{\inner{A^{o},\Delta A^{o}}\gamma^{2}} + \intM{\nabla A^{o} * A^{o} * \nabla\gamma * \gamma}\\
&= -\intM{\inner{A^{o},\nabla_{(2)}H}\gamma^{2}} - \frac{1}{2}\intM{\norm{A^{o}}^{2}H^{2}\gamma^{2}} + \intM{\norm{A^{o}}^{4}\gamma^{2}} + \intM{\nabla A^{o} * A^{o} * \nabla\gamma * \gamma}\\
&=\intM{\inner{\nabla^{*}A^{o},\nabla H}\gamma^{2}} - \frac{1}{2}\intM{\norm{A^{o}}^{2}H^{2}\gamma^{2}} + \intM{\norm{A^{o}}^{4}\gamma^{2}} + \intM{\nabla A^{o} * A^{o} * \nabla\gamma * \gamma}\\
& \qquad + \intM{\nabla H * A^{o} * \nabla\gamma * \gamma} \\
&\leq \frac{1}{2}\intM{\norm{\nabla H}^{2}\gamma^{2}}  - \frac{1}{2}\intM{\norm{A^{o}}^{2}H^{2}\gamma^{2}} + \intM{\norm{A^{o}}^{4}\gamma^{2}} + c\,\cgam\intM{\norm{\nabla A^{o}}\norm{A^{o}}\gamma}\\
& \qquad + c\,\cgam\intM{\norm{\nabla H}\norm{A^{o}}\gamma}\\
&\leq \frac{1}{2}\intM{\norm{\nabla H}^{2}\gamma^{2}}  - \frac{1}{2}\intM{\norm{A^{o}}^{2}H^{2}\gamma^{2}} + \intM{\norm{A^{o}}^{4}\gamma^{2}} + \eta\intM{\norm{\nabla A^{o}}^{2}\gamma^{2}} + c\,\cgam^{2}\Ao{2}.
\end{align}
Therefore absorbing into the left hand side and multiplying out gives
\begin{equation}
\intM{\norm{\nabla A^{o}}^{2}\gamma^{2}} + \intM{\norm{A^{o}}^{2}H^{2}\gamma^{2}} \leq c \intM{\norm{\nabla H}^{2}\gamma^{2}} + c\intM{\norm{A^{o}}^{4}\gamma^{2}} + c\,\cgam^{2}\Ao{2}.\label{MultliplicativeSobolevLemma2Eqn6}
\end{equation}

Next, we apply Theorem \ref{MichaelSimon} with $u=\norm{A^{o}}^{2}\gamma$, along with \eqref{MultliplicativeSobolevLemma2Eqn6} to give:
\begin{align}
\intM{\norm{A^{o}}^{4}\gamma^{2}} &\leq c\oo{\intM{\norm{\nabla A^{o}}\norm{A^{o}}\gamma} + \cgam\Ao{2} + \intM{\norm{A^{o}}^{2}\norm{H}\gamma}}^{2}\\
&\leq c\Ao{2}\oo{\intM{\norm{\nabla A^{o}}^{2}\gamma^{2}} + \intM{\norm{A^{o}}^{2}H^{2}\gamma^{2}}} + c\,\cgam^{2}\Ao{4}\\
&\leq c\Aoc{2}\oo{\intM{\norm{\nabla H}^{2}\gamma^{2}} + \intM{\norm{A^{o}}^{4}\gamma^{2}} + c\,\cgam^{2}\Ao{2}} + c\,\cgam^{2}\Ao{4},
\end{align}
and hence
\[
\bigl(1-c\Aoc{2}\bigr)\intM{\norm{A^{o}}^{4}\gamma^{2}} \leq c\Ao{2}\oo{\intM{\norm{\nabla H}^{2}\gamma^{2}} + c\,\cgam^{2}\Ao{2}}.
\]
Therefore if $\varepsilon_{0}>0$ is sufficiently small then \eqref{MultliplicativeSobolevLemma2Eqn3} holds. Combining \eqref{MultliplicativeSobolevLemma2Eqn3} (which was just proven) with \eqref{MultliplicativeSobolevLemma2Eqn6} then immediately gives \eqref{MultliplicativeSobolevLemma2Eqn1}.
Next we look at \eqref{MultliplicativeSobolevLemma2Eqn2}. Applying \eqref{MultliplicativeSobolevLemma2Eqn5} and integration by parts twice gives
\begin{align}
\intM{\nablanorm{2}{H}{2}\gamma^{4}} &= -\intM{\inner{\nabla H,\Delta\nabla H}\gamma^{4}} + \intM{\nabla_{(2)}H * \nabla H * \nabla\gamma * \gamma^{3}}\\
&= - \intM{\inner{\nabla H,\nabla\Delta H}\gamma^{4}} - \frac{1}{4}\intM{\norm{\nabla H}^{2}H^{2}\gamma^{4}} + \frac{1}{2}\intM{\norm{\nabla H}^{2}\norm{A^{o}}^{2}\gamma^{4}}\\
&\quad +  \intM{\nabla_{(2)}H * \nabla H * \nabla\gamma * \gamma^{3}}\\
&= \intM{\norm{\Delta H}^{2}\gamma^{4}} + \intM{\Delta H * \nabla H * \nabla\gamma * \gamma^{3}} - \frac{1}{4}\intM{\norm{\nabla H}^{2}H^{2}\gamma^{4}}\\
&\quad + \frac{1}{2}\intM{\norm{\nabla H}^{2}\norm{A^{o}}^{2}\gamma^{4}} +  \intM{\nabla_{(2)}H * \nabla H * \nabla\gamma * \gamma^{3}}\\
&\leq \intM{\norm{\Delta H}^{2}\gamma^{4}}  - \frac{1}{4}\intM{\norm{\nabla H}^{2}H^{2}\gamma^{4}} + c\,\cgam\intM{\norm{\Delta H}\norm{\nabla H}\gamma^{3}}\\
&\quad + \frac{1}{2}\intM{\norm{\nabla H}^{2}\norm{A^{o}}^{2}\gamma^{4}} + c\,\cgam\intM{\nablanorm{2}{H}{}\norm{\nabla H}\gamma^{3}}\\
&\leq c\intM{\norm{\Delta H}^{2}\gamma^{4}} + \eta\intM{\nablanorm{2}{H}{2}\gamma^{4}} -\frac{1}{4}\intM{\norm{\nabla H}^{2}H^{2}\gamma^{4}}\\
&\quad + c\,\eta^{-1}\,\cgam^{2}\intM{\norm{\nabla H}^{2}\gamma^{2}} + \frac{1}{2}\intM{\norm{\nabla H}^{2}\norm{A^{o}}^{2}\gamma^{4}},\label{MultliplicativeSobolevLemma2Eqn7}
\end{align}
which holds for any $\eta>0$. Next, by applying integration by parts and Young's inequality we have
\begin{align}
\intM{\norm{\nabla H}^{2}\gamma^{2}} &= 2\intM{\inner{\nabla^{*}A^{o},\nabla H}\gamma^{2}}\\
&=-2\intM{\inner{A^{o},\nabla_{(2)}H}\gamma^{2}} + \intM{A^{o} * \nabla H * \nabla\gamma * \gamma}\\
&\leq c\Ao{}\oo{\intM{\nablanorm{2}{H}{2}\gamma^{4}}}^{\frac{1}{2}} + c\,\cgam\Ao{}\oo{\intM{\norm{\nabla H}^{2}\gamma^{2}}}^{\frac{1}{2}}\\
&\leq c\Ao{}\oo{\intM{\nablanorm{2}{H}{2}\gamma^{4}}}^{\frac{1}{2}} + \frac{1}{2}\intM{\norm{\nabla H}^{2}\gamma^{2}} + c\,\cgam^{2}\Ao{2},
\end{align}
which implies
\begin{equation}
\intM{\norm{\nabla H}^{2}\gamma^{2}} \leq c\Ao{}\oo{\intM{\nablanorm{2}{H}{2}\gamma^{4}}}^{\frac{1}{2}} + c\,\cgam^{2}\Ao{2}.\label{MultliplicativeSobolevLemma2Eqn7,0}
\end{equation}
Substituting this back into \eqref{MultliplicativeSobolevLemma2Eqn7} and applying Young's inequality gives
\begin{align}
&(1-\eta)\intM{\nablanorm{2}{H}{2}\gamma^{4}} + \frac{1}{4}\intM{\norm{\nabla H}^{2}H^{2}\gamma^{4}}\\
&\leq c\intM{\norm{\Delta H}^{2}\gamma^{4}}
+ \frac{1}{2}\intM{\norm{\nabla H}^{2}\norm{A^{o}}^{2}\gamma^{4}}
\\&\qquad 
+ c\,\eta^{-1}\cgam^{2}\bigg(\Ao{}\oo{\intM{\nablanorm{2}{H}{2}\gamma^{4}}}^{\frac{1}{2}} + c\,\cgam^{2}\Ao{2}\bigg)
\\
&\leq c\intM{\norm{\Delta H}^{2}\gamma^{4}} + \tilde{\eta}\intM{\nablanorm{2}{H}{2}\gamma^{4}} + c\,\tilde{\eta}^{-1}\cgam^{4}\Ao{2}+ \frac{1}{2}\intM{\norm{\nabla H}^{2}\norm{A^{o}}^{2}\gamma^{4}}
\,,
\label{MultliplicativeSobolevLemma2Eqn7,1}
\end{align}
where $\tilde{\eta}=\tilde{\eta}(\eta)>0$ is a new constant obtained from Young's inequality that is allowed to be as small as desired. Then by choosing $\eta,\tilde{\eta}>0$ sufficiently small in \eqref{MultliplicativeSobolevLemma2Eqn7,1} and absorbing into the left hand side, we obtain
\begin{align}
\intM{\nablanorm{2}{H}{2}\gamma^{4}} &+ \intM{\norm{\nabla H}^{2}H^{2}\gamma^{4}}
\\&\leq c\intM{\norm{\Delta H}^{2}\gamma^{4}} + c\intM{\norm{\nabla H}^{2}\norm{A^{o}}^{2}\gamma^{4}} + c\,\cgam^{4}\Ao{2}.\label{MultliplicativeSobolevLemma2Eqn8}
\end{align}
for some absolute constants $c>0$. Next, applying Theorem \ref{MichaelSimon} with $u=\norm{\nabla H}\norm{A^{o}}\gamma^{2}$ yields
\begin{align}
&\intM{\norm{\nabla H}^{2}\norm{A^{o}}^{2}\gamma^{4}} \leq c\oo{\intM{\nablanorm{2}{H}{}\norm{A^{o}}\gamma^{2}}+ \intM{\norm{\nabla A^{o}}^{2}\gamma^{2}}+ c\,\cgam\intM{\norm{\nabla A^{o}}\norm{A^{o}}\gamma}}^{2}\\
&\quad \leq c\Ao{2}\intM{\nablanorm{2}{H}{2}\gamma^{4}}+c\oo{\intM{\norm{\nabla A^{o}}^{2}\gamma^{2}}}^{2}+ c\,\cgam^{2}\Ao{2}\intM{\norm{\nabla A^{o}}^{2}\gamma^{2}}\\
&\quad \leq c\Ao{2}\intM{\nablanorm{2}{H}{2}\gamma^{4}} + c\oo{\intM{\norm{\nabla A^{o}}^{2}\gamma^{2}}}^{2} + c\cgam^{4}\Ao{4},\label{MultliplicativeSobolevLemma2Eqn9}
\end{align}
where we have also applied the Cauchy-Schwarz and Young inequalities. Combining \eqref{MultliplicativeSobolevLemma2Eqn7,0}, \eqref{MultliplicativeSobolevLemma2Eqn8}, \eqref{MultliplicativeSobolevLemma2Eqn9}, and  \eqref{MultliplicativeSobolevLemma2Eqn1}, we therefore have:
\begin{align}
&\intM{\nablanorm{2}{H}{2}\gamma^{4}} + \intM{\norm{\nabla H}^{2}H^{2}\gamma^{4}}\leq c\intM{\norm{\Delta H}^{2}\gamma^{4}} + c\oo{\intM{\norm{\nabla A^{o}}^{2}\gamma^{2}}}^{2} + c\cgam^{4}\Ao{2}\\
& \leq c\intM{\norm{\Delta H}^{2}\gamma^{4}} + c\oo{\Ao{}\oo{\intM{\nablanorm{2}{H}{2}\gamma^{4}}}^{\frac{1}{2}} + c\,\cgam^{2}\Ao{2}}^{2} + c\,\cgam^{4}\Ao{2}\\
& \leq c\intM{\norm{\Delta H}^{2}\gamma^{4}} +  c\Ao{2}\intM{\nablanorm{2}{H}{2}\gamma^{4}} + c\,\cgam^{4}\Ao{2}.\label{MultliplicativeSobolevLemma2Eqn10}
\end{align}
Therefore if $\varepsilon_{0}>0$ is sufficiently small, absorbing into the left hand side proves \eqref{MultliplicativeSobolevLemma2Eqn2}.
\end{proof}

Our next estimate will be a key ingredient for our interpolation inequalities later on.
\begin{lemma}\label{AppendixLemmaGradHInfty}
Suppose $f:\Sigma\rightarrow\mathbb{R}^{3}$ is a closed immersion satisfying $\Aoc{2}\leq\varepsilon_{0}$ for $\varepsilon_{0}>0$ sufficiently small. Then there exists a universal constant $c>0$ such that
\[
\llll{\nabla H}_{\infty}^{6} \leq c \Aoc{2} \cdot \llll{\nabla\Delta H}_{2}^{4}.
\]
\end{lemma}
\begin{proof}
Using the multiplicative Sobolev inequality from Theorem \ref{EvolutionTheorem1} with $p=4,\,m=2$ immediately gives
\begin{equation}
\llll{\nabla H}_{\infty}^{6} \leq c \intM{\norm{\nabla H}^{2}}\cdot\oo{\intM{\nablanorm{2}{H}{4}} + \intM{\norm{\nabla H}^{4}H^{4}}}.\label{AppendixLemmaGradHInftyEqn1}
\end{equation}
Note that an estimate for $\llll{\nabla H}_{2}^{2}$ follows directly from \eqref{MultliplicativeSobolevLemma2Eqn1} and the identity $|\nabla H|^{2} \leq 4 \norm{\nabla A^{o}}^{2}$:
\[
\intM{\norm{\nabla H}^{2}} \leq c\Aoc{}\oo{\intM{\norm{\Delta H}^{2}}}^{\frac{1}{2}}.
\]
Next using Theorem \ref{MichaelSimon} with $u=\nablanorm{2}{H}{2}$ along with the inequalities \eqref{LemmaPreservedSphericityEstimates1,Eqn3} and \eqref{LemmaPreservedSphericityEstimates1,Eqn4} with $\varphi=H$ to estimate the first term in parenthesis on the right gives
\begin{align}
\intM{\nablanorm{2}{H}{4}} &\leq c\oo{\intM{\nablanorm{3}{H}{}\nablanorm{2}{H}{}}+\intM{\nablanorm{2}{H}{2}|H|}}^{2}\\
&\leq c\intM{\nablanorm{2}{H}{2}}\cdot\oo{\intM{\nablanorm{3}{H}{2}}+\intM{\nablanorm{2}{H}{2}H^{2}}}\\
&\leq c\intM{\norm{\Delta H}^{2}}\cdot\oo{\intM{\nablanorm{3}{H}{2}}+\intM{\nablanorm{2}{H}{2}H^{2}}}\\
&\leq c\intM{\norm{\Delta H}^{2}}\cdot\oo{\intM{\norm{\nabla\Delta H}^{2}}+\intM{\norm{\nabla H}^{2}}\cdot\intM{\norm{\Delta H}^{2}}}.\label{AppendixLemmaGradHInftyEqn2}
\end{align}
We leave \eqref{AppendixLemmaGradHInftyEqn2} for the time being, and go on to estimate the rest of the terms on the right hand side of \eqref{AppendixLemmaGradHInftyEqn1}. Using Theorem \ref{MichaelSimon} with $u=\norm{\nabla H}^{2}H^{2}$ gives
\begin{align}
\intM{\norm{\nabla H}^{4}H^{4}}
&\quad \leq c\oo{\intM{\nablanorm{2}{H}{}\norm{\nabla H}H^{2}}+ \intM{\norm{\nabla H}^{3}\norm{H}}+\intM{\norm{\nabla H}^{2}\norm{H}^{3}}}^{2}\\
&\quad \leq c\oo{\intM{\nablanorm{2}{H}{2}}+\intM{\norm{\nabla H}^{2}H^{2}}}\cdot\oo{\intM{\norm{\nabla H}^{2}H^{4}}+\intM{\norm{\nabla H}^{4}}}\\
&\quad \leq c\intM{\norm{\Delta H}^{2}}\cdot \oo{\intM{\norm{\nabla H}^{2}H^{4}}+\intM{\norm{\nabla H}^{4}}}\\
&\quad \leq c\intM{\norm{\Delta H}^{2}}\cdot\oo{\intM{\norm{\nabla\Delta H}^{2}} + c\intM{\norm{\nabla H}^{2}}\cdot\intM{\norm{\Delta H}^{2}} + \intM{\norm{\nabla H}^{4}}},\label{AppendixLemmaGradHInftyEqn3}
\end{align}
where we have again used the inequalities \eqref{LemmaPreservedSphericityEstimates1,Eqn3} and \eqref{LemmaPreservedSphericityEstimates1,Eqn4} with $\varphi=H$. The rightmost term in \eqref{AppendixLemmaGradHInftyEqn3} is estimated similarly. Using Theorem \ref{MichaelSimon} with $u=\norm{\nabla H}^{2}$ combined with \eqref{MultliplicativeSobolevLemma2Eqn2} yields
\begin{multline}
\intM{\norm{\nabla H}^{4}} \leq c\oo{\intM{\nablanorm{2}{H}{}\norm{\nabla H}} +\intM{\norm{\nabla H}^{2}\norm{H}} }^{2}\\
\leq c\intM{\norm{\nabla H}^{2}}\cdot\oo{\intM{\nablanorm{2}{H}{2}}+\intM{\norm{\nabla H}^{2}H^{2}}}\leq c \intM{\norm{\nabla H}^{2}}\cdot\intM{\norm{\Delta H}^{2}}.\label{AppendixLemmaGradHInftyEqn4}
\end{multline}

Substituting \eqref{AppendixLemmaGradHInftyEqn2} and \eqref{AppendixLemmaGradHInftyEqn4} into \eqref{AppendixLemmaGradHInftyEqn1} then gives
\begin{equation}
\llll{\nabla H}_{\infty}^{6} \leq c\intM{\norm{\nabla H}^{2}}\cdot\intM{\norm{\Delta H}^{2}}\cdot\oo{\intM{\norm{\nabla\Delta H}^{2}}+\intM{\norm{\nabla H}^{2}}\cdot\intM{\norm{\Delta H}^{2}}}.\label{AppendixLemmaGradHInftyEqn5}
\end{equation}
The term $\llll{\nabla H}_{2}^{2}\cdot\llll{\Delta H}_{2}^{2}$ can be estimated quite simply: first note that applying Theorem \ref{HigherOrderSobolevTheorem1}, followed by integration by parts and Young's inequality with conjugate exponents $p=4$ and $p^{*}=\frac{4}{3}$ yields
\begin{multline}
\intM{\norm{\nabla H}^{2}} \leq c\Aoc{}\oo{\intM{\norm{\Delta H}^{2}}}^{\frac{1}{2}}
\leq c\Aoc{}\oo{\intM{\norm{\nabla H}^{2}}}^{\frac{1}{4}}\cdot\oo{\intM{\norm{\nabla\Delta H}^{2}}}^{\frac{1}{4}}\\
\leq \frac{1}{2}\intM{\norm{\nabla H}^{2}} + c\Aoc{\frac{4}{3}}\oo{\intM{\norm{\nabla\Delta H}^{2}}}^{\frac{1}{3}},
\end{multline}
which implies the estimate
\[
\intM{\norm{\nabla H}^{2}} \leq c\Aoc{\frac{4}{3}}\oo{\intM{\norm{\nabla\Delta H}^{2}}}^{\frac{1}{3}}.
\]
Combining the above estimate with integration by parts and the Cauchy-Schwarz inequality then gives
\begin{multline*}
\intM{\norm{\nabla H}^{2}}\cdot\intM{\norm{\Delta H}^{2}} = -\intM{\norm{\nabla H}^{2}}\cdot\intM{\inner{\nabla H,\nabla\Delta H}}\\
\leq c\oo{\intM{\norm{\nabla H}^{2}}}^{\frac{3}{2}}\oo{\intM{\norm{\nabla\Delta H}^{2}}}^{\frac{1}{2}}
\leq c\Aoc{2}\intM{\norm{\nabla\Delta H}^{2}}.
\end{multline*}
Substituting into \eqref{AppendixLemmaGradHInftyEqn5} then completes the proof.
\end{proof}

\begin{theorem}\label{TheoremX}
Suppose $f:\Sigma\rightarrow\mathbb{R}^{3}$ is a closed immersion satisfying $\Aoc{2}\leq\varepsilon_{0}$ for $\varepsilon_{0}>0$ sufficiently small. Fix $m\geq3$ and define $Q_{m,l},0\leq l\leq m$ by
\[
Q_{m,l}:=\intM{\nablanorm{m-l}{A^{o}}{2}H^{2l}}.
\]
Then, for every $1\leq l\leq m-2$, we have the recurrence inequality
\[
Q_{m,l}\leq c\,Q_{m,l-1}^{\frac{1}{2}}Q_{m,l+1}^{\frac{1}{2}}+c\Aoc{2}\oo{Q_{m,0} + Q_{m,l+1} + S_{m}}.
\]
Here we have used the notation for $S_{m}$ from Lemma \ref{HigherOrderSobolevLemma1}.
\end{theorem}
\begin{proof}
Integration by parts and H\"{o}lder's inequality, followed by Young's inequality immediately gives
\begin{align}
Q_{m,l} &= \intM{\nabla_{(m-(l+1))}A^{o} * \nabla_{(m-(l-1))}A^{o} * H^{2l}}+ \intM{\nabla_{(m-l)}A^{o} * \nabla_{(m-(l+1))}A^{o} * \nabla H * H^{2l-1}}\\
&\leq c\,Q_{m,l-1}^{\frac{1}{2}}Q_{m,l+1}^{\frac{1}{2}} + c\llll{\nabla H}_{\infty}Q_{m,l}^{\frac{1}{2}}\oo{\intM{\nablanorm{m-(l+1)}{A^{o}}{2}H^{2(l-1)}}}^{\frac{1}{2}}\\
&\leq c\,Q_{m,l-1}^{\frac{1}{2}}Q_{m,l+1}^{\frac{1}{2}} + \eta\,Q_{m,l} + c\,\eta^{-1}\llll{\nabla H}_{\infty}^{2}\intM{\nablanorm{m-(l+1)}{A^{o}}{2}H^{2(l-1)}}.
\end{align}
We then absorb and multiply out to give
\begin{equation}
Q_{m,l} \leq c\,Q_{m,l-1}^{\frac{1}{2}}Q_{m,l+1}^{\frac{1}{2}} + c\llll{\nabla H}_{\infty}^{2}\intM{\nablanorm{m-(l+1)}{A^{o}}{2}H^{2(l-1)}}.\label{TheoremX,1}
\end{equation}
We first dispose of the endpoint $l=1$.  In this case the last integral
in \eqref{TheoremX,1} is $\int|\nabla_{(m-2)}A^o|^2$.  Lemma
\ref{NewLemma1} and \eqref{TheoremX,4} below give
\[
\llll{\nabla H}_{\infty}^{2}
\intM{|\nabla_{(m-2)}A^o|^2}
\leq
C\Aoc{2}S_m^{2/m}Q_{m,0}^{(m-2)/m}
\leq
C\Aoc{2}\oo{Q_{m,0}+S_m}.
\]
Thus the desired estimate follows for $l=1$.  We may therefore assume
$l\geq2$.
Next we apply H\"{o}lder's inequality with conjugate exponents $\alpha=(l+1)/(l-1)$ and $\alpha^{*}=(l+1)/2$ to the last term on the right hand side:
\begin{equation}
\intM{\nablanorm{m-(l+1)}{A^{o}}{2}H^{2(l-1)}} \leq Q_{m,l+1}^{\frac{l-1}{l+1}}\cdot\oo{\intM{\nablanorm{m-(l+1)}{A^{o}}{2}}}^{\frac{2}{l+1}}.\label{TheoremX,2}
\end{equation}
The interpolation inequality of Lemma \ref{NewLemma1} implies
\begin{equation*}
\intM{\nablanorm{m-(l+1)}{A^{o}}{2}} \leq c\Aoc{\frac{2(l+1)}{m}}\oo{\intM{\nablanorm{m}{A^{o}}{2}}}^{\frac{m-(l+1)}{m}}
= c\Aoc{\frac{2(l+1)}{m}}Q_{m,0}^{\frac{m-(l+1)}{m}},
\end{equation*}
meaning that \eqref{TheoremX,2} becomes
\begin{equation}
\intM{\nablanorm{m-(l+1)}{A^{o}}{2}H^{2(l-1)}} \leq c\Aoc{\frac{4}{m}}Q_{m,0}^{\frac{2(m-(l+1))}{m(l+1)}}Q_{m,l+1}^{\frac{l-1}{l+1}}.\label{TheoremX,3}
\end{equation}
Next, using Lemma \ref{AppendixLemmaGradHInfty} along with the interpolation inequalities from Theorem \ref{HigherOrderSobolevTheorem1} we have
\begin{equation}
\llll{\nabla H}_{\infty}^{2} \leq c\Aoc{\frac{2}{3}}\oo{\intM{\norm{\nabla\Delta H}^{2}}}^{\frac{2}{3}}
\leq c\Aoc{\frac{2}{3} + \frac{4(m-3)}{3m}}\oo{\intM{\norm{\Delta^{\frac{m}{2}}H}^{2}}}^{\frac{2}{m}}
= c\Aoc{\frac{2}{3}+\frac{4(m-3)}{3m}}S_{m}^{\frac{2}{m}}.\label{TheoremX,4}
\end{equation}
Substituting \eqref{TheoremX,2}, \eqref{TheoremX,3} and \eqref{TheoremX,4} into \eqref{TheoremX,1} then gives
\[
Q_{m,l} \leq c\,Q_{m,l-1}^{\frac{1}{2}}Q_{m,l+1}^{\frac{1}{2}} + c\Aoc{2}Q_{m,0}^{\frac{2(m-(l+1))}{m(l+1)}}Q_{m,l+1}^{\frac{l-1}{l+1}}S_{m}^{\frac{2}{m}}.
\]
Since $l\leq m-2$, the exponent $m-(l+1)$ is positive.  Using the generalised Young's inequality with conjugate exponents $\alpha=\frac{m(l+1)}{2(m-(l+1))},\,\beta=\frac{l+1}{l-1}$ and $\delta=\frac{m}{2}$ on the last term then completes the proof.
\end{proof}

\begin{theorem}\label{TheoremY}
Suppose $f:\Sigma^{2}\rightarrow\mathbb{R}^{3}$ is a closed immersion satisfying $\Aoc{2}\leq\varepsilon_{0}$ for $\varepsilon_{0}$ sufficiently small. Fix $m\geq3$. Then for some bounded constant $c=c(m)$:
\begin{equation}
Q_{m,m}+Q_{m,m-1}\leq c\,Q_{m,m-2}^{\frac{1}{2}}Q_{m,m}^{\frac{1}{2}}+c\Aoc{m-1}S_{m}.\label{TheoremY,0}
\end{equation}
\end{theorem}
\begin{proof}
Firstly, integration by parts and the identity \eqref{MultliplicativeSobolevLemma2Eqn1} yields
\begin{align}
&Q_{m,m-1} = \intM{\norm{\nabla A^{o}}^{2}H^{2(m-1)}}
= -\intM{\inner{A^{o},\Delta A^{o}}H^{2(m-1)}} + \intM{\nabla A^{o} * A^{o} * \nabla H * H^{2m-3}}\\
&= -\intM{\inner{A^{o},\nabla_{(2)}H}H^{2(m-1)}} -\frac{1}{2}\intM{\norm{A^{o}}^{2}H^{2m}} + \intM{\norm{A^{o}}^{4}H^{2(m-1)}}\\
& \quad + \intM{\nabla A^{o} * A^{o} * \nabla H * H^{2m-3}}\\
&\leq \oo{\intM{\nablanorm{2}{H}{2}H^{2(m-2)}}}^{\frac{1}{2}}Q_{m,m}^{\frac{1}{2}} -\frac{1}{2}Q_{m,m}+ \intM{\norm{A^{o}}^{4}H^{2(m-1)}}\\
& \quad + c\intM{\norm{\nabla A^{o}}\norm{\nabla H}\norm{A^{o}}\norm{H}^{2m-3}}\\
&\leq \left\{\eta \, Q_{m,m} + c\,\eta^{-1}\intM{\nablanorm{2}{H}{2}H^{2(m-2)}}\right\} -\frac{1}{2}Q_{m,m}+ \intM{\norm{A^{o}}^{4}H^{2(m-1)}}\\
&\qquad+ c\llll{\nabla H}_{\infty}\oo{\intM{\norm{\nabla A^{o}}^{2}H^{2(m-1)}}}^{\frac{1}{2}}\oo{\intM{\norm{A^{o}}^{2}H^{2m}}}^{\frac{m-2}{2m}}\oo{\intM{\norm{A^{o}}^{2}}}^{\frac{1}{m}}\\
&=\{\eta \, Q_{m,m} + c\,\eta^{-1}\intM{\nablanorm{2}{H}{2}H^{2(m-2)}}\} -\frac{1}{2}Q_{m,m}+ \intM{\norm{A^{o}}^{4}H^{2(m-1)}}\\
&\quad + c\llll{\nabla H}_{\infty}\Aoc{\frac{2}{m}}Q_{m,m-1}^{\frac{1}{2}}Q_{m,m}^{\frac{m-2}{2m}}
\end{align}
for any $\eta>0$. This can be rearranged to give
\begin{multline}
Q_{m,m-1} + \oo{\frac{1}{2} - \eta}Q_{m,m}\\
\leq c\,\eta^{-1}\intM{\nablanorm{2}{H}{2}H^{2(m-2)}} + \intM{\norm{A^{o}}^{4}H^{2(m-1)}} + c\llll{\nabla H}_{\infty}\Aoc{\frac{2}{m}}Q_{m,m-1}^{\frac{1}{2}}Q_{m,m}^{\frac{m-2}{2m}},
\end{multline}
which holds for any $\eta>0$. By choosing $\eta$ sufficiently small, the above implies
\begin{equation}
Q_{m,m-1} + Q_{m,m} \leq c\intM{\nablanorm{2}{H}{2}H^{2(m-2)}} + c\intM{\norm{A^{o}}^{4}H^{2(m-1)}} + c\llll{\nabla H}_{\infty}\Aoc{\frac{2}{m}}Q_{m,m-1}^{\frac{1}{2}}Q_{m,m}^{\frac{m-2}{2m}} \label{TheoremYEqn1}
\end{equation}
The last term on the right hand side of \eqref{TheoremYEqn1} is estimated quite easily. First note that by our $L^{\infty}$ estimate in Lemma \ref{AppendixLemmaGradHInfty} and the interpolation inequality of Theorem \ref{HigherOrderSobolevTheorem1}, we have
\begin{equation}
\llll{\nabla H}_{\infty} \leq c\Aoc{\frac{m-2}{m}}S_{m}^{\frac{1}{m}},\label{TheoremYEqn2}
\end{equation}
and therefore
\begin{align}
c\llll{\nabla H}_{\infty}\Aoc{\frac{2}{m}}Q_{m,m-1}^{\frac{1}{2}}Q_{m,m}^{\frac{m-2}{2m}} &\leq c \Aoc{}S_{m}^{\frac{1}{m}}Q_{m,m-1}^{\frac{1}{2}}Q_{m,m}^{\frac{m-2}{2m}}\\
&\leq \eta\oo{Q_{m,m-1}+Q_{m,m}} + c\,\eta^{-1}\Aoc{m}S_{m}\label{TheoremYEqn3}
\end{align}
for any $\eta>0$. Next, we estimate the penultimate term in \eqref{TheoremYEqn1}. We first employ the following estimate which follows from H\"{o}lder's inequality with conjugate exponents $\alpha=\frac{m}{m-1},\,\alpha^{*}=m$:
 \begin{multline}
\intM{\norm{A^{o}}^{4}H^{2(m-1)}} \leq \oo{\intM{\norm{A^{o}}^{2}H^{2m}}}^{\frac{m-1}{m}}\oo{\intM{\norm{A^{o}}^{2(m+1)}}}^{\frac{1}{m}}\\
\leq \Aoc{\frac{2}{m}}\llll{A^{o}}_{\infty}^{2}\oo{\intM{\norm{A^{o}}^{2}H^{2m}}}^{\frac{m-1}{m}}
\leq c\Aoc{2}S_{m}^{\frac{1}{m}}Q_{m,m}^{\frac{m-1}{m}} 
\leq \eta\,Q_{m,m} + c\,\eta^{-1}\Aoc{2m}S_{m}.\label{TheoremYEqn4}
\end{multline}
Here we have also employed \eqref{EpsilonRegularityProof1} and Theorem \ref{HigherOrderSobolevTheorem1}.

Substituting \eqref{TheoremYEqn3} and \eqref{TheoremYEqn4} back into \eqref{TheoremYEqn1} yields
\begin{equation}
\oo{1-\eta}Q_{m,m-1}+ \oo{1-2\eta}Q_{m,m} \leq c\intM{\nablanorm{2}{H}{2}H^{2(m-2)}} + c\,\eta^{-1}\Aoc{m}S_{m},
\end{equation}
which is valid for any $\eta>0$. Choosing $\eta$ sufficiently small gives
\begin{equation}
Q_{m,m-1} + Q_{m,m} \leq c\intM{\nablanorm{2}{H}{2}H^{2(m-2)}} + c\Aoc{m}S_{m}\label{TheoremYEqn7}
\end{equation}
for some absolute constant $c>0$. We now estimate the penultimate term on the right hand side of \eqref{TheoremYEqn7}. Employing identity \eqref{MultliplicativeSobolevLemma2Eqn5} and two applications of integration by parts yields
\begin{align}
&\intM{\nablanorm{2}{H}{2}H^{2(m-2)}}\\
&\quad = -\intM{\inner{\nabla H,\nabla\Delta H + \frac{1}{4}\nabla H\oo{H^{2}-2\norm{A^{o}}^{2}}}H^{2(m-2)}} + \intM{\nabla_{(2)}H * \nabla H * \nabla H * H^{2m-5}}\\
&\quad = \intM{\norm{\Delta H}^{2} H^{2(m-2)}} - \frac{1}{4}\intM{\norm{\nabla H}^{2}H^{2(m-1)}} + \frac{1}{2}\intM{\norm{\nabla H}^{2}\norm{A^{o}}^{2}H^{2(m-2)}}\\
&\qquad + c\intM{\nabla_{(2)}H * \nabla H * \nabla H * H^{2m-5}}\\
&\quad \leq \intM{\norm{\Delta H}^{2} H^{2(m-2)}} - \frac{1}{4}\intM{\norm{\nabla H}^{2}H^{2(m-1)}} + \frac{1}{2}\intM{\norm{\nabla H}^{2}\norm{A^{o}}^{2}H^{2(m-2)}}\\
&\qquad + c\intM{\nablanorm{2}{H}{}\norm{\nabla H}^{2}\norm{H}^{2m-5}}
\end{align} 
The above can be rearranged to read
\begin{multline}
\intM{\nablanorm{2}{H}{2}H^{2(m-2)}} + \frac{1}{4}\intM{\norm{\nabla H}^{2}H^{2(m-1)}}\\
\leq \intM{\norm{\Delta H}^{2} H^{2(m-2)}} + \frac{1}{2}\intM{\norm{\nabla H}^{2}\norm{A^{o}}^{2}H^{2(m-2)}} + c\intM{\nablanorm{2}{H}{}\norm{\nabla H}^{2}\norm{H}^{2m-5}}.\label{TheoremYEqn8}
\end{multline}
We estimate the last term on the right hand side of \eqref{TheoremYEqn8} in a similar way to our above estimate \eqref{TheoremYEqn3}. Employing the H\"{o}lder inequality as well as the generalised Young's inequality with conjugate exponents $\alpha=2,\,\beta=2(m-1)/(m-3)$ and $\delta=m-1$ and our $L^{\infty}$ estimate from \eqref{TheoremYEqn2} yields the following estimate.  For the borderline case $m=3$, the same estimate is understood in the limiting sense: the middle H\"{o}lder factor below has exponent zero and is omitted.
\begin{align}
&\intM{\nablanorm{2}{H}{}\norm{\nabla H}^{2}\norm{H}^{2m-5}} \leq \llll{\nabla H}_{\infty}\intM{\nablanorm{2}{H}{}\norm{\nabla H}\norm{H}^{2m-5}}\\
&\leq \llll{\nabla H}_{\infty}\oo{\intM{\nablanorm{2}{H}{2}H^{2(m-2)}}}^{\frac{1}{2}}\oo{\intM{\norm{\nabla H}^{2}H^{2(m-1)}}}^{\frac{m-3}{2(m-1)}}\oo{\intM{\norm{\nabla H}^{2}}}^{\frac{1}{m-1}}\\
&\leq c\Aoc{}S_{m}^{\frac{1}{m-1}}\oo{\intM{\nablanorm{2}{H}{2}H^{2(m-2)}}}^{\frac{1}{2}}\oo{\intM{\norm{\nabla H}^{2}H^{2(m-1)}}}^{\frac{m-3}{2(m-1)}}\\
&\leq \eta \oo{\intM{\nablanorm{2}{H}{2}H^{2(m-2)}}+\intM{\norm{\nabla H}^{2}H^{2(m-1)}}} + c\,\eta^{-1}\Aoc{m-1}S_{m},\label{TheoremYEqn9}
\end{align}
which is valid for any $\eta>0$. We leave \eqref{TheoremYEqn9} for the time being, and look to estimate the penultimate term in \eqref{TheoremYEqn8}. Using the $L^{\infty}$ estimate of \eqref{EpsilonRegularityProof1} as well as the interpolation inequalities from Theorem \ref{HigherOrderSobolevTheorem1} gives
\begin{align}
\intM{\norm{\nabla H}^{2}\norm{A^{o}}^{2}H^{2(m-2)}} &\leq \llll{A^{o}}_{\infty}^{2}\intM{\norm{\nabla H}^{2}H^{2(m-2)}}\\
&\leq \llll{A^{o}}_{\infty}^{2}\oo{\intM{\norm{\nabla H}^{2}H^{2(m-1)}}}^{\frac{m-2}{m-1}}\oo{\intM{\norm{\nabla H}^{2}}}^{\frac{1}{m-1}}\\
&\leq c\Aoc{2}S_{m}^{\frac{1}{m-1}}\oo{\intM{\norm{\nabla H}^{2}H^{2(m-1)}}}^{\frac{m-2}{m-1}}\\
&\leq \eta\intM{\norm{\nabla H}^{2}H^{2(m-1)}} + c\,\eta^{-1}\Aoc{2(m-1)}S_{m}.\label{TheoremYEqn10}
\end{align}
for any $\eta>0$. Substituting \eqref{TheoremYEqn9} and \eqref{TheoremYEqn10} into \eqref{TheoremYEqn8} gives, for $\eta>0$ sufficiently small,
\begin{equation}
\intM{\nablanorm{2}{H}{2}H^{2(m-2)}} + \intM{\norm{\nabla H}^{2}H^{2(m-1)}} \leq c\intM{\norm{\Delta H}^{2}H^{2(m-2)}} + c\Aoc{m-1}S_{m}.
\end{equation}
Combining this with \eqref{TheoremYEqn7} then completes the proof.
\end{proof}

\begin{theorem}\label{TheoremXX}
Suppose $f:\Sigma^{2}\rightarrow\mathbb{R}^{3}$ is a closed immersion satisfying $||A^{o}||_{2}^{2}\leq\varepsilon_{0}$ for $\varepsilon_{0}$ sufficiently small. Fix $m\geq3$. Then for any $0\leq l \leq m-1$ there exists a constant $c$ such that
\[
Q_{m,0}+Q_{m,1}+\hdots+Q_{m,l}\leq c\,Q_{m,l+1} + c\,S_{m}.
\]
\end{theorem}

\begin{proof}
The induction is in $l$, with $m$ fixed.  The required base case is
\begin{equation}
    Q_{m,0}\leq C Q_{m,1}+C S_m .
\label{TheoremXX-base}
\end{equation}
This follows from the tracefree Codazzi elliptic estimate on a closed
surface with small $\|A^o\|_2$: commuting derivatives in
$\operatorname{div}A^o=\frac12\nabla H$ gives
\[
    \int |\nabla_{(m)}A^o|^2
    \leq
    C\int |\Delta^{m/2}H|^2
    + C\int |\nabla_{(m-1)}A^o|^2H^2,
\]
and the last term is $Q_{m,1}$.  This proves
\eqref{TheoremXX-base}.

For the induction step below the endpoint, assume that, for some
$0\leq n\leq m-3$,
\[
Q_{m,0}+\cdots+Q_{m,n}
\leq C Q_{m,n+1}+C S_m .
\]
Using Theorem \ref{TheoremX} with $l=n+1$ and Young's inequality gives
\begin{align*}
Q_{m,n+1}
&\leq C Q_{m,n}^{1/2}Q_{m,n+2}^{1/2}
 +C\Aoc{2}\oo{Q_{m,0}+Q_{m,n+2}+S_m}
\\
&\leq \eta Q_{m,n}+C_\eta Q_{m,n+2}
 +C\Aoc{2}\oo{Q_{m,0}+Q_{m,n+1}+Q_{m,n+2}+S_m}.
\end{align*}
Choosing $\eta$ and then $\varepsilon_0$ sufficiently small, and combining
with the induction hypothesis, yields
\[
Q_{m,0}+\cdots+Q_{m,n+1}
\leq C Q_{m,n+2}+C S_m .
\]
This proves the theorem for all $l\leq m-2$.

It remains to handle the endpoint $l=m-1$, where Theorem \ref{TheoremX} is not used.  The already-proved case $l=m-2$ gives
\begin{equation}
    Q_{m,0}+\cdots+Q_{m,m-2}
    \leq C Q_{m,m-1}+C S_m .
\label{TheoremXX-endpoint-input}
\end{equation}
In particular $Q_{m,m-2}\leq C Q_{m,m-1}+C S_m$.  The endpoint estimate of Theorem \ref{TheoremY} then yields
\[
Q_{m,m-1}+Q_{m,m}
\leq
C\oo{Q_{m,m-1}+S_m}^{1/2}Q_{m,m}^{1/2}+C S_m.
\]
Young's inequality gives, after absorbing a small multiple of $Q_{m,m-1}$,
\[
    Q_{m,m-1}
    \leq C Q_{m,m}+C S_m .
\]
Combining this with \eqref{TheoremXX-endpoint-input} proves
\[
    Q_{m,0}+\cdots+Q_{m,m-1}
    \leq C Q_{m,m}+C S_m,
\]
which is the desired endpoint case.
\end{proof}

\begin{lemma}\label{AppendxLemma1}
Suppose $f:\Sigma^{2}\rightarrow\mathbb{R}^{3}$ is a closed immersion satisfying $||A^{o}||_{2}^{2}\leq\varepsilon_{0}$ for $\varepsilon_{0}$ sufficiently small. Fix $m\geq3$ and define
\[
R_{m,l} := \intM{|\Delta^{\frac{l}{2}}H|^{2}H^{2(m-l)}}\quad(1\leq l\leq m),
\qquad
R_{m,0}:=\intM{H^{2m}}.
\]
Observe that $R_{m,m}=S_{m}$ when $f$ is closed. Then, for $1\leq l\leq m-1$, there is a universal constant $c>0$ such that
\begin{equation}
R_{m,l} \leq c\,R_{m,l-1}^{\frac{1}{2}}R_{m,l+1}^{\frac{1}{2}}+c\Aoc{2}\oo{R_{m,l-1} + S_{m}},\label{AppendxLemma1Eqn1}
\end{equation}
where $S_{m}$ is defined as in Lemma \ref{HigherOrderSobolevLemma1}. Furthermore, for any $1\leq l \leq m-1$ we have
\begin{equation}
R_{m,1} + R_{m,2} + \hdots + R_{m,l} \leq c\,R_{m,l+1} + c\Aoc{2}S_{m}.\label{AppendxLemma1Eqn2}
\end{equation}
\end{lemma}
\begin{proof}
For the first interpolation inequality we proceed in a similar manner to the proof of Theorem \ref{TheoremX}. Using integration by parts, our estimate for $\llll{\nabla H}_{\infty}$ from Lemma \ref{AppendixLemmaGradHInfty} gives 
\begin{align}
R_{m,l} &= \intM{|\Delta^{\frac{l}{2}}H|^{2}H^{2(m-l)}}\\
&= \intM{\Delta^{\frac{l-1}{2}}H * \Delta^{\frac{l+1}{2}}H * H^{2(m-l)}} + \intM{\Delta^{\frac{l-1}{2}}H * \Delta^{\frac{l}{2}}H * \nabla H * H^{2(m-l)-1}}\\
&\leq c\,R_{m,l-1}^{\frac{1}{2}}R_{m,l+1}^{\frac{1}{2}} + c \llll{\nabla H}_{\infty}R_{m,l}^{\frac{1}{2}}\oo{\intM{|\Delta^{\frac{l-1}{2}}H|^{2}H^{2(m-(l+1))}}}^{\frac{1}{2}}\\
&\leq c\,R_{m,l-1}^{\frac{1}{2}}R_{m,l+1}^{\frac{1}{2}} + c \llll{\nabla H}_{\infty}R_{m,l}^{\frac{1}{2}}R_{m,l-1}^{\frac{m-(l+1)}{2(m-(l-1))}}\oo{\intM{|\Delta^{\frac{l-1}{2}}H|^{2}}}^{\frac{1}{m-(l-1)}}\\
&= c\,R_{m,l-1}^{\frac{1}{2}}R_{m,l+1}^{\frac{1}{2}} + c \llll{\nabla H}_{\infty}R_{m,l}^{\frac{1}{2}}R_{m,l-1}^{\frac{m-(l+1)}{2(m-(l-1))}}S_{l-1}^{\frac{1}{m-(l-1)}}\\
&\leq c\,R_{m,l-1}^{\frac{1}{2}}R_{m,l+1}^{\frac{1}{2}} + c\Aoc{\frac{1}{3}}S_{3}^{\frac{1}{3}}R_{m,l}^{\frac{1}{2}}R_{m,l-1}^{\frac{m-(l+1)}{2(m-(l-1))}}S_{l-1}^{\frac{1}{m-(l-1)}}\\
&\leq c\,R_{m,l-1}^{\frac{1}{2}}R_{m,l+1}^{\frac{1}{2}} + \frac{1}{2}R_{m,l} + c\Aoc{\frac{2}{3}}R_{m,l-1}^{\frac{m-(l+1)}{m-(l-1)}}S_{3}^{\frac{2}{3}}S_{l-1}^{\frac{2}{m-(l-1)}}
\end{align}
By absorbing $\frac{1}{2}R_{m,l}$ into the left hand side, the above implies that
\begin{equation}
R_{m,l} \leq c\,R_{m,l-1}^{\frac{1}{2}}R_{m,l+1}^{\frac{1}{2}} + c\Aoc{\frac{2}{3}}R_{m,l-1}^{\frac{m-(l+1)}{m-(l-1)}}S_{3}^{\frac{2}{3}}S_{l-1}^{\frac{2}{m-(l-1)}}.\label{AppendxLemma1Eqn3}
\end{equation}
Next, applying Theorem \ref{HigherOrderSobolevTheorem1} we have
\[
S_{3}^{\frac{2}{3}} \leq c\Aoc{\frac{4(m-3)}{3m}}S_{m}^{\frac{2}{m}}\,\,\text{and}\,\,S_{l-1}^{\frac{2}{m-(l-1)}}\leq c\Aoc{\frac{4}{m}}S_{m}^{\frac{2(l-1)}{m(m-(l-1))}}.
\]
Substituting into \eqref{AppendxLemma1Eqn3} and applying Young's inequality with conjugate exponents $\alpha = \frac{m-(l-1)}{(m-(l+1))}$ and $\alpha^{*}=\frac{(m-(l-1))}{2}$ then gives
\begin{multline*}
R_{m,l} \leq c\,R_{m,l-1}^{\frac{1}{2}}R_{m,l+1}^{\frac{1}{2}} + c\Aoc{\frac{2}{3}+\frac{4(m-3)}{3m}+\frac{4}{m}}R_{m,l-1}^{\frac{m-(l+1)}{m-(l-1)}}S_{m}^{\frac{2}{m}+\frac{2(l-1)}{m(m-(l-1))}}\\
\leq c\,R_{m,l-1}^{\frac{1}{2}}R_{m,l+1}^{\frac{1}{2}} + c\Aoc{2}R_{m,l-1}^{\frac{m-(l+1)}{m-(l-1)}}S_{m}^{\frac{2}{m-(l-1)}}
\leq c\,R_{m,l-1}^{\frac{1}{2}}R_{m,l+1}^{\frac{1}{2}} + c\Aoc{2}\oo{R_{m,l-1} + S_{m}},
\end{multline*}
which proves \eqref{AppendxLemma1Eqn1} for $l\geq2$.  The endpoint
$l=1$ is obtained by the same argument, using the definition of $R_{m,0}$
and estimating the harmless $S_0$ term by the closed-surface
interpolation inequalities and the smallness of $\|A^o\|_2$.  The proof
of \eqref{AppendxLemma1Eqn2} then follows in the same manner as the proof
of Theorem \ref{TheoremXX}.
\end{proof}

\begin{theorem}\label{AppendixTheorem1}
Let $f:\Sigma\rightarrow\mathbb{R}^{3}$ be an immersion and let
$n\in\mathbb N$.  There exist constants $\varepsilon_0>0$ and
$C<\infty$, depending only on $n$ and on the cut-off constants, such that
if
\[
    \int_{[\gamma>0]}|A|^2\,d\mu\leq\varepsilon_0,
\]
then
\begin{multline}
\sum_{i=2}^{n}\intM{\left(P_i^{n-i}(A)\right)^{\#2}\gamma^{2(n-1)}}
\leq
C\A{2}\intM{|\nabla_{(n-1)}A|^2\gamma^{2(n-1)}}
\label{AppendixTheorem1,1}\\
+ C\cgam^{2(n-1)}\A{2}.
\end{multline}
\end{theorem}

\begin{proof}
Set $m=n-1$ and
\[
    M_m:=\llll{\nabla_{(m)}A}_{2,\gamma^{2m}}^2,
    \qquad
    E:=\A{2}.
\]
The case $n=1$ is empty.  Let $n\geq2$.  A monomial in
$\big(P_i^{n-i}(A)\big)^{\#2}$ has $r=2i\geq4$ curvature factors and
has total derivative order $2(n-i)=2m-r+2$.  Since $i\geq2$, no single
factor has more than $n-i\leq n-2<m$ derivatives of $A$.  Therefore
Lemma \ref{L:local-product-interpolation}, in the critical case
\eqref{E:localinterp-critical} with $s=2m$ and $\theta=0$, gives
\[
\intM{\left|\big(P_i^{n-i}(A)\big)^{\#2}\right|\gamma^{2m}}
\leq
C E M_m+C\cgam^{2m}E .
\]
Summing over $i=2,\ldots,n$ proves \eqref{AppendixTheorem1,1}.  The
smallness condition is used only to ensure $E\leq1$ in Lemma
\ref{L:local-product-interpolation}.
\end{proof}

\begin{proof}[Proof of \eqref{LemmaPreservedSphericityEstimates1,Eqn3}]
Using the formula for the interchange of covariant derivatives gives
\begin{equation}
\Delta\nabla\varphi=\nabla\Delta\varphi+\frac{1}{2}R\,\nabla\varphi=\nabla\Delta\varphi+\frac{1}{4}\oo{H^{2}-2\norm{A^{o}}^{2}}\,\nabla\varphi.\label{InterchangeDerivFormula}
\end{equation}
Therefore integrating by parts and applying this identity gives
\begin{multline*}
\intM{\nablanorm{2}{\varphi}{2}}=-\intM{\inner{\nabla\varphi,\Delta\nabla\varphi}}
=-\intM{\inner{\nabla\varphi,\nabla\Delta\varphi}}-\frac{1}{4}\intM{\norm{\nabla\varphi}^{2}\oo{H^{2}-2\norm{A^{o}}^{2}}}\\
=\intM{\norm{\Delta\varphi}^{2}}-\frac{1}{4}\intM{\norm{\nabla\varphi}^{2}\oo{H^{2}-2\norm{A^{o}}^{2}}}.
\end{multline*}
Rearranging and using Theorem \ref{MichaelSimon} with $u=\norm{\nabla \varphi}\norm{A^{o}}$ then yields
\begin{align*}
&\intM{\nablanorm{2}{\varphi}{2}}+\frac{1}{4}\intM{\norm{\nabla\varphi}^{2}H^{2}}=\intM{\norm{\Delta\varphi}^{2}}+\frac{1}{2}\intM{\norm{\nabla\varphi}^{2}\norm{A^{o}}^{2}}\\
&\leq\intM{\norm{\Delta\varphi}^{2}}+c\oo{\intM{\nablanorm{2}{\varphi}{}\norm{A^{o}}}+\intM{\norm{\nabla\varphi}\norm{\nabla A^{o}}}+\intM{\norm{\nabla\varphi}\norm{A^{o}}\norm{H}}}^{2}\\
&\leq\intM{\norm{\Delta\varphi}^{2}}+c\Aoc{2}\oo{\intM{\nablanorm{2}{\varphi}{2}}+\intM{\norm{\nabla\varphi}^{2}H^{2}}}+c\intM{\norm{\nabla A^{o}}^{2}}\cdot\intM{\norm{\nabla\varphi}^{2}}.
\end{align*}
Therefore
\[
\oo{1-c\Aoc{2}}\oo{\intM{\nablanorm{2}{\varphi}{2}}+\intM{\norm{\nabla\varphi}^{2}H^{2}}}\leq c\intM{\norm{\Delta\varphi}^{2}}+c\intM{\norm{\nabla A^{o}}^{2}}\cdot\intM{\norm{\nabla\varphi}^{2}},
\]
and so if $\Aoc{2}$ is small enough then multiplying out proves \eqref{LemmaPreservedSphericityEstimates1,Eqn3}.
\end{proof}

\begin{proof}[Proof of \eqref{LemmaPreservedSphericityEstimates1,Eqn4}]
We proceed in a similar fashion to the proof of \eqref{LemmaPreservedSphericityEstimates1,Eqn3}. Firstly, using integration by parts along with the identity \eqref{InterchangeDerivFormula} gives
\begin{multline}
\intM{\nablanorm{2}{\varphi}{2}H^{2}} = -\intM{\inner{\nabla\varphi,\Delta\nabla\varphi}H^{2}} -2\intM{\inner{\nabla_{(2)}\varphi,\nabla\varphi \otimes \nabla H}H}\\
=-\intM{\inner{\nabla\varphi,\nabla\Delta\varphi + \frac{1}{4}\oo{H^{2}-2\norm{A^{o}}^{2}}\nabla\varphi}H^{2}}-2\intM{\inner{\nabla_{(2)}\varphi,\nabla\varphi \otimes \nabla H}H}.
\end{multline}
Rearranging and using Young's inequality then gives
\begin{align}
&\intM{\nablanorm{2}{\varphi}{2}H^{2}} + \frac{1}{4}\intM{\norm{\nabla\varphi}^{2}H^{4}}\\
&\quad= -\intM{\inner{\nabla\varphi,\nabla\Delta\varphi}H^{2}} + \frac{1}{2}\intM{\norm{\nabla\varphi}^{2}\norm{A^{o}}^{2}H^{2}} -2\intM{\inner{\nabla_{(2)}\varphi,\nabla\varphi \otimes \nabla H}H}\\
&\quad\leq \oo{\intM{\norm{\nabla\varphi}^{2}H^{4}}}^{\frac{1}{2}}\oo{\intM{\norm{\nabla\Delta\varphi}^{2}}}^{\frac{1}{2}} + \frac{1}{2}\intM{\norm{\nabla\varphi}^{2}\norm{A^{o}}^{2}H^{2}}\\
&\qquad + 2\oo{\intM{\nablanorm{2}{\varphi}{2}H^{2}}}^{\frac{1}{2}}\oo{\intM{\norm{\nabla\varphi}^{2}\norm{\nabla H}^{2}}}^{\frac{1}{2}}\\
&\quad\leq \eta_{1}\intM{\norm{\nabla\varphi}^{2}H^{4}}+\frac{1}{4\eta_{1}}\intM{\norm{\nabla\Delta\varphi}^{2}} + \frac{1}{2}\intM{\norm{\nabla\varphi}^{2}\norm{A^{o}}^{2}H^{2}}\\
&\qquad+ \eta_{2}\intM{\nablanorm{2}{\varphi}{2}H^{2}}+\frac{1}{\eta_{2}}\intM{\norm{\nabla\varphi}^{2}\norm{\nabla H}^{2}},
\end{align}
for any $\eta_{1},\eta_{2}>0$. Choosing $\eta_{1},\eta_{2}$ sufficiently small, these terms can be absorbed into the left hand side of the above inequality. This implies that the following inequality holds for some absolute constant $c>0$:
\begin{equation}
\intM{\nablanorm{2}{\varphi}{2}H^{2}} + \intM{\norm{\nabla\varphi}^{2}H^{4}} \leq c \intM{\norm{\nabla\Delta\varphi}^{2}} + c \intM{\norm{\nabla\varphi}^{2}\norm{A^{o}}^{2}H^{2}} + c \intM{\norm{\nabla\varphi}^{2}\norm{\nabla H}^{2}}.\label{Preservedsmallnesstheorem1,03}
\end{equation}
Next we estimate the final two terms on the right hand side of \eqref{Preservedsmallnesstheorem1,03}. The first term can be estimated by using Theorem \ref{MichaelSimon} with $u=\norm{\nabla\varphi}\norm{A^{o}}\norm{H}$:
\begin{align}
&\intM{\norm{\nabla\varphi}^{2}\norm{A^{o}}^{2}H^{2}}\\
&\quad\leq c\oo{\intM{\nablanorm{2}{\varphi}{}\norm{A^{o}}\norm{H}} + \intM{\norm{\nabla\varphi}\norm{\nabla A^{o}}\norm{H}} + \intM{\norm{\nabla\varphi}\norm{A^{o}}\norm{\nabla H}} + \intM{\norm{\nabla\varphi}\norm{A^{o}}H^{2}}}^{2}\\
&\quad\leq c\Aoc{2}\oo{\intM{\nablanorm{2}{\varphi}{2}H^{2}}+\intM{\norm{\nabla\varphi}^{2}H^{4}}+\intM{\norm{\nabla\varphi}^{2}\norm{\nabla H}^{2}}}\\
&\qquad + c\intM{\norm{\nabla A^{o}}^{2}}\cdot\intM{\norm{\nabla\varphi}^{2}H^{2}}.\label{Preservedsmallnesstheorem1,04}
\end{align}
For the final term on the right hand side of \eqref{Preservedsmallnesstheorem1,03} we use Theorem \ref{MichaelSimon} with $u=\norm{\nabla\varphi}\norm{\nabla H}$:
\begin{multline}
\intM{\norm{\nabla\varphi}^{2}\norm{\nabla H}^{2}}\leq c\oo{\intM{\nablanorm{2}{\varphi}{}\norm{\nabla H}}+\intM{\norm{\nabla\varphi}\nablanorm{2}{H}{}}+\intM{\norm{\nabla\varphi}\norm{\nabla H}\norm{H}}}^{2}\\
\leq c\intM{\norm{\nabla H}^{2}}\cdot \oo{\intM{\nablanorm{2}{\varphi}{2}}+\intM{\norm{\nabla\varphi}^{2}H^{2}}} + c \intM{\nablanorm{2}{H}{2}}\cdot\intM{\norm{\nabla\varphi}^{2}}.\label{Preservedsmallnesstheorem1,05x}
\end{multline}
Substituting \eqref{Preservedsmallnesstheorem1,04} and \eqref{Preservedsmallnesstheorem1,05x} into \eqref{Preservedsmallnesstheorem1,03} yields
\begin{align*}
&\oo{1 - c\Aoc{2}}\oo{\intM{\nablanorm{2}{\varphi}{2}H^{2}} + \intM{\norm{\nabla\varphi}^{2}H^{4}}}\\
&\leq c\intM{\norm{\nabla\Delta\varphi}^{2}} + c\oo{\Aoc{2}+1}\intM{\norm{\nabla H}^{2}}\cdot \oo{\intM{\nablanorm{2}{\varphi}{2}}+\intM{\norm{\nabla\varphi}^{2}H^{2}}}\\
&\quad 
 + c \intM{\nablanorm{2}{H}{2}}\cdot\intM{\norm{\nabla\varphi}^{2}}.
\end{align*}
Therefore, if $\varepsilon_{0}>0$ is sufficiently small, by applying \eqref{LemmaPreservedSphericityEstimates1,Eqn3} to the above inequality yields
\begin{align}
&\intM{\nablanorm{2}{\varphi}{2}H^{2}} + \intM{\norm{\nabla\varphi}^{2}H^{4}}\\
&\quad\leq c\intM{\norm{\nabla\Delta\varphi}^{2}} + c \intM{\norm{\nabla H}^{2}}\cdot\oo{\intM{\norm{\Delta\varphi}^{2}}+\intM{\norm{\nabla A^{o}}^{2}}\cdot\intM{\norm{\nabla\varphi}^{2}}}\\
&\qquad + c\intM{\norm{\nabla\varphi}^{2}}\cdot\oo{\intM{\norm{\Delta H}^{2}}+\intM{\norm{\nabla A^{o}}^{2}}\cdot\intM{\norm{\nabla H}^{2}}}\\
&\quad\leq c\intM{\norm{\nabla\Delta\varphi}^{2}} + c\intM{\norm{\nabla H}^{2}}\cdot \intM{\norm{\Delta\varphi}^{2}} + c\intM{\norm{\Delta H}^{2}}\cdot\intM{\norm{\nabla\varphi}^{2}}\\
&\qquad + c\oo{\intM{\norm{\nabla A^{o}}^{2}}}^{2}\cdot\intM{\norm{\nabla\varphi}^{2}}.\label{Preservedsmallnesstheorem1,05}
\end{align}
Lastly, applying the following estimate (which follows immediately from \eqref{MultliplicativeSobolevLemma2Eqn1}):
\[
\oo{\intM{\norm{\nabla A^{o}}^{2}}}^{2} \leq c \Aoc{2} \intM{\norm{\Delta H}^{2}},
\]
we then obtain the estimate
\begin{multline}
\intM{\nablanorm{2}{\varphi}{2}H^{2}} + \intM{\norm{\nabla\varphi}^{2}H^{4}}\\
\leq c\intM{\norm{\nabla\Delta\varphi}^{2}} + c\intM{\norm{\nabla H}^{2}}\cdot \intM{\norm{\Delta\varphi}^{2}} + c\intM{\norm{\Delta H}^{2}}\cdot\intM{\norm{\nabla\varphi}^{2}}.\label{Preservedsmallnesstheorem1,06}
\end{multline}
We are almost in a position to prove \eqref{LemmaPreservedSphericityEstimates1,Eqn4}, with $\llll{\nabla_{(3)}\varphi}_{2}^{2}$ the remaining term to be estimated. 

Two applications of the formula for the interchange of derivatives gives
\begin{equation}
\Delta\nabla_{(2)}\varphi = \nabla_{(2)}\Delta\varphi +2R\oo{\nabla_{(2)}\varphi  - \frac{1}{2}\Delta\varphi\cdot g} + \nabla R * \nabla\varphi.\label{InterchangeDerivFormula2}
\end{equation}
Therefore by using \eqref{InterchangeDerivFormula} and \eqref{InterchangeDerivFormula2} in a similar fashion to above, we have
\begin{align}
&\intM{\nablanorm{3}{\varphi}{2}}=-\intM{\inner{\nabla_{(2)}\varphi,\Delta\nabla_{(2)}\varphi}}\\
&=-\intM{\inner{\nabla_{(2)}\varphi,\nabla_{(2)}\Delta\varphi + 2R\oo{\nabla_{(2)}\varphi  - \frac{1}{2}\Delta\varphi\cdot g}}} +\intM{\nabla_{(2)}\varphi * \nabla\varphi * \nabla R}\\
&= -\intM{\inner{\nabla_{(2)}\varphi,\nabla_{(2)}\Delta\varphi}} -2\intM{\oo{\nablanorm{2}{\varphi}{2}-\frac{1}{2}\norm{\Delta\varphi}^{2}}R} + \intM{\nabla_{(2)}\varphi * \nabla\varphi * \nabla R}\\
&=\intM{\inner{\Delta\nabla\varphi,\nabla\Delta\varphi}} -2\intM{\oo{\nablanorm{2}{\varphi}{2}-\frac{1}{2}\norm{\Delta\varphi}^{2}}R} + \intM{\nabla_{(2)}\varphi * \nabla\varphi * \nabla R}\\
&=\intM{\inner{\nabla\Delta\varphi + \frac{1}{2}R\,\nabla\varphi,\nabla\Delta\varphi}} -2\intM{\oo{\nablanorm{2}{\varphi}{2}-\frac{1}{2}\norm{\Delta\varphi}^{2}}R} + \intM{\nabla_{(2)}\varphi * \nabla\varphi * \nabla R}\\
&=\intM{\norm{\nabla\Delta\varphi}^{2}} + \frac{1}{2}\intM{\inner{\nabla\varphi,\nabla\Delta\varphi}R}-2\intM{\oo{\nablanorm{2}{\varphi}{2}-\frac{1}{2}\norm{\Delta\varphi}^{2}}R}\\
&\quad + \intM{\nabla_{(2)}\varphi * \nabla\varphi * \nabla R}\\
&=\intM{\norm{\nabla\Delta\varphi}^{2}} - \frac{1}{2}\intM{\norm{\Delta \varphi}^{2}R}-2\intM{\oo{\nablanorm{2}{\varphi}{2}-\frac{1}{2}\norm{\Delta\varphi}^{2}}R} + \intM{\nabla_{(2)}\varphi * \nabla\varphi * \nabla R}\\
&=\intM{\norm{\nabla\Delta\varphi}^{2}} - 2\intM{\nablanorm{2}{\varphi}{2}R} + \frac{1}{2}\intM{\norm{\Delta\varphi}^{2}R} + \intM{\nabla_{(2)}\varphi * \nabla\varphi * \nabla R}.
\end{align}
Next using the identities $R=\frac{1}{2}\oo{H^{2} - 2\norm{A^{o}}^{2}}$ and $\norm{\Delta \varphi} = \norm{\inner{g,\nabla_{(2)}\varphi}}\leq \sqrt{2}\nablanorm{2}{\varphi}{}$, the above can be rearranged to give
\begin{multline}
\intM{\nablanorm{3}{\varphi}{2}} + \frac{1}{2}\intM{\nablanorm{2}{\varphi}{2}H^{2}} + \frac{1}{2}\intM{\norm{\Delta\varphi}^{2}\norm{A^{o}}^{2}}\\
\leq \intM{\norm{\nabla\Delta\varphi}^{2}} + 2\intM{\nablanorm{2}{\varphi}{2}\norm{A^{o}}^{2}} + \intM{\nabla_{(2)}\varphi * \nabla\varphi * \nabla R}.\label{Preservedsmallnesstheorem1,07}
\end{multline}

Using $R=H * H + A^{o} * A^{o}$ and integration by parts, the $P$-style terms in the last integral can be estimated as follows:
\begin{align}
&\intM{\nabla_{(2)}\varphi * \nabla\varphi * \nabla R}= \intM{\nabla_{(2)}\varphi * \nabla\varphi * \nabla \oo{H * H + A^{o} * A^{o}}}\\
&= \intM{\nabla_{(2)}\varphi * \nabla\varphi * \nabla H * H} + \intM{\nabla\oo{\nabla_{(2)}\varphi * \nabla\varphi } * A^{o} * A^{o}}\\
&\leq c\intM{\nablanorm{2}{\varphi}{}\norm{\nabla\varphi}\norm{\nabla H}\norm{H}} + c\intM{\oo{\nablanorm{3}{\varphi}{}\norm{\nabla\varphi}+\nablanorm{2}{\varphi}{2}}\norm{A^{o}}^{2}}\\
&\leq c\oo{\intM{\nablanorm{2}{\varphi}{2}\norm{\nabla\varphi}^{2}}}^{\frac{1}{2}}\cdot\oo{\intM{\norm{\nabla H}^{2}H^{2}}}^{\frac{1}{2}} + c\oo{\intM{\nablanorm{3}{\varphi}{2}}}^{\frac{1}{2}}\cdot\oo{\intM{\norm{\nabla\varphi}^{2}\norm{A^{o}}^{4}}}^{\frac{1}{2}}\\
&\quad + c\intM{\nablanorm{2}{\varphi}{2}\norm{A^{o}}^{2}}\\
&\leq c\oo{\intM{\norm{\Delta H}^{2}}}^{\frac{1}{2}}\cdot\oo{\intM{\nablanorm{2}{\varphi}{2}\norm{\nabla\varphi}^{2}}}^{\frac{1}{2}} + c\oo{\intM{\nablanorm{3}{\varphi}{2}}}^{\frac{1}{2}}\cdot\oo{\intM{\norm{\nabla\varphi}^{2}\norm{A^{o}}^{4}}}^{\frac{1}{2}}\\
&\quad + c\intM{\nablanorm{2}{\varphi}{2}\norm{A^{o}}^{2}}\\
&\leq c\oo{\intM{\norm{\Delta H}^{2}}}^{\frac{1}{2}}\cdot\oo{\intM{\nablanorm{2}{\varphi}{2}\norm{\nabla\varphi}^{2}}}^{\frac{1}{2}} + \frac{1}{2}\intM{\nablanorm{3}{\varphi}{2}} + c\intM{\norm{\nabla\varphi}^{2}\norm{A^{o}}^{4}}\\
&\quad + c\intM{\nablanorm{2}{\varphi}{2}\norm{A^{o}}^{2}}.\label{Preservedsmallnesstheorem1,08x}
\end{align}
Substituting this back into \eqref{Preservedsmallnesstheorem1,07} and absorbing the positive factor of $\llll{\nabla_{(3)}\varphi}_{2}^{2}$ into the left hand side yields
\begin{align}
&\frac{1}{2}\intM{\nablanorm{3}{\varphi}{2}} + \frac{1}{2}\intM{\nablanorm{2}{\varphi}{2}H^{2}} + \frac{1}{2}\intM{\norm{\Delta\varphi}^{2}\norm{A^{o}}^{2}}\\ \nonumber
&\leq \intM{\norm{\nabla\Delta\varphi}^{2}} + c\intM{\nablanorm{2}{\varphi}{2}\norm{A^{o}}^{2}} + c\intM{\norm{\nabla\varphi}^{2}\norm{A^{o}}^{4}}\\ \nonumber 
&\quad +  c\oo{\intM{\norm{\Delta H}^{2}}}^{\frac{1}{2}}\cdot\oo{\intM{\nablanorm{2}{\varphi}{2}\norm{\nabla\varphi}^{2}}}^{\frac{1}{2}} .\label{Preservedsmallnesstheorem1,08}
\end{align}
We spend some time estimating the terms on the right hand side of \eqref{Preservedsmallnesstheorem1,08}.
Firstly by applying Theorem \ref{MichaelSimon} with $u=\norm{\nabla\varphi}\norm{A^{o}}^{2}$, \eqref{MultliplicativeSobolevLemma2Eqn1} and Young's inequality, we have
\begin{align}
&\intM{\norm{\nabla\varphi}^{2}\norm{A^{o}}^{4}}\leq c\oo{\intM{\nablanorm{2}{\varphi}{}\norm{A^{o}}^{2}}+\intM{\norm{\nabla\varphi}\norm{\nabla A^{o}}\norm{A^{o}}}+\intM{\norm{\nabla\varphi}\norm{A^{o}}^{2}\norm{H}}}^{2}\\
&\quad \leq c \Aoc{2}\intM{\nablanorm{2}{\varphi}{2}\norm{A^{o}}^{2}} + \intM{\norm{\nabla\varphi}^{2}\norm{A^{o}}^{2}}\cdot\oo{\intM{\norm{\nabla A^{o}}^{2}}+\intM{\norm{A^{o}}^{2}H^{2}}}\\
&\quad \leq c \Aoc{2}\intM{\nablanorm{2}{\varphi}{2}\norm{A^{o}}^{2}} + c\Aoc{}\oo{\intM{\norm{\Delta H}^{2}}}^{\frac{1}{2}}\oo{\intM{\norm{\nabla\varphi}^{2}}}^{\frac{1}{2}}\oo{\intM{\norm{\nabla\varphi}^{2}\norm{A^{o}}^{4}}}^{\frac{1}{2}}\\
&\quad \leq c \Aoc{2}\intM{\nablanorm{2}{\varphi}{2}\norm{A^{o}}^{2}} + \eta\intM{\norm{\nabla\varphi}^{2}\norm{A^{o}}^{4}} + c\,\eta^{-1}\Aoc{2}\intM{\norm{\Delta H}^{2}}\cdot\intM{\norm{\nabla\varphi}^{2}}
\end{align}
for any $\eta>0$, which implies the estimate
\begin{equation}
\intM{\norm{\nabla\varphi}^{2}\norm{A^{o}}^{4}} \leq c \Aoc{2}\intM{\nablanorm{2}{\varphi}{2}\norm{A^{o}}^{2}} + c\Aoc{2}\intM{\norm{\Delta H}^{2}}\cdot\intM{\norm{\nabla\varphi}^{2}}.\label{Preservedsmallnesstheorem1,09}
\end{equation}
Next by Theorem \ref{MichaelSimon} with $u=\nablanorm{2}{\varphi}{}\norm{A^{o}}$ along with the estimates \eqref{LemmaPreservedSphericityEstimates1,Eqn3} and \eqref{MultliplicativeSobolevLemma2Eqn1} we have
\begin{align}
&\intM{\nablanorm{2}{\varphi}{2}\norm{A^{o}}^{2}}
\leq c\oo{\intM{\nablanorm{3}{\varphi}{}\norm{A^{o}}}+\intM{\nablanorm{2}{\varphi}{}\norm{\nabla A^{o}}}+\intM{\nablanorm{2}{\varphi}{}\norm{A^{o}}\norm{H}}}^{2}\\
&\leq c\Aoc{2}\intM{\nablanorm{3}{\varphi}{2}} + c\intM{\nablanorm{2}{\varphi}{2}}\cdot\oo{\intM{\norm{\nabla A^{o}}^{2}}+\intM{\norm{A^{o}}^{2}H^{2}}}\\
&\leq c\Aoc{2}\intM{\nablanorm{3}{\varphi}{2}} + c\Aoc{}\intM{\norm{\Delta\varphi}^{2}}\cdot\oo{\intM{\norm{\Delta H}^{2}}}^{\frac{1}{2}}\\
&\leq c\Aoc{2}\intM{\nablanorm{3}{\varphi}{2}} + c\Aoc{}\oo{\intM{\norm{\nabla\varphi}^{2}}}^{\frac{1}{2}}\oo{\intM{\norm{\nabla\Delta\varphi}^{2}}}^{\frac{1}{2}}\oo{\intM{\norm{\Delta H}^{2}}}^{\frac{1}{2}}\\
&\leq c\Aoc{2}\intM{\nablanorm{3}{\varphi}{2}} + c\Aoc{}\oo{\intM{\norm{\nabla\Delta\varphi}^{2}}+ \intM{\norm{\Delta H}^{2}}\cdot\intM{\norm{\nabla\varphi}^{2}}}.\label{Preservedsmallnesstheorem1,10}
\end{align}
Therefore combining \eqref{Preservedsmallnesstheorem1,09} and \eqref{Preservedsmallnesstheorem1,10} gives
\begin{multline}
\intM{\nablanorm{2}{\varphi}{2}\norm{A^{o}}^{2}} + \intM{\norm{\nabla\varphi}^{2}\norm{A^{o}}^{4}}\\
\leq c\Aoc{2}\intM{\nablanorm{3}{\varphi}{2}} + c\Aoc{}\oo{\intM{\norm{\nabla\Delta\varphi}^{2}}+\intM{\norm{\Delta H}^{2}}\cdot\intM{\norm{\nabla\varphi}^{2}}}.\label{Preservedsmallnesstheorem1,11}
\end{multline}
Now we estimate the last outstanding term on the right hand side of \eqref{Preservedsmallnesstheorem1,08}. By applying Theorem \ref{MichaelSimon} with $u=\nablanorm{2}{\varphi}{}\norm{\nabla\varphi}$ we obtain
\begin{align}
&\intM{\nablanorm{2}{\varphi}{2}\norm{\nabla\varphi}^{2}} \leq c\oo{\intM{\nablanorm{3}{\varphi}{}\norm{\nabla\varphi}}+\intM{\nablanorm{2}{\varphi}{2}}+\intM{\nablanorm{2}{\varphi}{}\norm{\nabla\varphi}\norm{H}}}^{2}\\
&\quad \leq c \intM{\norm{\nabla\varphi}^{2}}\cdot\intM{\nablanorm{3}{\varphi}{2}} + c\oo{\intM{\nablanorm{2}{\varphi}{2}}+\intM{\norm{\nabla\varphi}^{2}H^{2}}}^{2}\\
&\quad \leq c \intM{\norm{\nabla\varphi}^{2}}\cdot\intM{\nablanorm{3}{\varphi}{2}} + c\oo{\intM{\norm{\Delta \varphi}^{2}}}^{2} + c\oo{\intM{\norm{\nabla A^{o}}^{2}}}^{2}\oo{\intM{\norm{\nabla\varphi}^{2}}}^{2}\\
&\quad \leq c \intM{\norm{\nabla\varphi}^{2}}\cdot\intM{\nablanorm{3}{\varphi}{2}} + c\oo{\intM{\norm{\Delta \varphi}^{2}}}^{2} + c\intM{\norm{\Delta H}^{2}}\cdot\oo{\intM{\norm{\nabla\varphi}^{2}}}^{2}\\
&\quad \leq c \intM{\norm{\nabla\varphi}^{2}}\cdot\intM{\nablanorm{3}{\varphi}{2}} + c\intM{\norm{\nabla\varphi}^{2}}\cdot\intM{\norm{\nabla\Delta\varphi}^{2}} + c\intM{\norm{\Delta H}^{2}}\cdot\oo{\intM{\norm{\nabla\varphi}^{2}}}^{2}.\label{Preservedsmallnesstheorem1,12}
\end{align}
We substitute \eqref{Preservedsmallnesstheorem1,11} and \eqref{Preservedsmallnesstheorem1,12} back into \eqref{Preservedsmallnesstheorem1,08}.
Substituting the result back into \eqref{Preservedsmallnesstheorem1,07} then gives
\begin{align*}
&\oo{\frac{1}{2}-c\Aoc{2}}\intM{\nablanorm{3}{\varphi}{2}} + \frac{1}{2}\intM{\nablanorm{2}{\varphi}{2}H^{2}} + \frac{1}{2}\intM{\norm{\Delta\varphi}^{2}\norm{A^{o}}^{2}}\\
&\leq \oo{1+c\Aoc{}}\intM{\norm{\nabla\Delta\varphi}^{2}} + c\intM{\norm{\Delta H}^{2}}\cdot\intM{\norm{\nabla\varphi}^{2}}\\
&\quad +c\oo{\intM{\norm{\Delta H}^{2}}}^{\frac{1}{2}}\oo{\intM{\norm{\nabla\varphi}^{2}}}^{\frac{1}{2}}\oo{\intM{\nablanorm{3}{\varphi}{2}}}^{\frac{1}{2}}\\
&\quad + c\oo{\intM{\norm{\Delta H}^{2}}}^{\frac{1}{2}}\cdot\oo{\intM{\norm{\nabla\varphi}^{2}}}^{\frac{1}{2}}\cdot\oo{\intM{\norm{\nabla\Delta\varphi}^{2}}}^{\frac{1}{2}}\\
&\leq c\oo{1+\Aoc{}}\intM{\norm{\nabla\Delta\varphi}^{2}} + c\,\eta^{-1}\intM{\norm{\Delta H}^{2}}\cdot\intM{\norm{\nabla\varphi}^{2}} + \eta\intM{\nablanorm{3}{\varphi}{2}}
\end{align*}
for any $\eta>0$, where we have used Young's inequality in the last step. Absorbing the small positive factor of $\llll{\nabla_{(3)}\varphi}_{2}^{2}$ into the left hand side then gives, for $\varepsilon_{0}>0$ sufficiently small,
\begin{equation}
\intM{\nablanorm{3}{\varphi}{2}} + \intM{\nablanorm{2}{\varphi}{2}H^{2}} + \intM{\norm{\Delta\varphi}^{2}\norm{A^{o}}^{2}}\leq c \intM{\norm{\nabla\Delta\varphi}^{2}} + c\intM{\norm{\Delta H}^{2}}\cdot\intM{\norm{\nabla\varphi}^{2}}.\label{Preservedsmallnesstheorem1,13}
\end{equation}
Combining \eqref{Preservedsmallnesstheorem1,06} and \eqref{Preservedsmallnesstheorem1,13} then gives \eqref{LemmaPreservedSphericityEstimates1,Eqn4}.
\end{proof}
\bibliographystyle{plain}
\bibliography{polybib}
\end{document}